\documentclass[12pt,letterpaper]{article}

\usepackage[utf8]{inputenc}
\usepackage{amsmath}
\usepackage{amsfonts}
\usepackage{amssymb}
\usepackage{hyperref}
\usepackage[left=2cm,right=2cm,top=2cm,bottom=2cm]{geometry}
\usepackage{graphicx}
\usepackage{tikz}

\usepackage[ backend=biber, sorting= nyt, style=alphabetic, url = false, isbn = false, doi = false, minalphanames=4, maxalphanames=10, maxbibnames=10]{biblatex}
\DeclareMathOperator{\Aut}{Aut} 
\DeclareMathOperator{\res}{Res} 
\DeclareMathOperator{\Hom}{Hom} 
\DeclareMathOperator{\End}{End} 
\DeclareMathOperator{\ad}{ad} 
\DeclareMathOperator{\one}{\mathbf{1}} 
\DeclareMathOperator{\id}{id} 
\DeclareMathOperator{\wt}{wt} 
\DeclareMathOperator{\Op}{\mathcal{O}} 
\DeclareMathOperator{\Span}{Span} 
\DeclareMathOperator{\degree}{deg}
\DeclareMathOperator{\GL}{GL} 
\newcommand{\iu}{\boldsymbol{i}} 
\newcommand{\lvl}{\kappa} 
\newcommand{\el}{\ell} 
\newcommand{\vect}{\boldsymbol} 
\newcommand{\type}[3]{\left( \begin{smallmatrix} #3\\ #1 \,  #2 \end{smallmatrix}\right)} 
\newcommand{\gtype}{$\left(\!\begin{smallmatrix} g\\ 1 \  g \end{smallmatrix}\!\right)$-type} 
\newcommand{\D}{\mathbb{D}} 
\newcommand{\N}{N} 

\usepackage{amsthm}
\newtheorem{theorem}{Theorem}[section]
\newtheorem{proposition}[theorem]{Proposition}
\newtheorem{corollary}[theorem]{Corollary}
\newtheorem{lemma}[theorem]{Lemma}

\theoremstyle{definition}
\newtheorem{definition}[theorem]{Definition}
\newtheorem{remark}[theorem]{Remark}

\usetikzlibrary{decorations.markings,arrows}

\pgfdeclarelayer{edgelayer}
\pgfdeclarelayer{nodelayer}
\pgfsetlayers{edgelayer,nodelayer,main}
\tikzstyle{none}=[]
\tikzstyle{plainedge}=[-, draw=black=]
\tikzstyle{dotted line}=[dashed, draw=black=]

\author{Daniel Tan}
\date{}
\title{\bf Differential equations for intertwining operators among untwisted and twisted modules}

\numberwithin{equation}{section}

\begin{document}

\maketitle

\begin{abstract}
Given any vertex operator algebra $ V $ with an automorphism $ g $, we derive a Jacobi identity for a twisted intertwining operator $ \mathcal{Y} $ of type $ \type{W_1}{W_2}{W_3} $ when $ W_1 $ is an untwisted $ V $-module, and $ W_2 $ and $ W_3 $ are $ g $-twisted $ V $-modules. We say such a twisted intertwining operator is of \gtype. Using the Jacobi identity, we obtain homogeneous linear differential equations satisfied by the multi-series $ \langle w_0, \mathcal{Y}_1(w_1,z_1) \cdots \mathcal{Y}_\N(w_\N,z_\N) w_{\N+1} \rangle $ when $ \mathcal{Y}_j $ are of \gtype\ and the modules are $ C_1 $-cofinite and discretely graded. In the special case that $ V $ is an affine vertex operator algebra, we derive the ``twisted KZ equations" and show that their solutions have regular singularities at certain prescribed points when $ g $ has finite order. When $ V $ is general and $ g $ has finite order, we use the theory of regular singular points to prove that the multi-series $ \langle w_0, \mathcal{Y}_1(w_1,z_1) \cdots \mathcal{Y}_\N(w_\N,z_\N) w_{\N+1} \rangle $ converges absolutely to a multivalued analytic function when $ |z_1| > \cdots > |z_\N| > 0 $ and analytically extends to the region $ z_i, z_i - z_j \neq 0 $. Furthermore, when $ \N = 2 $, we show that these multivalued functions have regular singularities at certain prescribed points.
\end{abstract}

\setcounter{section}{-1}

\section{Introduction}
The representation theory of vertex operator algebras has made major progress in mathematically constructing two-dimensional conformal field theories (CFTs). One interesting feature of vertex operator algebra representation theory is the concept of ``twisting" a representation by an automorphism. The first-discovered elements of this concept within mathematics include the twisted vertex operators in \cite{LepowskyWilson,FLM_E8, LepowskyCalculus}, and the construction of the Moonshine Module in \cite{FLM84} by Frenkel, Lepowsky and Meurman. In \cite{FLM}, a vertex operator algebra structure was constructed on the Moonshine Module. Their construction was interpreted by physicists to describe a string propagating on the quotient of the torus $ \mathbb{R}^{24}/\Lambda $ (given by the Leech lattice $ \Lambda $) by the action of negation. This came to be understood as the first example of a process in physics known as \textit{orbifolding}, where a known CFT is ``quotiented" by a finite group of its symmetries.

Inspired by the Moonshine Module construction, physicists Dixon, Harvey, Vafa and Witten initiated a general study of orbifold CFT in \cite{DHVW1, DHVW2} on a physical level of rigor. Using a finite group  $ G $ of symmetries of the target manifold $ \mathcal{M} $ in a string theory, they proposed a construction that describes the physics of a string propagating in the orbifold $ \mathcal{M}/ G $; hence the term ``orbifold". However, in the current study and elsewhere in vertex operator algebra theory, the group consists of abstract automorphisms of the CFT, and does not need to come from any manifold.

As discussed in \cite{DHVW1, DHVW2}, the orbifolding procedure has two parts. First, \textit{twisted sectors} $ \mathcal{H}_g $ are included in the state space for each $ g \in G $. These sectors are characterized by the property that the fields of the original theory are twisted by the action of $ g $ when they circle an insertion of a state in $ \mathcal{H}_g $. By circling other twisted sectors, each twisted sector $ \mathcal{H}_g $ gains an action of the  centralizer $ C_G(g) $. And second, there is a \textit{projection} of each $ \mathcal{H}_g $ onto the $ C_G(g) $-invariant states. In this paper, we focus on just the first step. More specifically, we study how the untwisted sector $ \mathcal{H}_1 $ acts on a $ g $-twisted sector $ \mathcal{H}_g $.
 
The chiral properties of orbifold CFTs were studied by Dijkgraaf, Vafa, E. Verlinde and H. Verlinde in \cite{DVVV}. They explained, on a physical level of rigor, how orbifold CFTs can be constructed from the twisted representations of the chiral algebra (i.e., twisted modules for a vertex operator algebra) of the original CFT. The fields and their correlation functions are \textit{assumed} to satisfy the usual convergence, radial ordering, and operator product expansion properties.

The chiral study is closest in spirit to the vertex-operator-algebra-theoretic study of orbifold CFT. If we are to rigorously construct orbifold CFTs using the representation theory of vertex operator algebras, we must prove that the convergence, radial ordering, and operator product expansion properties of the correlation functions follow from certain natural conditions for vertex operator algebras and their modules that are relatively easy to verify. 

The mathematical study of chiral orbifold CFT is outlined as follows. Given an automorphism $ g $ of a vertex operator algebra $ V $, a \textit{$ g $-twisted $ V $-module} is a vector space equipped with an action of $ V $, similar to that of a $ V $-module action, but twisted by $ g $ as the vertex operator formally circles the origin, that is $ Y(u,x) = e^{2 \pi \iu x \frac{d}{dx}} Y(gu,x) $. Examples of twisted vertex operators were first discovered in mathematics in \cite{LepowskyWilson} when constructing the affine Lie algebra $ A_1^{(1)} $, and more generally in \cite{FLM_E8, LepowskyCalculus}. Twisted modules for vertex operator algebras were essentially introduced in \cite{FLM}. They were formally axiomatized in \cite{FFR}, and in \cite{DongTwisted} for all automorphisms of finite order, by compiling the basic properties discovered and proved in \cite{FLM} (see \cite{FLM_VOCalc} and \cite{Lepowsky_VOandMonster}), notably the ``twisted Jacobi identity". The notion of twisted module introduced in \cite{Gentwistmod} allows for general automorphisms, but the algebraic Jacobi identity is replaced with the analytic axiom of ``duality". 

Given a group $ G $ of automorphisms of $ V $, the subspace $ V^G $ of elements fixed under the action of $ G $ is a subalgebra of $ V $. One can extend $ V^G $ by first adjoining $ g $-twisted $ V $-modules to $ V $, for each $ g \in G $, and twisted intertwining operators between them before taking the fixed-point subalgebra. 

The Moonshine Module $ V^\natural $ was the first example of a vertex operator algebra obtained by, what was later interpreted as, an orbifold construction. In \cite{FLM}, Frenkel, Lepowsky and Meurman constructed $ V^\natural $ starting from the Leech lattice vertex operator algebra $ V_\Lambda $ and the (2-element) group generated by the involution induced by negation of the Leech lattice. The physical interpretation of the mathematical theory presented in \cite{FLM} as an example of the orbifold procedure was explained in \cite{DGHBeauty}, and explained conceptually with mathematical rigor in \cite{HuangNonmeromorphic}. In \cite{EMS}, a general construction of orbifold CFTs was achieved for simple, rational, $ C_2 $-cofinite, holomorphic vertex operator algebras of CFT-type (i.e., positive energy) with $ G = \langle g \rangle $ finite cyclic by use of \cite{Miyamoto,CM}.

A key part in a general study of the orbifolding procedure is the study of intertwining operators $ \mathcal{Y}( \cdot, x) \cdot \colon W_1 \otimes W_2 \to W_3\{x\}[\log x] $ among twisted $ V $-modules $ W_1 $, $ W_2 $ and $ W_3 $ as introduced in the most general form in \cite{HUANG2018346, DuHuang}. These \textit{twisted intertwining operators} are believed to give an intertwining-operator-algebraic structure to the direct sum of twisted modules, which provides the vertex-operator-algebraic structure when restricted to the fixed-point subspace. Twisted intertwining operators can be thought of as the mathematical counterparts to the chiral factors of the fields (or chiral vertex operators) corresponding to states in the twisted sectors acting on other twisted sectors.

It is not only necessary to define twisted intertwining operators, but one must also prove that the series $ \langle w_0, \mathcal{Y}_1(w_1,x_1) \cdots \mathcal{Y}_\N(w_\N,x_\N)w_{\N+1} \rangle|_{x_1 = z_1, \dots, x_\N = z_\N} $ converges to a multivalued analytic function in the domain $ |z_1| > \cdots > |z_\N| > 0 $ and analytically continues to the entire region $ z_i, z_i - z_j \neq 0 $, $ i \neq j $. We refer to this as the \textit{convergence property for products of $ \N $ twisted intertwining operators}.

Furthermore, one must also prove that $ \langle w_0, \mathcal{Y}_1(w_1,x_1) \mathcal{Y}_2(w_2,x_2)w_{3} \rangle|_{x_1 = z_1, x_2 = z_2} $ analytically continues to $ \langle w_0, \mathcal{Y}_3(\mathcal{Y}_4(w_1,x_0)w_2,x_2)w_3 \rangle|_{x_0 = z_1 - z_2, x_2 = z_2} $ and $ \langle w_0, \mathcal{Y}_5(w_2,x_2)\mathcal{Y}_6(w_1,x_1)w_3 \rangle|_{x_1 = z_1, x_2 = z_2} $ in the domains $ | z_2 | > |z_1 - z_2| > 0 $ and $ |z_2| > |z_1| > 0 $, respectively, for some twisted intertwining operators $ \mathcal{Y}_3, \mathcal{Y}_4, \mathcal{Y}_5 , \mathcal{Y}_6 $. This condition is known as \textit{associativity} and \textit{commutativity} for twisted intertwining operators. Physically speaking, it says that the chiral factors of correlation functions respect radial ordering, and fields have operator product expansions.

One of the main requirements to prove associativity and commutativity is to prove that the analytic continuation of the product of two twisted intertwining operators has ``regular singularities" at certain points. This also plays an important role in constructing a $ G $-crossed tensor category structure on the category of the twisted $ V $-modules, as explained in \cite{Reptheoryandorbifoldcft}. Similarly to the construction of the braided tensor category structure on the category of $ V $-modules in the eight-part series \cite{HLZ1,HLZ2,HLZ3,HLZ4,HLZ5,HLZ6,HLZ7,HLZ8}, the convergence property together with regular singularities at certain points is used to construct the associativity natural isomorphism and verify its properties.

Little is known about the convergence property for general $ V $ and $ G $, with the present understanding explained in \cite{Reptheoryandorbifoldcft}. If $ G $ is trivial, the most general results are proved by Huang (originally in \cite{HuangDEs} and extended in \cite{HLZ7}) by deriving certain differential equations when the modules are $ C_1 $-cofinite and discretely graded. This includes the special case that $ V $ is $ C_2 $-cofinite. In \cite{Miyamoto}, Miyamoto proved that $ V^G $ is $ C_2 $-cofinite when $ V $ is $ C_2 $-cofinite, CFT-type (i.e., positive energy) and simple, and $ G $ is finite and solvable. Every $ g $-twisted $ V $-module, with $ g \in G $, restricts to a $ V^G $-module. In this context, twisted intertwining operators among $ V $-modules twisted by elements in $ G $ become intertwining operators among $ V^G $-modules. Hence, Huang's theory \cite{HuangDEs} can be applied to the $ C_2 $-cofinite vertex operator algebra $ V^G $ to show the convergence property among $ g $-twisted $ V $-modules, for $ g \in G $. One downside to this approach is that it relies on knowing that $ V^G $ is $ C_2 $-cofinite, which is difficult to prove and is not the most general assumption.

We hope to prove the convergence property directly by generalizing Huang's method \cite{HuangDEs} of differential equations with regular singularities to twisted intertwining operators. In fact, physicists have discovered explicit examples of differential equations in the twisted case in \cite{deBoer:2001nw} and \cite{TwistedBPZ}. The \emph{twisted KZ equations} derived by de Boer, Halpern and Obers in \cite{deBoer:2001nw} are explicit enough to see their regular singularities. We emphasize that their results are for a certain special case. Their derivation relies on the assumption that the chiral vertex operators in the chiral correlation function correspond to states from the untwisted sector in the presence of one twisted sector. We formulate this more precisely by saying a twisted intertwining operator $ \mathcal{Y}( \cdot, x) : W_1 \otimes W_2 \to W_3\{ x\}[\log x] $ is of \emph{\gtype } if $ W_1 $ is an untwisted $ V $-module, and $ W_2 $ and $ W_3 $ are $ g $-twisted $ V $-modules. That is, a twisted intertwining operator of \gtype\ is a twisted intertwining operator among an untwisted and twisted module (the third module is necessarily a twisted module twisted by the same element).
 
By generalizing the method of differential equations used in \cite{HuangDEs}, we derive differential equations satisfied by the ``chiral correlation functions" $ \langle w_0, \mathcal{Y}_1(w_1,x_1) \cdots \mathcal{Y}_\N(w_\N,x_\N)w_{\N+1} \rangle $ given by the product of $ \N $ \gtype\ intertwining operators $ \mathcal{Y}_i $. These differential equations hold when the modules are $ C_1 $-cofinite and discretely graded, and when $ g $ has a finitely generated set of eigenvalues, for which there are many natural examples. Our result does not require any $ C_2 $-cofiniteness assumption on $ V $ nor $ V^G $. In fact, we do not require any knowledge of a larger group $ G $ containing $ g $. 

When the order of $ g $ is finite, we use these differential equations and the theory of regular singular points to prove the convergence property for the product of $ \N $ \gtype\ intertwining operators. We still have homogeneous linear differential equations when $ g $ has infinite order. There are, however, logarithms in the coefficients that prevent us from using the theory of regular singular points at the singularities.

Physically speaking, the correlation functions that we study in this paper describe how the chiral CFT acts on the $ g $-twisted sector. This should not be confused with the description of the chiral algebra (i.e., vertex operator algebra) acting on the $ g $-twisted sector, which is just the notion of a $ g $-twisted $ V $-module. A chiral CFT consists of more than a vertex operator algebra in general; it also contains modules for the vertex operator algebra and intertwining operators among them satisfying certain properties. The consideration of \gtype\ intertwining operators can be avoided if $ V $ is assumed to be holomorphic (i.e., $ V $ has one irreducible module up to isomorphism, namely itself) as is usually done in the currently known mathematical orbifold constructions \cite{FLM,HuangNonmeromorphic,EMS,HohnMoller,GemundenKeller}. The present work does not require $ V $ to be holomorphic, and is therefore more general. 

Though not proved here, we hope to extend Huang's method of differential equations to correlation functions involving twisted intertwining operators that are not of \gtype. The main difficulty in doing so comes from the multivalued $ V $-action on the twisted modules. 

Formal calculus is used to derive differential equations without assuming any convergence. The Jacobi identity is the duality axiom written in the language of formal calculus, and hence is used in our computations. While duality for twisted intertwining operators is easy to write down \cite{HUANG2018346}, the multivaluedness occurring in the twisted modules makes it difficult to convert to a Jacobi identity. Furthermore, once a Jacobi identity is obtained for twisted intertwining operators, it is much more computationally difficult to work with than the Jacobi identity for untwisted intertwining operators. Despite this, we believe that Huang's method of differential equations can be generalized to the case where the modules are twisted by elements of a general finite group. 

In this paper, we first set notation and conventions for a multivalued version of formal calculus. In Section~\ref{sec:Jacobi}, we derive a Jacobi identity for \gtype\ intertwining operators, which must be done for the formal calculus approach of deriving differential equations. We immediately obtain some useful relations satisfied by \gtype\ intertwining operators. In Section~\ref{sec:DEs}, we derive higher-order homogeneous linear differential equations satisfied by the correlation functions of a product of $ \N $ \gtype\ intertwining operators under certain technical assumptions. In Section~\ref{sec:twisted KZ}, we explicitly derive a first-order homogeneous linear system of partial differential equations, known as the twisted KZ equations, when $ V $ is an affine vertex operator algebra, generalizing results of \cite{deBoer:2001nw}. Then, in Section~\ref{sec:kz regularity}, we show that the solutions to the twisted KZ equations have regular singularities when the order of $ g $ is finite. In Section~\ref{sec:regular singularities}, we assume that $ V $ is general and $ g $ has finite order to prove the systems derived in Section~\ref{sec:DEs} can be chosen to have regular singularities at certain prescribed points. We use this to prove the convergence property for products of $ \N $ \gtype\ intertwining operators, and we show that the product of two \gtype\ intertwining operators has regular singularities at certain points.

\paragraph{Acknowledgments} I would like to thank Professor Yi-Zhi Huang and Professor James Lepowsky for their helpful discussions, guidance and suggestions. Thank you to the Rutgers Department of Mathematics for the support.

\section{Notation and conventions}
\label{sec:notation}
Our base field for vector spaces will always be $ \mathbb{C} $, which will be important for the theory. The notion of vertex operator algebra and its morphisms can be found in Definition 3.1.22 of \cite{LepowskyLi}. The framework of logarithmic formal calculus that we use can be found in \cite{HLZ2}. To easily distinguish between formal variables (all formal variables commute) and complex numbers, we will use $ x, \log x $ and $ x_i, \log x_i $ for formal variables and use $ z, z_i, \zeta, \zeta_i, \xi, \xi_i, \eta, \eta_i $ for complex numbers. For a vector space  $ W $, we use $
W\{x_1, \dots, x_\N \} $ to denote the space of formal generalized power series $ \sum_{n_1, \dots, n_\N} w_{n_1, \dots, n_\N} x_1^{n_1} \cdots x_\N^{n_\N} $ with coefficients in $ W $ and powers in $ \mathbb{C} $. We will regularly consider formal series of the form
\begin{align*}
f(x_1, \dots, x_\N, \log x_1, \dots , \log x_\N ) \in \mathbb{C}\{x_1, \dots ,x_\N\} [ \log x_1, \dots , \log x_\N ],
\end{align*}
but we suppress the log variables and write $ f(x_1, \dots, x_\N) $ for brevity. We use $ \frac{\partial}{\partial x_i} $ to denote the derivation defined by
\begin{align*}
\frac{\partial}{\partial x_i} x_j^n = n x_j^{n-1} \delta_{i,j} \quad \text{and} \quad \frac{\partial}{\partial x_i} \log x_j = x_j^{-1} \delta_{i,j}.
\end{align*}
We sometimes write $ \frac{d}{dx} $ when there is only one pair of variables $ x $ and $ \log x $ involved. For any operator $ \mathcal{A} $ on a vector space $ W $, locally nilpotent operator $ \mathcal{N} $ on $ W $, and formal variables $ x, x_1, \log x $, we have the following meanings for formal logarithms and exponentials, using logarithm and exponential notation:
\begin{equation}\label{eq:formal meanings}
\begin{gathered}
\log(1 + x) = \sum_{k = 1}^\infty \frac{(-1)^{k+1}}{k} x^k , \qquad \log(x + x_1) = \log x + \log(1 + x_1/x) , \\
e^{\mathcal{A}x} = \sum_{k = 0}^\infty \frac{1}{k!} \mathcal{A}^k x^k, \qquad x^\mathcal{N} = e^{\mathcal{N} \log x}, \qquad (1 + x)^\mathcal{A} = \sum_{k = 0}^\infty \binom{\mathcal{A}}{k} x^k.
\end{gathered}
\end{equation}
These formulas remain valid when we can make a substitution of a formal series (in possibly multiple variables) that is still well defined, for example
\begin{align*}
(x + x_1)^\mathcal{N} = e^{\mathcal{N}\log(x + x_1) } = e^{\mathcal{N}(\log x + \log(1+x_1/x))} = x^\mathcal{N} e^{\mathcal{N}\log(1+x_1/x)} = x^\mathcal{N} (1+x_1/x)^\mathcal{N}.
\end{align*}

When $ x^\mathcal{S} $ is used for a semisimple operator $ \mathcal{S} $ on $ W $, it is understood to mean $ x^\mathcal{S} u = x^\alpha u $ when $ \mathcal{S} u = \alpha u $, and then extended to all of $ W $. If an operator $ \mathcal{A} $ has commuting semisimple and locally nilpotent parts $ \mathcal{S} $ and $ \mathcal{N} $, we use $ x^\mathcal{A} $ to mean $ x^\mathcal{S} x^\mathcal{N} $ and
\begin{align*}
(x + x_1)^\mathcal{A} = (x + x_1)^\mathcal{S}(x + x_1)^\mathcal{N} = x^\mathcal{A} (1 + x_1/x)^\mathcal{A}.
\end{align*}

Given a formal series 
\begin{align*}
f(x_1, \dots, x_\N) = \sum_{\substack{r_1, \dots, r_\N \in \mathbb{C} \\ k_1 , \dots, k_\N \in \mathbb{Z}_{\geq 0}}} a_{r_1,\dots, r_\N, k_1, \dots, k_\N} x_1^{r_1} \cdots x_\N^{r_\N} (\log x_1)^{k_1} \cdots (\log x_\N)^{k_\N},
\end{align*} 
we can specialize a formal variable to a sum of two formal variables by understanding logarithms of formal variables as explained above in \eqref{eq:formal meanings}. For example,
\begin{align*}
f(x + x_0, \dots, x_\N) = \sum_{\substack{r_1, \dots, r_\N \in \mathbb{C} \\ k_1 , \dots, k_\N \in \mathbb{Z}_{\geq 0}}} a_{r_1,\dots, r_\N, k_1, \dots, k_\N} (x + x_0)^{r_1} \cdots x_\N^{r_\N} (\log (x + x_0))^{k_1} \cdots (\log x_\N)^{k_\N}.
\end{align*}

For each $ z \in \mathbb{C}^\times $, define $ \arg z $ to be the argument of $ z $ satisfying $ 0 \leq \arg z < 2\pi $. For each $ p \in \mathbb{Z} $, we use $ l_p(z) $ to denote the logarithm of $ z $ given by $ \log |z| +  ( \arg z + 2 p \pi) \iu $.  When a formal series $ f(x_1, \dots, x_\N) $ has countable support, it can be specialized to the series 
\begin{align*}
f^{p_1, \dots, p_\N} (z_1, \dots, z_\N) = \sum_{\substack{r_1, \dots, r_\N \in \mathbb{C} \\ k_1 , \dots, k_\N \in \mathbb{Z}_{\geq 0}}} a_{r_1,\dots, r_\N, k_1, \dots, k_\N} e^{r_1 l_{p_1}(z_1) } \cdots e^{r_\N l_{p_\N}(z_\N) } (l_{p_1}(z_1))^{k_1} \cdots (l_{p_\N}(z_\N))^{k_\N},
\end{align*}
for any $ z_1, \dots, z_\N \in \mathbb{C}^\times $. If a series $ f $ does not have the variables explicitly stated, we will write 
\begin{align*}
f|_{x_1 = z_1, \dots , x_\N = z_\N}^{p_1, \dots , p_\N}
\end{align*}
to denote $ f^{p_1, \dots, p_\N} (z_1, \dots, z_\N) $, again suppressing the $ \log $ variables for brevity. If $ f^{p_1, \dots, p_\N} (z_1, \dots, z_\N) $ converges absolutely in some appropriate region, it converges to a multivalued analytic function $ f(z_1, \dots, z_\N) $ with branches given by $ f^{p_1, \dots, p_\N} (z_1, \dots, z_\N) $. We also write expressions such as
\begin{equation*}
\langle w_0, \mathcal{Y}(w_1,z_1) \cdots \mathcal{Y}(w_\N,z_\N) w_{\N+1} \rangle : = 
\langle w_0, \mathcal{Y}(w_1,x_1) \cdots \mathcal{Y}(w_\N,x_\N) w_{\N+1} \rangle|
_{x_1 = z_1, \dots,  x_\N = z_\N} 
\end{equation*}
to denote the (multivalued) series of complex numbers. Each branch will have the value of $ z_i^r = e^{r \log z_i} $ dependent on the choice of value $ \log z_i $.

Suppose we have operators $ g $ and $ \mathcal{L}_g $ on a space $ W $ such that $ g = e^{2 \pi \iu \mathcal{L}_g} $, and $ \mathcal{L}_g $ has commuting semisimple and locally nilpotent parts $ \mathcal{S}_g $ and $ \mathcal{N}_g $. (This is always possible when $ W $ is the direct sum of finite-dimensional subspaces on which $ g $ acts invertibly.) Then $ \mathcal{S}_g $ is determined up to its eigenvectors with eigenvalues modulo $\mathbb{Z} $. We will always assume $ \mathcal{S}_g $ is chosen so that its eigenvalues lie in $ \{ z \in \mathbb{C} : 0 \leq \Re(z) < 1 \} $. We sometimes omit the subscript $ g $ from $ \mathcal{L} $, $ \mathcal{S} $ and $ \mathcal{N} $ when the dependence on $ g $ is clear from context.

We will use $ V_{(n)} $ to denote the $ L(0) $-eigenspaces with eigenvalue $ n \in \mathbb{Z} $ of a vertex operator algebra $ V $, and use $ V_+ $ to denote $ \coprod_{n \in \mathbb{Z}_{> 0}} V_{(n)} $. If $ g $ is an automorphism of $ V $, we use $ V^{[\alpha]} $, $ V^{[\alpha]}_{(n)} $ and $ V^{[\alpha]}_+ $ to denote the spaces of generalized eigenvectors of $ g $ (acting on $ V $, $ V_{(n)} $ and $ V_+ $, respectively) with eigenvalue $ e^{2 \pi i \alpha} $. We will use $ W_{[n]} $ to denote the $ L(0) $-generalized eigenspaces of a twisted $ V $-module (see Definition 3.1 of \cite{HUANG2018346}) with eigenvalue $ n \in \mathbb{C} $. We say $ \wt w = n $ is the (conformal) weight of a vector $ w \in W_{[n]} $.

Suppose we have a tensor product $ W_0 \otimes \dots \otimes W_{\N+1} $ of $ \mathbb{C} $-graded spaces graded by weight. Given weight-homogeneous elements $ w_i \in W_i $, $ i = 0, \dots, \N + 1 $, we will frequently use $ \sigma $ to denote $ \Re( \wt w_0 + \dots + \wt w_{\N+1} ) $, i.e., the real part of the weight of the tensor product $ w_0 \otimes \cdots \otimes w_{\N+1} $.

\section[A Jacobi identity for intertwining operators among untwisted and twisted modules]{A Jacobi identity for \gtype\ intertwining operators}
\label{sec:Jacobi}
We use the definition of vertex operator algebra $ V = (V, Y, \one, \omega) $ in Definition 3.1.22 of \cite{LepowskyLi}. It is important that our base field is $ \mathbb{C} $. Recall that an automorphism $ g $ of a vertex operator algebra $ V $ is a linear automorphism of $ V $ such that $ gY(u,x)v = Y(gu,x)gv $ for all $ u,v \in V $, $ g\!\one = \one $ and $ g\omega = \omega $. We will use the notions of twisted modules and twisted intertwining operators among twisted modules as defined in Definition 3.1 and Definition 4.1 of \cite{HUANG2018346}, respectively. In this paper, we will not require the $ g $-action on the $ g $-twisted module. We recall the definition of an intertwining operator among twisted modules. 

\begin{definition}
Let $ V $ be a vertex operator algebra with automorphisms $ g_1, g_2, g_3 $. Let $ W_i $ be a $ g_i $-twisted $ V $-module, for $ i = 1,2,3 $. A \emph{twisted intertwining operator of type $\type{W_1}{W_2}{W_3}$} is a linear map
\begin{equation}
\begin{aligned}
\mathcal{Y}( \cdot, x) \cdot : W_1 \otimes W_2 &\to W_3 \{ x \}[\log x], \\
w_1 \otimes w_2 &\mapsto \sum_{k=0}^K \sum_{n \in \mathbb{C}} \mathcal{Y}_{n,k} (w_1) w_2 x^{-n-1} (\log x)^k
\end{aligned}
\end{equation}
satisfying the following conditions:
\begin{enumerate}
\item[(i)] (\emph{Lower truncation}) For all $ w_1 \in W_1 $, $ w_2 \in W_2 $, $ n \in \mathbb{C} $ and $ k = 0, \dots, K $,
\begin{equation*}
\mathcal{Y}_{n+l,k}(w_1)w_2 = 0 \text{ for all sufficiently large } l \in \mathbb{Z}.
\end{equation*} 
\item[(ii)] (\emph{$ L(-1) $-derivative property}) For all $ w_1 \in W_1 $,
\begin{equation*}
\frac{d}{dx} \mathcal{Y}( w_1, x) = \mathcal{Y}( L(-1) w_1, x).
\end{equation*}
\item[(iii)] (\emph{Duality}) For all $ u \in V $, $ w_1 \in W_1 $, $ w_2 \in W_2 $ and $ w_3' \in W_3' $, there exists a formal series
\begin{equation}\label{eq:converge function}
\begin{aligned} 
f(x_0,x_1,x_2) &= \sum_{i,j,k,l,m,n = 0}^N a_{ijklmn} x_0^{r_i} x_1^{s_j} x_2^{t_k}  (\log x_0)^l (\log x_1)^m (\log x_2)^n \\
&\in \mathbb{C}\{x_0,x_1,x_2\}[\log x_0, \log x_1, \log x_2 ] 
\end{aligned}
\end{equation}
such that for $ p_1, p_2, p_{12} \in \mathbb{Z} $, the series
\begin{equation}
\begin{aligned}
 \langle w_3', (Y_{W_3})^{p_1}(u,z_1) \mathcal{Y}^{p_2}(w_1,z_2) w_2 \rangle &= \langle w_3', Y_{W_3}(u,x_1) \mathcal{Y}(w_1,x_2) w_2 \rangle |_{x_1 = z_1, x_2 = z_2}^{p_1,p_2} , \\
 \langle w_3',  \mathcal{Y}^{p_2}(w_1,z_2) (Y_{W_2})^{p_1}(u,z_1)w_2 \rangle &= \langle w_3',  \mathcal{Y}(w_1,x_2) Y_{W_2}(u,x_1)w_2 \rangle|_{x_1 = z_1, x_2 = z_2}^{p_1,p_2} , \\
 \langle w_3',  \mathcal{Y}^{p_2}((Y_{W_1})^{p_{12}}(u,z_1 - z_2)w_1,z_2) w_2 \rangle &= \langle w_3',  \mathcal{Y}(Y_{W_1}(u,x_0)w_1,x_2) w_2 \rangle |_{x_0 = z_1-z_2, x_2 = z_2}^{p_{12},p_2} ,
 \end{aligned}
\end{equation}
are absolutely convergent in the regions $ |z_1| > |z_2| > 0 $, $|z_2| > |z_1| > 0 $ and $ |z_2| > |z_1 - z_2| > 0 $, respectively. Moreover, they converge to 
\begin{equation}\label{eq:branches}
\begin{gathered}
f^{p_1,p_1,p_2}(z_1 - z_2,z_1,z_2), \quad f^{p_2,p_1,p_2}(z_1-z_2,z_1,z_2), \quad f^{p_{12}, p_2, p_2}(z_1-z_2,z_1,z_2) 
\end{gathered}
\end{equation} 
in the regions 
\begin{equation}\label{eq:regions}
\begin{gathered}
|z_1| > |z_2| > 0 \text{ and } -\frac{\pi}{2} < \arg(z_1 - z_2) - \arg(z_1)  < \frac{\pi}{2} ,\\
|z_2| > |z_1| > 0 \text{ and } - \frac{3 \pi }{2} < \arg(z_1 - z_2) - \arg(z_2) < -\frac{\pi}{2} ,\\
|z_2| > |z_1 - z_2| > 0 \text{ and } -\frac{\pi}{2} < \arg(z_1) - \arg(z_2)  < \frac{\pi}{2} , \\
\end{gathered}
\end{equation} 
respectively. 
\end{enumerate}
\end{definition}

We are interested in the following special case. 

\begin{definition}
Let $ V $ be a vertex operator algebra with an automorphism $ g $. Let $ W_1 $ be an untwisted $ V $-module (i.e., twisted by $ 1 = \id_V $). Let $ W_2 $ and $ W_3 $ be $ g $-twisted $ V $-modules. A twisted intertwining operator of type $ \type{W_1}{W_2}{W_3} $ is said to be of \emph{\gtype}, and we refer to it as a \emph{\gtype\ intertwining operator}.
\end{definition}  

We derive a Jacobi identity for \gtype\ intertwining operators in what follows. Fix $ V $, $ g $, and a \gtype\ intertwining operator $ \mathcal{Y}(\cdot, x): W_1 \otimes W_2 \to W_3 \{ x \}[\log x] $. Denote by $ \mathcal{L} $ the logarithm of $ g $ divided by $ 2 \pi \iu $ with commuting semisimple and locally nilpotent parts $ \mathcal{S} $ and $ \mathcal{N} $ such that the real parts of the eigenvalues of $ \mathcal{S} $ are in $ [0,1) $.  Let $ u \in V $. Recall from \cite{Huang2016AssociativeAF} that, for $ i = 2,3 $, we have 
\begin{align*}
 {Y_{W_i,0}}( u, x) = Y_{W_i}(x^{\mathcal{N}}u,x) ,
\end{align*}  
where $ x^{\mathcal{N}} = e^{\mathcal{N} \log x } $ and $ Y_{W_i,0}(u,x) = \sum_{n\in \mathbb{C}} u_{n,0} x^{-n-1} $ contains exactly the $( \log x) $-free terms of $ Y_{W_i}(u,x) $.  For each $ k \in \mathbb{Z}_{\geq 0} $, there is a formal series $ f_k(x_0,x_1,x_2)  $ as in \eqref{eq:converge function}, such that
\begin{gather*}
\langle w_3', Y_{W_3}(\mathcal{N}^k u,x_1) \mathcal{Y}(w_1,x_2) w_2 \rangle |_{x_1 = z_1, x_2 = z_2}^{p_1,p_2} , \\
\langle w_3',  \mathcal{Y}(w_1,x_2) Y_{W_2}(\mathcal{N}^k u,x_1)w_2 \rangle|_{x_1 = z_1, x_2 = z_2}^{p_1,p_2} , \\
\langle w_3',  \mathcal{Y}(Y_{W_1}(\mathcal{N}^k u,x_0)w_1,x_2) w_2 \rangle |_{x_0 = z_1-z_2, x_2 = z_2}^{p_{12},p_2} ,
\end{gather*}
converge to the respective branches of $ f_k(z_1 - z_2,z_1,z_2) $ in the respective regions as in \eqref{eq:branches} and \eqref{eq:regions}. This implies that
\begin{gather*}
\langle w_3', Y_{W_3}((\log x_1)^k \mathcal{N}^k u,x_1) \mathcal{Y}(w_1,x_2) w_2 \rangle |_{x_1 = z_1, x_2 = z_2}^{p_1,p_2} , \\
\langle w_3',  \mathcal{Y}(w_1,x_2) Y_{W_2}((\log x_1)^k \mathcal{N}^k u,x_1)w_2 \rangle|_{x_1 = z_1, x_2 = z_2}^{p_1,p_2} , \\
\langle w_3',  \mathcal{Y}(Y_{W_1} ((\log x_1)^k \mathcal{N}^k u,x_0)w_1,x_2) w_2 \rangle |_{x_0 = z_1-z_2, x_1 = z_1, x_2 = z_2}^{p_{12},p_2, p_2} ,
\end{gather*}
converge to the respective branches of $ (\log z_1)^k f_k(z_1 - z_2,z_1,z_2) $ in the respective regions. Hence, by defining
\begin{equation}
f(x_0,x_1,x_2) = \sum_{k = 0}^K \frac{1}{k!}(\log x_1)^k f_k(x_0,x_1,x_2) ,
\end{equation}
where $ K $ is sufficiently large so that $ \mathcal{N}^{K+1} u = 0 $, we find that
\begin{equation} \label{eq:Y_0 series}
\begin{gathered}
\langle w_3', Y_{W_3,0}(u,x_1) \mathcal{Y}(w_1,x_2) w_2 \rangle |_{x_1 = z_1, x_2 = z_2}^{p_1,p_2} , \\
\langle w_3',  \mathcal{Y}(w_1,x_2) Y_{W_2,0}(u,x_1)w_2 \rangle|_{x_1 = z_1, x_2 = z_2}^{p_1,p_2} , \\
\sum_{k=0}^K \frac{1}{k!}\langle w_3',  \mathcal{Y}(Y_{W_1}((\log x_1)^k \mathcal{N}^k u,x_0)w_1,x_2) w_2 \rangle |_{x_0 = z_1-z_2, x_1 = z_1, x_2 = z_2}^{p_{12},p_2, p_2} ,
\end{gathered}
\end{equation}
converge to the respective branches of $ f(z_1 - z_2,z_1,z_2)  $ in the respective regions. 

Since $ \log x_1 |_{x_1 = z_1}^{p_2}$ and $ \left(\log x_2 + \log (1 + x_0/x_2)\right)|_{x_0 = z_1-z_2, x_2 = z_2}^{p_{12}, p_2} $ both converge to the same function in the region $ |z_2| > |z_1 - z_2| > 0 $ and $ -\frac{\pi}{2} < \arg(z_1) - \arg(z_2)  < \frac{\pi}{2} $, we can replace the third series in \eqref{eq:Y_0 series} with 
\begin{gather*}
\langle w_3',  \mathcal{Y}(Y_{W_1}((x_2 + x_0)^{\mathcal{N}} u,x_0)w_1,x_2) w_2 \rangle |_{x_0 = z_1-z_2, x_2 = z_2}^{p_{12}, p_2},
\end{gather*}
where $ ( x_2 + x_0 )^{\mathcal{N}} $ means  $ \sum_{k = 0}^\infty \frac{1}{k!} \left( \log x_2 + \log\left( 1+x_0/x_2 \right) \right)^k \mathcal{N}^k $, for the same result. 

Now assume that $ \mathcal{S} u = \alpha u $ for some $ \alpha \in \mathbb{C} $ with $ \Re(\alpha) \in [0,1) $, i.e., $ u \in V^{[\alpha]} $.
The \emph{equivariance property} 
\begin{equation}\label{eq:equivariance}
e^{2 \pi \iu x \frac{d}{dx}} Y_{W_i}(gv,x) =  Y_{W_i}(v,x) , \quad \text{for all } v \in V 
\end{equation}
satisfied by the $ g $-twisted modules $ W_i $ for $ i = 2,3 $, gives
\begin{align*}
e^{2 \pi \iu x \frac{d}{dx}} Y_{W_i,0}(u,x) &= e^{2 \pi \iu x \frac{d}{dx}} Y_{W_i}(x^{\mathcal{N}} u,x) = \sum_{k = 0}^K e^{2 \pi \iu x \frac{d}{dx}} \frac{1}{k!}(\log x)^k Y_{W_i}( \mathcal{N}^k u,x) \\
&= \sum_{k = 0}^K \frac{1}{k!}(\log x + 2\pi \iu)^k e^{2 \pi \iu x \frac{d}{dx}} Y_{W_i}( \mathcal{N}^k u,x) = \sum_{k = 0}^K \frac{1}{k!}(\log x + 2\pi \iu)^k Y_{W_i}( g^{-1}\mathcal{N}^k u,x) \\
&= Y_{W_i}( g^{-1} e^{\mathcal{N}(\log x + 2 \pi \iu)} u,x) = Y_{W_i}( e^{-2 \pi \iu \mathcal{S} } x^{\mathcal{N}} u,x)\\
&= Y_{W_i,0}(e^{-2 \pi \iu \mathcal{S} }u,x) = e^{-2 \pi \iu \alpha} Y_{W_i,0}(u,x).
\end{align*}
Hence,
\begin{align*}
Y_{W_i,0}(u,x_1) x_1^{\alpha} \in (\End W_i)[[x_1, x_1^{-1} ]] \quad \text{ for } i = 2,3.
\end{align*} 
Since $ W_1 $ is untwisted, we have 
\begin{align*}
Y_{W_1}(u,x_0) \in (\End W_1)[[x_0, x_0^{-1} ]].
\end{align*}

If $ w_1 $, $ w_2 $, $ w_3' $ and $ u $ are weight homogeneous, then we can see that the coefficients in
\begin{equation} \label{eq:sequences}
\begin{gathered}
\begin{aligned}
\langle w_3', Y_{W_3,0}(u,x_1) \mathcal{Y}(w_1,x_2) w_2 \rangle &= \sum_{n \in \alpha + \mathbb{Z}} \sum_{h \in \mathbb{C}} \sum_{k = 0}^K \langle w_3', u_n \mathcal{Y}_{h,k}(w_1) w_2 \rangle x_1^{-n-1} x_2^{-h-1} (\log x_2)^k, \\
\langle w_3',  \mathcal{Y}(w_1,x_2) Y_{W_2,0}(u,x_1)w_2 \rangle &= \sum_{n \in \alpha + \mathbb{Z}} \sum_{h \in \mathbb{C}} \sum_{k = 0}^K \langle w_3',  \mathcal{Y}_{h,k}(w_1) u_n w_2 \rangle x_1^{-n-1} x_2^{-h-1} (\log x_2)^k, 
\end{aligned}  
\\
\begin{aligned}
&\langle w_3',  \mathcal{Y}(Y_{W_1}((x_2 + x_0)^\mathcal{N}u,x_0)w_1,x_2) w_2 \rangle \\
&\qquad= \sum_{n \in \mathbb{Z}} \sum_{h \in \mathbb{C}} \sum_{j = 0 }^J \sum_{k = 0}^K \frac{1}{k!} \langle w_3',  \mathcal{Y}_{h,j}( (\mathcal{N}^k u)_n w_1) w_2 \rangle x_0^{-n-1} x_2^{-h-1} (\log x_2)^j (\log(x_2 + x_0))^k
\end{aligned}
\end{gathered}
\end{equation}
are zero unless $ \wt w_3' = \wt u + \wt w_1+ \wt w_2 -n-1 - h - 1  $, by the $ L(0) $-bracket relations. Extending this to the case where $ w_1 $, $ w_2 $, $ w_3' $ and $ u $ need not be weight homogeneous, we see that these series have support only when the powers of $ x_2 $ are real numbers added to finitely many complex numbers. By the truncation property, the real numbers that are added to the finitely many complex numbers form a strictly monotonic sequence indexed by $ \mathbb{Z}_{>0} $. 

We recall that a \emph{unique expansion set} is a subset $ S $ of $ \mathbb{C} \times \mathbb{C} $ satisfying the following condition: If a series
\begin{align*}
\sum_{(\alpha,\beta) \in S} a_{\alpha,\beta} z^\alpha (\log z)^\beta = \sum_{(\alpha,\beta) \in S} a_{\alpha,\beta} e^{\alpha \log z} (\log z)^\beta
\end{align*}
is absolutely convergent and convergent to 0 on some non-empty open subset of $ \mathbb{C}^\times $ (for some chosen branch of $ \log z $), then $ a_{\alpha,\beta} = 0 $ for all $ (\alpha, \beta) \in S $. 

In \cite{Huang:2017uke}, it was shown that any strictly monotonic sequence $ \{ n_i \}_{i \in \mathbb{Z}_{>0}} $ of real numbers together with a list $ m_1, \dots , m_\el $ of distinct real numbers produces a unique expansion set
\begin{equation}\label{eq:exp set}
\{ n_i + m_j \iu : i \in \mathbb{Z}_{>0}, \ j = 1, \dots , \el \} \times \{0, \dots, N \}
\end{equation}
for any $ N \in \mathbb{Z}_{\geq 0 } $. We can see that the three series in \eqref{eq:sequences}, and the three series $ f(x_1-x_2,x_1,x_2) $, $ f(-x_2+x_1, x_1, x_2) $ and $ f(x_0,x_2+x_0,x_2) $ are double sums over unique expansion sets of type \eqref{eq:exp set} with the directions of the monotonic sequences matching. 

Assume that $ S $ and $ T $ are unique expansion sets and we have two series
\begin{equation}\label{eq:two series}
\sum_{(\alpha,\beta) \in S} \sum_{(\gamma,\delta) \in T} a_{\alpha,\beta,\gamma,\delta} z_1^\alpha (\log z_1)^\beta z_2^\gamma (\log z_2)^\delta
 \quad \text{and} \quad \sum_{(\alpha,\beta) \in S} \sum_{(\gamma,\delta) \in T} b_{\alpha,\beta,\gamma,\delta} z_1^\alpha (\log z_1)^\beta z_2^\gamma (\log z_2)^\delta
\end{equation}
that are both absolutely convergent in some non-empty open set of $ \mathbb{C}^\times \times \mathbb{C}^\times $, and furthermore converge to the same values. Then for fixed $ z_1 $ (in some open set of $ \mathbb{C}^\times $), 
\begin{align*}
 \sum_{(\gamma,\delta) \in T} \left( \sum_{(\alpha,\beta) \in S}(a_{\alpha,\beta,\gamma,\delta} - b_{\alpha,\beta,\gamma,\delta} ) z_1^\alpha (\log z_1)^\beta \right) z_2^\gamma (\log z_2)^\delta  
\end{align*}
is absolutely convergent for all $ z_2 $ in some open set of $ \mathbb{C}^\times $ and is furthermore convergent to zero. Since $ T $ is a unique expansion set, 
\begin{align*}
 \sum_{(\alpha,\beta) \in S}(a_{\alpha,\beta,\gamma,\delta} - b_{\alpha,\beta,\gamma,\delta} ) z_1^\alpha (\log z_1)^\beta 
\end{align*}
converges to zero for all $ (\gamma,\delta) \in T $. Furthermore, this series converges absolutely for all $ z_1 $ in an open set of $ \mathbb{C}^\times $.  Since $ S $ is a unique expansion set, $ a_{\alpha,\beta,\gamma,\delta} - b_{\alpha,\beta,\gamma,\delta} = 0 $ for all $ (\alpha,\beta,\gamma,\delta) \in S \times T $. Hence, the two series in \eqref{eq:two series} are equivalent. 

We now use this argument to show that the following series are equivalent. Since
\begin{align*}
x_1^\alpha \langle w_3', Y_{W_3,0}(u,x_1) \mathcal{Y}(w_1,x_2) w_2 \rangle|_{x_1 = z_1,x_2 = z_2}^{p_1,p_2} &=  x_1^\alpha f(x_0,x_1,x_2)|_{x_0 = z_1-z_2, x_1 = z_1,x_2 = z_2}^{p_1,p_1,p_2}\\
&= x_1^\alpha f(x_1 - x_2,x_1,x_2) |_{x_1 = z_1,x_2 = z_2}^{p_1,p_2} 
\end{align*}
in the region $ |z_1| > |z_2| > 0 $ and $ -\frac{\pi}{2} < \arg(z_1 -z_2) - \arg(z_1) < \frac{\pi}{2} $, we have 
\begin{align*}
x_1^\alpha f(x_1 - x_2,x_1,x_2) = x_1^\alpha \langle w_3', Y_{W_3,0}(u,x_1) \mathcal{Y}(w_1,x_2) w_2 \rangle \in \mathbb{C}\{x_2\}[[x_1,x_1^{-1}]][\log x_2 ]
\end{align*}
being lower truncated in $ x_2 $. Since
\begin{align*}
x_1^\alpha \langle w_3',  \mathcal{Y}(w_1,x_2) Y_{W_2,0}(u,x_1)w_2 \rangle|_{x_1 = z_1, x_2 = z_2}^{p_1,p_2} &=x_1^\alpha f(x_0,x_1,x_2)|_{x_0 = z_1 - z_2, x_1 = z_1, x_2 = z_2}^{p_2,p_1,p_2}\\
&= x_1^\alpha f(-x_2 + x_1,x_1,x_2)|_{x_1 = z_1, x_2 = z_2}^{p_1,p_2}
\end{align*}
in the region $ |z_2| > |z_1| > 0 $ and $ -\frac{3\pi}{2} < \arg(z_1 - z_2) - \arg(z_2) < - \frac{\pi}{2} $, we have
\begin{align*}
x_1^\alpha f(-x_2 + x_1,x_1,x_2) = x_1^\alpha \langle w_3',  \mathcal{Y}(w_1,x_2) Y_{W_2,0}(u,x_1)w_2 \rangle \in \mathbb{C}\{x_2\}[[x_1,x_1^{-1}]][\log x_2 ] 
\end{align*}
being lower truncated in $ x_1 $.
Since
\begin{align*}
&(x_2 + x_0)^\alpha \langle w_3',  \mathcal{Y}(Y_{W_1}((x_2 + x_0)^{\mathcal{N}} u,x_0)w_1,x_2) w_2 \rangle|_{x_0 = z_1 - z_2, x_2 = z_2}^{p_{12}, p_2} \\
&=(x_2 + x_0)^\alpha f(x_0,x_1,x_2)|_{x_0 = z_1 - z_2, x_1 = z_1, x_2 = z_2}^{p_{12}, p_2, p_2}\\
&= (x_2 + x_0)^\alpha f(x_0,x_2+x_0,x_2)|_{x_0 = z_1 - z_2, x_2 = z_2}^{p_{12}, p_2}
\end{align*}
in the region $ |z_2| > |z_1 - z_2| > 0 $ and $ - \frac{\pi}{2} < \arg(z_1) - \arg(z_2) < \frac{\pi}{2} $, we have
\begin{align*}
(x_2 + x_0)^\alpha f(x_0,x_2+x_0,x_2) &= (x_2 + x_0)^\alpha \langle w_3',  \mathcal{Y}(Y_{W_1}((x_2 + x_0)^{\mathcal{N}} u,x_0)w_1,x_2) w_2 \rangle \\
& \in \mathbb{C}\{x_2\}[[x_0,x_0^{-1}]][\log x_2 ] 
\end{align*}
being lower truncated in $ x_0 $. So, we have
\begin{align*}
x_1^\alpha f(x_0,x_1,x_2) \in \mathbb{C}\{x_2\}[x_0,x_0^{-1},x_1,x_1^{-1} , \log x_2]
\end{align*}
such that the formal series
\begin{gather*}
x_1^\alpha \langle w_3', Y_{W_3,0}(u,x_1) \mathcal{Y}(w_1,x_2) w_2 \rangle  = x_1^\alpha f(x_1 - x_2,x_1,x_2) \in \mathbb{C}\{x_2\}[[x_1,x_1^{-1}]][\log x_2 ], \\
x_1^\alpha \langle w_3',  \mathcal{Y}(w_1,x_2) Y_{W_2,0}(u,x_1)w_2 \rangle  = x_1^\alpha f(-x_2 + x_1,x_1,x_2) \in \mathbb{C}\{x_2\}[[x_1,x_1^{-1}]][\log x_2 ], \\
(x_2 + x_0)^\alpha \langle w_3',  \mathcal{Y}(Y_{W_1}((x_2 + x_0)^{\mathcal{N}} u,x_0)w_1,x_2) w_2 \rangle  \\
= (x_2 + x_0)^\alpha f(x_0,x_2+x_0,x_2) \in \mathbb{C}\{x_2\}[[x_0,x_0^{-1}]][\log x_2 ],
\end{gather*}
are lower truncated in $ x_2 $, $ x_1 $ and $ x_0 $, respectively. Since we can multiply both sides of 
\begin{gather*}
x_0^{-1} \delta \left( \frac{x_1 - x_2}{x_0} \right) - x_0^{-1} \delta \left( \frac{-x_2 + x_1}{x_0} \right) = x_1^{-1} \delta \left( \frac{x_2 + x_0}{x_1} \right) 
\end{gather*}
by $ x_1^\alpha f(x_0,x_1,x_2) $, which has no logarithms of $ x_0 $ and $ x_1 $ and only integral powers of $ x_0 $ and $ x_1 $, we have
\begin{gather*}
x_0^{-1} \delta \left( \frac{x_1 - x_2}{x_0} \right) x_1^\alpha f(x_1 - x_2,x_1,x_2) - x_0^{-1} \delta \left( \frac{-x_2 + x_1}{x_0} \right) x_1^\alpha f(-x_2+x_1,x_1,x_2) \\
= x_1^{-1} \delta \left( \frac{x_2 + x_0}{x_1} \right) (x_2+x_0)^\alpha f(x_0,x_2+x_0,x_2),
\end{gather*}
hence
\begin{gather*}
x_0^{-1} \delta \left( \frac{x_1 - x_2}{x_0} \right) x_1^\alpha \langle w_3', Y_{W_3,0}(u,x_1) \mathcal{Y}(w_1,x_2) w_2 \rangle  - x_0^{-1} \delta \left( \frac{-x_2 + x_1}{x_0} \right) x_1^\alpha \langle w_3',  \mathcal{Y}(w_1,x_2) Y_{W_2,0}(u,x_1)w_2 \rangle  \\
= x_1^{-1} \delta \left( \frac{x_2 + x_0}{x_1} \right) (x_2 + x_0)^\alpha \langle w_3',  \mathcal{Y}(Y_{W_1}(( x_2 + x_0)^{\mathcal{N}}  u,x_0)w_1,x_2) w_2 \rangle,
\end{gather*}
for each fixed $ u \in V^[\alpha] $, $ w_1 \in W_1 $ and for all $ w_2 \in W_2 $, $ w_3' \in W_3' $. Hence, we have the following result.
\begin{proposition}
Let $ \mathcal{Y}( \cdot, x) \cdot : W_1 \otimes W_2 \to W_3 \{ x \}[\log x] $ be a twisted intertwining operator of \gtype. Then, for all $ u \in V^{[\alpha]} $, $ \alpha \in \mathbb{C} $ with $ \Re(\alpha) \in [0,1) $, and for all $ w_1 \in W_1 $, we have
\begin{equation}\label{eq:Y_0 Jacobi}
\begin{gathered}
x_0^{-1} \delta \left( \frac{x_1 - x_2}{x_0} \right) x_1^\alpha Y_{W_3,0} (u,x_1) \mathcal{Y}(w_1, x_2) - x_0^{-1} \delta \left( \frac{-x_2 + x_1}{x_0} \right) x_1^\alpha \mathcal{Y}(w_1, x_2) Y_{
W_2,0}(u,x_1) \\
= x_1^{-1} \delta \left( \frac{x_2 + x_0}{x_1} \right) \mathcal{Y}\left( Y_{W_1} \left( ( x_2 + x_0 )^{\mathcal{L}} u,x_0 \right)w_1, x_2 \right).
\end{gathered}\vspace{-2em}
\end{equation}
\qed
\end{proposition}
Here, we use a more condensed form of writing $ (x_2 + x_0)^{\mathcal{L}} = (x_2 + x_0)^{\mathcal{S}}(x_2 + x_0)^{\mathcal{N}} $, since $ \mathcal{S} $ and $ \mathcal{N} $ commute. 

Recalling that 
\begin{align*}
Y_{W_i,0}(x^{-\mathcal{N}} u ,x) = Y_{W_i}(u,x), \qquad \text{ for } i = 2,3,
\end{align*}
and observing that this Jacobi identity holds for $ u $  replaced with $ \mathcal{N}^k u $ (another eigenvector of $ S $ with eigenvalue $ \alpha $), we can multiply both sides by $ \frac{(-1)^k}{k!} (\log x_1)^k $, and sum over $ k = 0, \dots, K $ to obtain
\begin{equation*}
\begin{gathered}
x_0^{-1} \delta \left( \frac{x_1 - x_2}{x_0} \right) Y_{W_3}(u,x_1) \mathcal{Y}(w_1, x_2) - x_0^{-1} \delta \left( \frac{-x_2 + x_1}{x_0} \right) \mathcal{Y}(w_1, x_2) Y_{W_2}(u,x_1) \\
= x_1^{-1} \delta \left( \frac{x_2 + x_0}{x_1} \right)  \mathcal{Y}\left( Y_{W_1} \left( \left(\frac{x_2 + x_0}{x_1} \right)^{\mathcal{L} }u,x_0 \right)w_1, x_2 \right).
\end{gathered}
\end{equation*}

Since this form of the Jacobi identity makes no reference to the eigenvalue of $ u $, we can extend this to all $ u \in V $ that are sums of generalized eigenvectors of $ g $. Since $ V $ is spanned by generalized eigenvectors of $ g $, we have the following general result.
\begin{proposition}
Let $ \mathcal{Y}( \cdot, x) \cdot : W_1 \otimes W_2 \to W_3 \{ x \}[\log x] $ be a twisted intertwining operator of \gtype. Then, for all $ u \in V $ and $ w_1 \in W_1 $, we have
\begin{equation}\label{eq:Y Jacobi}
\begin{gathered}
x_0^{-1} \delta \left( \frac{x_1 - x_2}{x_0} \right) Y_{W_3}(u,x_1) \mathcal{Y}(w_1, x_2) - x_0^{-1} \delta \left( \frac{-x_2 + x_1}{x_0} \right) \mathcal{Y}(w_1, x_2) Y_{W_2}(u,x_1) \\
= x_1^{-1} \delta \left( \frac{x_2 + x_0}{x_1} \right)  \mathcal{Y}\left( Y_{W_1} \left( \left(\frac{x_2 + x_0}{x_1} \right)^{\mathcal{L} }u,x_0 \right)w_1, x_2 \right).
\end{gathered} \vspace{-2em}
\end{equation}
\qed
\end{proposition}

In what follows, the subscript of $ Y_{W_i} $ and the $ 0 $ in $ u_{n,0} $ will be dropped for brevity. 

The Jacobi identities \eqref{eq:Y_0 Jacobi} and \eqref{eq:Y Jacobi} allow us to work with \gtype\ intertwining operators using formal calculus. Equation \eqref{eq:Y_0 Jacobi} has the benefit of being free of $ \log x_1 $ with only integral powers of $ x_1 $, allowing residues with respect to $ x_1 $ to be taken. Since both equations are free of $ \log x_0 $ with only integral powers of $ x_0 $, we can use \eqref{eq:Y Jacobi} when taking residues with respect to $ x_0 $. We will now prove several results using the Jacobi identities. 

Given $ u \in V^{[\alpha]} $ for $ \alpha \in \mathbb{C} $ with $ \Re(\alpha) \in [0,1) $, we define
\begin{equation}
Y^+ (u, x) := \sum_{k=0}^K \sum_{n \in \alpha + \mathbb{Z}_{<0}} u_{n,k} x^{-n-1} (\log x)^k \quad \text{and} \quad Y^- (u, x) := \sum_{k=0}^K  \sum_{n \in \alpha + \mathbb{Z}_{\geq 0}} u_{n,k} x^{-n-1} (\log x)^k  .
\end{equation}
We use the notation for regular and singular parts of $ Y(u,x) $ because they contain the non-negative and negative integral powers of $ x $, respectively, when $ W $ is untwisted.

\begin{lemma}
Let $ u \in V^{[\alpha]} $ for $ \alpha \in \mathbb{C} $ with $ \Re(\alpha) \in [0,1) $, and let $ w_1 \in W_1 $. Then 
\begin{equation}\label{eq:take u(-1) out Y}
\begin{aligned}
Y^+ (u, x_2)\mathcal{Y}(w_1, x_2)  +  \mathcal{Y}(w_1, x_2) Y^- (u, x_2) &= \sum_{k \geq 0} x_2^{- k} \mathcal{Y}\left( \left( {\mathcal{L} \choose k}   u  \right)_{-1+k} w_1 , x_2 \right) .
\end{aligned}
\end{equation}
\end{lemma}

\begin{proof}
Multiplying the right-hand side of \eqref{eq:Y_0 Jacobi} by $ x_0^{-1} $ and then taking $ \res_{x_0} \res_{x_1} $ gives
\begin{align*}
&\res_{x_0} x_0^{-1} \mathcal{Y}( Y( (x_2 + x_0)^{\mathcal{L}}u, x_0)w_1, x_2) \\
&= \res_{x_0} x_0^{-1} \sum_{n \in \mathbb{Z}} \sum_{k \geq 0} \mathcal{Y}\left( \left( {\mathcal{L} \choose k}  x_2^{\mathcal{L}-k} u  \right)_n w_1 , x_2 \right) x_0^k x_0^{-n-1}  \\
&= \sum_{k \geq 0} \mathcal{Y}\left( \left( {\mathcal{L} \choose k}  x_2^{\mathcal{L} - k} u  \right)_{-1+k} w_1 , x_2 \right).
\end{align*}
And doing the same to the left-hand side of \eqref{eq:Y_0 Jacobi} gives
\begin{align*}
& \res_{x_1} x_1^\alpha\left( (x_1 - x_2)^{-1} Y_0(u,x_1) \mathcal{Y} (w_1, x_2) -  (-x_2 + x_1)^{-1} \mathcal{Y}(w_1, x_2) Y_0(u,x_1) \right) \\
&= \res_{x_1} \left( \sum_{k \in \mathbb{Z}_{\geq 0}} x_1^{-1-k} x_2^k \sum_{n \in \mathbb{Z}} u_{\alpha+n} x_1^{-n-1} \mathcal{Y} (w_1, x_2) + \sum_{k \in \mathbb{Z}_{\geq 0}} x_2^{-1-k} x_1^k \mathcal{Y}(w_1, x_2)  \sum_{n \in \mathbb{Z}} u_{\alpha+n} x_1^{-n-1}  \right) \\
&= \sum_{n \in \mathbb{Z}_{< 0}} u_{\alpha+n} x_2^{-n-1}\mathcal{Y}(w_1, x_2)  + \mathcal{Y}(w_1, x_2) \sum_{n \in \mathbb{Z}_{\geq 0}} u_{\alpha + n} x_2^{-n-1}  \\
&= x_2^\alpha \sum_{n \in \alpha + \mathbb{Z}_{< 0}} u_n x_2^{-n-1}\mathcal{Y}(w_1, x_2)  + x_2^\alpha \mathcal{Y}(w_1, x_2) \sum_{n \in \alpha + \mathbb{Z}_{\geq 0}} u_n x_2^{-n-1}  \\
&= x_2^\alpha Y_0^+ (u, x_2)\mathcal{Y}(w_1, x_2)  +  x_2^\alpha \mathcal{Y}(w_1, x_2) Y_0^- (u, x_2),
\end{align*}
where we have defined
\begin{equation}
Y_0^+ (u, x) := \sum_{n \in \alpha + \mathbb{Z}_{<0}} u_n x^{-n-1} \quad \text{and} \quad Y_0^- (u, x) := \sum_{n \in \alpha + \mathbb{Z}_{\geq 0}} u_n x^{-n-1} .
\end{equation}
Hence, we have
\begin{equation}\label{eq:take u(-1) out Y_0}
\begin{aligned}
Y_0^+ (u, x_2)\mathcal{Y}(w_1, x_2)  +  \mathcal{Y}(w_1, x_2) Y_0^- (u, x_2) = \sum_{k \geq 0} \mathcal{Y}\left( \left( {\mathcal{L} \choose k}  x_2^{\mathcal{N} - k} u  \right)_{-1+k} w_1 , x_2 \right).
\end{aligned}
\end{equation}
Replacing $ u $ with $ \frac{(-1)^j}{j!}(\log x_2)^j \mathcal{N}^j u $ and summing over $ j = 0 , \dots, K $ gives the desired identity.
\end{proof}

\begin{lemma}
Let $ u \in V^{[\alpha]} $ for $ \alpha \in \mathbb{C} $ with $ \Re(\alpha) \in [0,1) $, and let $ w_1 \in W_1 $. Then 
\begin{equation}\label{eq: Y+/- commutators}
\begin{aligned}
[Y^-(u,x_1), \mathcal{Y}(w_1, x_2) ] &= \sum_{j,k \geq 0} (x_1 - x_2)^{-1-j}  x_2^{-k} \mathcal{Y}\left( \left( \left(\frac{x_2}{x_1} \right)^{\mathcal{L}} {\mathcal{L} \choose k}   u\right)_{j+k} w_1, x_2 \right) , \\
[Y^+(u,x_1), \mathcal{Y}(w_1, x_2) ] &=  - \sum_{j,k \geq 0} (-x_2 + x_1)^{-1-j}  x_2^{-k} \mathcal{Y}\left( \left( \left(\frac{x_2}{x_1} \right)^{\mathcal{L}} {\mathcal{L} \choose k}   u\right)_{j+k} w_1, x_2 \right)  .
\end{aligned}
\end{equation}
\end{lemma}

\begin{proof}
From equation \eqref{eq:Y Jacobi}, we have
\begin{align}
&[Y(u,x_1), \mathcal{Y}(w_1, x_2) ] \nonumber\\
&= \res_{x_0} \left(x_0^{-1} \delta \left( \frac{x_1 - x_2}{x_0} \right) Y(u,x_1) \mathcal{Y}(w_1, x_2) - x_0^{-1} \delta \left( \frac{-x_2 + x_1}{x_0} \right) \mathcal{Y}(w_1, x_2) Y(u,x_1) \right) \nonumber\\
&= \res_{x_0} x_1^{-1} \delta \left( \frac{x_2 + x_0}{x_1} \right)  \mathcal{Y}\left( Y \left( \left(\frac{x_2 + x_0}{x_1} \right)^{\mathcal{L} }u,x_0 \right)w_1, x_2 \right) \nonumber\\
&= \res_{x_0} e^{x_0 \frac{\partial}{\partial x_2}} \left( x_1^{-1} \delta \left( \frac{x_2}{x_1} \right) \right)  \mathcal{Y}\left( Y \left( \left(\frac{x_2}{x_1} \right)^{\mathcal{L} } \left( 1 + x_0/x_2 \right)^\mathcal{L} u,x_0 \right)w_1, x_2 \right)\nonumber\\
&= \res_{x_0} \sum_{j \geq 0} \frac{1}{j!} x_0^j \left( \frac{\partial}{\partial x_2} \right)^j \left( x_1^{-1} \delta \left( \frac{x_2}{x_1} \right) \right) \sum_{n \in \mathbb{Z}} \sum_{k \geq 0} \mathcal{Y}\left( \left( \left(\frac{x_2}{x_1} \right)^{\mathcal{L}} {\mathcal{L} \choose k}   u\right)_n w_1, x_2 \right) x_0^k x_2^{-k}x_0^{-n-1} \nonumber\\
&= \sum_{j \geq 0} \sum_{k \geq 0} \frac{1}{j!} \left( \frac{\partial}{\partial x_2} \right)^j \left( x_1^{-1} \delta \left( \frac{x_2}{x_1} \right) \right)   \mathcal{Y}\left( \left( \left(\frac{x_2}{x_1} \right)^{\mathcal{L}} {\mathcal{L} \choose k}   u\right)_{j+k} w_1, x_2 \right)  x_2^{-k} \label{eq:commutator}.
\end{align}
Using
\begin{align*}
 \frac{1}{j!} \left( \frac{\partial}{\partial x_2} \right)^j x_1^{-1} \delta \left( \frac{x_2}{x_1} \right) = (x_1 - x_2)^{-1-j} - (-x_2 + x_1)^{-1-j} ,
\end{align*}
we can separate the terms with $ x_1^{-\alpha + n} $, for $ n \in \mathbb{Z}_{< 0} $, from the terms with $ x_1^{-\alpha + n} $, for $ n \in \mathbb{Z}_{\geq 0} $. Hence, we obtain the desired identities.
\end{proof}

\begin{lemma}
Let $ u \in V^{[\alpha]} $ for $ \alpha \in \mathbb{C} $ with $ \Re(\alpha) \in [0,1) $, and let $ w_1 \in W_1 $. Then 
\begin{equation} \label{eq:u alpha - 1 commutator}
[u_{\alpha-1}, \mathcal{Y}(w_1, x_2) ] = \sum_{k \geq 0} x_2^{-1 - k}  \mathcal{Y} \left(  \left( x_2^\mathcal{L} { \mathcal{L} - 1 \choose k} u \right)_{k} w_1, x_2 \right).
\end{equation}
\end{lemma}

\begin{proof}
Applying $ \res_{x_0} \res_{x_1} x_1^{-1} $ to both sides of \eqref{eq:Y_0 Jacobi} gives 
\begin{align*}
[u_{\alpha - 1}, \mathcal{Y}(w_1,x_2)] &= \res_{x_0} (x_2 + x_0)^{-1} \mathcal{Y} \left( Y \left( (x_2 + x_0)^\mathcal{L} u ,x_0 \right) w_1, x_2 \right)\\
&= \res_{x_0} \mathcal{Y} \left( Y \left( (x_2 + x_0)^{\mathcal{L}-1} u ,x_0 \right) w_1, x_2 \right)\\
&=  \res_{x_0} \sum_{n \in \mathbb{Z}} \sum_{k \geq 0} \mathcal{Y} \left(  \left( x_2^\mathcal{L} { \mathcal{L} - 1 \choose k} u \right)_n w_1, x_2 \right) x_2^{-1-k} x_0^k x_0^{-n-1} \\ 
&= \sum_{k \geq 0} x_2^{-1 - k}  \mathcal{Y} \left(  \left( x_2^\mathcal{L} { \mathcal{L} - 1 \choose k} u \right)_{k} w_1, x_2 \right),
\end{align*}
our desired result.
\end{proof}

The lemmas above will be used in the following section to derive differential equations for \gtype\ intertwining operators among twisted modules satisfying certain finiteness conditions. The special version below will be used to derive a ``twisted KZ equation". We say that an element $ w $ of a twisted module $ W $ is of \emph{lowest weight} if $ u_{n} w = 0 $ for all $ u \in V_{(1)} $ and $ n \in \mathbb{C} $ with $ \Re(n) > 0 $.
\begin{lemma}
Let $ u \in V^{[\alpha]}_{(1)} $ with $ \alpha \in \mathbb{C} $ and $ \Re(\alpha) \in [0,1) $, and let $ w_1 \in W_1 $ be of lowest weight. Then
\begin{equation}\label{eq:take u(-1) out Y - KZ}
\begin{aligned}
\mathcal{Y}( u_{-1} w_1, x_2 ) &= Y^+ (u, x_2)\mathcal{Y}(w_1, x_2)  +  \mathcal{Y}(w_1, x_2) Y^- (u, x_2)  -  x_2^{-1} \mathcal{Y}( (\mathcal{L} u)_{0} w_1 , x_2 )
\end{aligned}
\end{equation}
and
\begin{equation}\label{eq: Y+/- commutators - KZ}
\begin{aligned}
[Y^-(u,x_1), \mathcal{Y}(w_1, x_2) ] &= (x_1 - x_2)^{-1} \mathcal{Y}\left( \left( \left(x_2/x_1 \right)^{\mathcal{L}}  u\right)_{0} w_1, x_2 \right), \\
[Y^+(u,x_1), \mathcal{Y}(w_1, x_2) ] &= (x_2 - x_1)^{-1} \mathcal{Y}\left( \left( \left(x_2/x_1 \right)^{\mathcal{L}}   u\right)_{0} w_1, x_2 \right).
\end{aligned}
\end{equation}
\end{lemma}
\begin{proof}
This follows from the equations \eqref{eq:take u(-1) out Y} and \eqref{eq: Y+/- commutators} using $ u_n w_1 = 0 $ when $ n > 0 $.
\end{proof}

\section[Differential equations for intertwining operators among untwisted and twisted modules]{Differential equations for \gtype \ intertwining operators}
\label{sec:DEs}

In this section, we generalize Huang's method in \cite{HuangDEs} to derive differential equations for products of \gtype\ intertwining operators. The main idea of the method is to show that differential equations exist as a consequence of certain finiteness conditions.

\begin{definition}
Let $ V $ be a vertex operator algebra with an automorphism $ g $. Given a $ g $-twisted $ V $-module $ W $, define $ C_1(W) $ to be the subspace of $ W $ spanned by the elements of the form $ u_{\alpha - 1,0} w $ for all $ u \in V_{+}^{[\alpha]} $, $ \alpha \in \mathbb{C} $ with $ \Re(\alpha) \in [0,1) $, and all $ w \in W $. We say that $ W $ is \emph{$ C_1 $-cofinite} if $ W $ satisfies the condition $ \dim W/C_1(W) < \infty $. 
\end{definition}

\begin{definition}
Let $ W = \coprod_{n \in \mathbb{C}} W_{[n]} $ be a $ \mathbb{C} $-graded vector space. We say that $ W $ is \emph{discretely graded} (also known as \emph{quasi-finite dimensional}) if 
\begin{equation}\label{eq:discrete graded}
\dim \coprod_{\substack{ n \in \mathbb{C} \\ \Re(n) < r}} W_{[n]} < \infty , \qquad \text{ for all } r \in \mathbb{R}.
\end{equation} 
\end{definition}

\begin{remark}
When $ g = \id_V = 1 $, the above definition of $ C_1 $-cofiniteness reduces to the usual notion of $ C_1 $-cofiniteness (see, for example, \cite{HuangDEs}). In this case, there are many well-known examples of $ C_1 $-cofinite discretely graded (untwisted) $ V $-modules. When $ g \neq 1 $, there are non-trivial $ C_1 $-cofinite discretely graded $ g $-twisted $ V $-modules, for example, those constructed in \cite{Huang_twisted_af_modules} for affine vertex operator algebras.
\end{remark}

We now provide the setting for the current section. Let $ \N \in \mathbb{Z}_{>0} $. Let $ W_1 , \dots, W_{\N} $ be untwisted $ V $-modules, and let $ \widetilde{W}_0 , \dots , \widetilde{W}_{\N} $ be $ g $-twisted $ V $-modules. Let $ \mathcal{Y}_i : W_i \otimes \widetilde{W}_i \to \widetilde{W}_{i-1} \{x\} [ \log x ] $ be twisted intertwining operators, for $ i = 1, \dots , \N $. To aid with notation later on, we also denote $ \widetilde{W}_0' $ by $ W_0 $ (a $ g^{-1} $-twisted $ V $-module) and $ \widetilde{W}_{\N} $ by $ W_{\N+1} $. 

We assume that all modules $ W_0, \dots, W_{\N}, \widetilde{W}_0 , \dots , \widetilde{W}_{\N} $ are $ C_1 $-cofinite and discretely graded. (Note that we require the assumption that both $ W_0 $ and $ W_0' = \widetilde{W}_0 $ are $ C_1 $-cofinite and discretely graded.)

\begin{figure}[h]
\caption{A diagram summarizing the untwisted and twisted $ V $-modules, the automorphism twisting each module, and the twisted intertwining operators in the chiral correlation function. All twisted intertwining operators are of \gtype.} \label{fig:diagram}
\centering

\scalebox{0.8}{
\begin{tikzpicture}[scale=0.5]
	\begin{pgfonlayer}{nodelayer}
		\node [style=none] (0) at (-5, 0) {};
		\node [style=none] (1) at (-3, 0) {};
		\node [style=none] (2) at (-2, 0) {};
		\node [style=none] (3) at (0, 0) {};
		\node [style=none] (4) at (1.75, 0) {};
		\node [style=none] (5) at (3.75, 0) {};
		\node [style=none] (6) at (4.75, 0) {};
		\node [style=none] (7) at (6.75, 0) {};
		\node [style=none] (8) at (3.25, 4) {};
		\node [style=none] (9) at (5.25, 4) {};
		\node [style=none] (19) at (2, 6) {};
		\node [style=none] (20) at (0.25, 6.25) {};
		\node [style=none] (21) at (2, 6) {};
		\node [style=none] (22) at (3, 6) {};
		\node [style=none] (23) at (5, 6) {};
		\node [style=none] (24) at (1.5, 10) {};
		\node [style=none] (25) at (3.5, 10) {};
		\node [style=none] (26) at (-1.25, 10) {};
		\node [style=none] (27) at (0.5, 10) {};
		\node [style=none] (28) at (1.5, 10) {};
		\node [style=none] (29) at (3.5, 10) {};
		\node [style=none] (30) at (0, 14) {};
		\node [style=none] (31) at (2, 14) {};
		\node [style=none] (32) at (-4.5, -1.5) {$W_1$};
		\node [style=none] (33) at (-1.5, -1.5) {$W_2$};
		\node [style=none] (34) at (2.25, -1.5) {$W_\N$};
		\node [style=none] (35) at (6.25, -1.5) {$W_{\N+1} = \widetilde{W}_\N$};
		\node [style=none] (36) at (-4.5, -2.75) {$1$};
		\node [style=none] (37) at (-1.5, -2.75) {$1$};
		\node [style=none] (38) at (2.25, -2.75) {$1$};
		\node [style=none] (39) at (6.25, -2.75) {$g$};
		\node [style=none] (40) at (7.25, 4) {$\widetilde{W}_{\N-1}$};
		\node [style=none] (41) at (6.25, 6) {$\widetilde{W}_2$};
		\node [style=none] (42) at (5, 10) {$\widetilde{W}_1$};
		\node [style=none] (43) at (0.25, 15.75) {$W_0' = \widetilde{W}_0$};
		\node [style=none] (44) at (8.75, 4) {$g$};
		\node [style=none] (45) at (7.5, 6) {$g$};
		\node [style=none] (46) at (6.25, 10) {$g$};
		\node [style=none] (47) at (3, 15.75) {$g$};
		\node [style=none] (48) at (1, 0) {$\cdots$};
		\node [style=none] (49) at (4, 5.25) {$\vdots$};
		\node [style=none] (50) at (1, 12.25) {$\mathcal{Y}_1$};
		\node [style=none] (51) at (2.5, 8.25) {$\mathcal{Y}_2$};
		\node [style=none] (52) at (4.25, 2.5) {$\mathcal{Y}_\N$};
	\end{pgfonlayer}
	\begin{pgfonlayer}{edgelayer}
		\draw [style=plainedge, bend right=90, looseness=0.75] (6.center) to (7.center);
		\draw [style=plainedge, bend right=90, looseness=0.75] (4.center) to (5.center);
		\draw [style=plainedge, bend right=90, looseness=0.75] (2.center) to (3.center);
		\draw [style=plainedge, bend right=90, looseness=0.75] (0.center) to (1.center);
		\draw [style=dotted line, bend left=90, looseness=0.75] (6.center) to (7.center);
		\draw [style=dotted line, bend left=90, looseness=0.75] (4.center) to (5.center);
		\draw [style=dotted line, bend left=90, looseness=0.75] (2.center) to (3.center);
		\draw [style=dotted line, bend left=90, looseness=0.75] (0.center) to (1.center);
		\draw [style=plainedge, bend right=90, looseness=0.75] (8.center) to (9.center);
		\draw [style=plainedge, bend left=270, looseness=0.75] (9.center) to (8.center);
		\draw [style=plainedge, in=105, out=-75, looseness=0.50] (9.center) to (7.center);
		\draw [style=plainedge, in=105, out=75, looseness=4.00] (5.center) to (6.center);
		\draw [style=plainedge, in=-90, out=90, looseness=0.25] (4.center) to (8.center);
		\draw [style=plainedge, in=90, out=-105, looseness=0.25] (19.center) to (3.center);
		\draw [style=plainedge, bend right=90, looseness=0.75] (22.center) to (23.center);
		\draw [style=dotted line, bend left=90, looseness=0.75] (22.center) to (23.center);
		\draw [style=plainedge, bend right=90, looseness=0.75] (24.center) to (25.center);
		\draw [style=dotted line, bend left=270, looseness=0.75] (25.center) to (24.center);
		\draw [style=plainedge, in=105, out=-75, looseness=0.75] (25.center) to (23.center);
		\draw [style=plainedge, in=105, out=75, looseness=3.25] (21.center) to (22.center);
		\draw [style=plainedge, in=-90, out=75, looseness=0.50] (20.center) to (24.center);
		\draw [style=plainedge, bend right=90, looseness=0.75] (30.center) to (31.center);
		\draw [style=plainedge, in=105, out=-75, looseness=0.75] (31.center) to (29.center);
		\draw [style=plainedge, in=105, out=75, looseness=3.00] (27.center) to (28.center);
		\draw [style=plainedge, in=-90, out=60, looseness=0.25] (26.center) to (30.center);
		\draw [style=plainedge, in=90, out=-105, looseness=0.25] (20.center) to (2.center);
		\draw [style=plainedge, in=-105, out=90, looseness=0.25] (1.center) to (27.center);
		\draw [style=plainedge, in=-105, out=90, looseness=0.25] (0.center) to (26.center);
		\draw [style=plainedge, bend left=90, looseness=0.75] (30.center) to (31.center);
	\end{pgfonlayer}
\end{tikzpicture}
}
\end{figure}
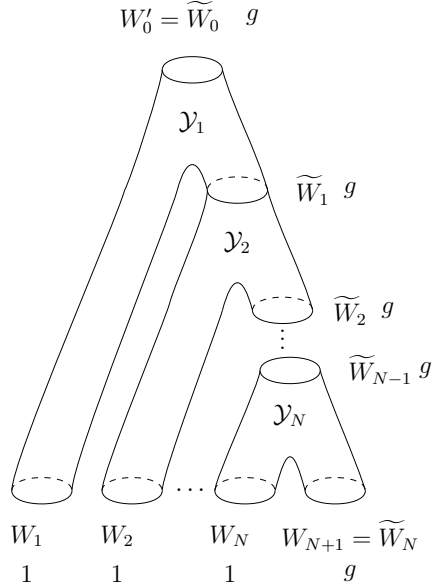

Let $ A $ be the additive subgroup of $ \mathbb{C} $ generated by the union of $ \{1\} $ and the set of eigenvalues for $ \mathcal{S} $. 
Let $ R $ be the subring of
\begin{align*}
\mathbb{C} \{x_1, \dots, x_\N\}[(x_i - x_j)^{-1}:i,j = 1, \dots ,\N , \ i < j] [\log x_1, \dots, \log x_\N]
\end{align*}
generated by $ \{ x_i^\alpha , (x_i - x_j)^{-1} , \log x_i : i,j = 1, \dots ,\N , \ i < j, \  \alpha \in A \} $. We will write the elements of $ R $ in function notation as $ f(x_1, \dots, x_\N) $ with explicit dependence on $ x_1, \dots, x_\N $ only. 

In the case that $ g $ acts semisimply on $ V $, we can remove all of the $ \log x_i $ from $ R $. This occurs in the important case that $ g $ has finite order (see Section~\ref{sec:regular singularities}). 

We have the map
\begin{equation}
\iota_{|x_1| > \dots > |x_\N|}: R \to \mathbb{C} \{x_1, \dots, x_\N\}[\log x_1, \dots, \log x_\N]  
\end{equation}
that expands $ (x_i-x_j)^{-1} $ in non-negative powers of $ x_j $ when $ j > i $. 

Note that $ R $ is Noetherian if $ A $ is finitely generated. For example, if $ g $ has finite order $ t $, then $ A = \frac{1}{t} \mathbb{Z} $. This will be important in Section~\ref{sec:regular singularities}.

In the case that $ V $ is finitely generated, we can project the finitely many generators onto the subspaces $ V_{(n)} $ and consider the finitely many subspaces $ V_{(n)} $ where the image of the projection is non-trivial. The automorphism $ g $ decomposes these $ V_{(n)} $ into eigenspaces for $ \mathcal{S} $. We can then take a basis of eigenvectors for $ \mathcal{S} $ in each of these subspaces. Now we have a finite generating set for $ V $ consisting of eigenvectors for $ \mathcal{S} $. Hence, $ A $ is finitely generated when $ V $ is finitely generated.  

Consider the $ R $-module
\begin{equation}
T = R \otimes W_0 \otimes \dots \otimes W_{\N+1}.
\end{equation}
Since $ T $ is isomorphic as an $ R $-module to $ (R \otimes W_0) \otimes_R \dots \otimes_R (R \otimes W_{\N+1}) $, we will sometimes write elements in $ T $ as 
\begin{align*}
( f_0(x_1,\dots,x_\N) w_0 ) \otimes \dots \otimes (f_{\N+1}(x_1,\dots,x_\N) w_{\N+1} )
\end{align*} 
in place of 
\begin{align*}
f_0(x_1,\dots,x_\N) \cdots f_{\N+1}(x_1,\dots,x_\N) \otimes w_0 \otimes \dots \otimes  w_{\N+1}. 
\end{align*} 

The $ \mathbb{C} $-gradings by conformal weight on $ W_0,\dots , W_{\N+1} $, together with the trivial grading on $ R $, induce a grading on $ T $ called the \emph{weight}. We denote the subspace spanned by all weight-homogeneous elements of weight $ n \in \mathbb{C} $ with $ \Re(n) = r $ by $ T_{(r)} $ so that $ T = \coprod_{r\in \mathbb{R}} T_{(r)} $. Observe that the discrete gradings on $ W_0, \dots, W_{\N+1} $ ensure that $ W_0 \otimes \cdots \otimes W_{\N+1} $ is a discretely graded vector space. Hence, each subspace 
\begin{equation}
F_r(T) = \coprod_{s \in \mathbb{R}_{\leq r}} T_{(s)}
\end{equation} is a finitely generated $ R $-module. Hence, we have a filtration $ F_r(T) \subseteq F_s(T) $, for $ r \leq s $, of finitely generated submodules of $ T $. (The chosen grading provides a filtration of finitely generated submodules, whereas the grading $ \wt (w_0 \otimes \cdots \otimes w_{\N+1} ) = -\wt w_0 + \cdots + \wt w_{\N+1} $ does not.) 

Define the map
\begin{equation}
\begin{gathered}
\phi : T \to \mathbb{C}\{x_1, \dots, x_\N\}[\log x_1, \dots, \log x_\N ], \\
\begin{aligned}
&\phi( f(x_1,\dots,x_\N) \otimes w_0 \otimes \dots \otimes w_{\N+1})  \\
&\qquad = \iota_{|x_1| > \dots > |x_\N|} f(x_1,\dots,x_\N) \langle w_0, \mathcal{Y}_1(w_1,x_1) \cdots \mathcal{Y}_\N(w_\N,x_\N) w_{\N+1} \rangle
\end{aligned}
\end{gathered}
\end{equation}
that evaluates elements in $ T $ as matrix coefficients using $ \mathcal{Y}_1 , \dots, \mathcal{Y}_\N $. We seek to define a submodule $ J $ of $ T $ that lies in the kernel of $ \phi $, satisfying certain properties to be discussed later in Lemma~\ref{lem:grade shift}.

Given $ u \in V^{[\alpha]} $ for $ \alpha \in \mathbb{C} $ with $ \Re(\alpha) \in [0,1) $, we use the notation
\begin{equation}
\alpha' = \begin{cases}
- \alpha &\text{if } \Re(\alpha) = 0, \\
1 - \alpha &\text{if } \Re(\alpha) > 0.
\end{cases}
\end{equation}
This is the eigenvalue of $ \mathcal{S}_{g^{-1}} $ corresponding to $ u $. When clear from context, we also drop the subscripts from $ Y_{W} $ and $ \mathcal{Y}_i $, which can be deduced by looking at the subscripts of their arguments. 

We now define elements that will generate the submodule $ J $ of $ T $. Let $ u \in V^{[\alpha]} $ and $ w_i \in W_i $ for $ i = 0, \dots, \N+1 $. Then for $ \el = 1, \dots, \N $, by \eqref{eq:take u(-1) out Y} and \eqref{eq: Y+/- commutators}, we observe that
\begin{align*}
0 &= \langle w_0, \mathcal{Y}(w_1,x_1) \cdots \Big( Y^+ (u, x_\el)\mathcal{Y}(w_\el, x_\el)  +  \mathcal{Y}(w_\el, x_\el) Y^- (u, x_\el) \\
&\quad - \sum_{k \geq 0} x_\el^{- k} \mathcal{Y}\left( \left( {\mathcal{L} \choose k}   u  \right)_{-1+k} w_\el , x_\el \right) \Big) \cdots \mathcal{Y}(w_\N,x_\N) w_{\N+1} \rangle \\
&= \langle w_0, Y^+ (u, x_\el) \mathcal{Y}(w_1,x_1) \cdots \mathcal{Y}(w_\N,x_\N) w_{\N+1} \rangle \\ 
&+ \sum_{\substack{p=1,\dots,\N \\ p \neq \el}} \langle w_0, \mathcal{Y}(w_1,x_1) \cdots \sum_{j,k \geq 0} :(x_\el - x_p)^{-1-j}:  \\
& \qquad \qquad x_p^{-k} \mathcal{Y}\left( \left( \left(\frac{x_p}{x_\el} \right)^{\mathcal{L}} {\mathcal{L} \choose k}   u\right)_{j+k} w_p, x_p \right)  \cdots \mathcal{Y}(w_\N,x_\N) w_{\N+1} \rangle \\
&-\langle w_0, \mathcal{Y}(w_1,x_1) \cdots \sum_{k \geq 0} x_\el^{- k} \mathcal{Y}\left( \left( {\mathcal{L} \choose k}   u  \right)_{-1+k} w_\el , x_\el \right) \cdots \mathcal{Y}(w_\N,x_\N) w_{\N+1} \rangle \\
&+\langle w_0, \mathcal{Y}(w_1,x_1) \cdots \mathcal{Y}(w_\N,x_\N) Y^- (u, x_\el)w_{\N+1} \rangle ,
\end{align*}
where we use the notation 
\begin{align*}
\colon (x_\el - x_p)^{-1-j} \colon  :=  \iota_{|x_1| > \dots > |x_\N|}(x_\el - x_p)^{-1-j} = \begin{cases}
(x_\el - x_p)^{-1-j} & \text{if } \el < p, \\
(-x_p + x_\el)^{-1-j} & \text{if } p < \el.
\end{cases}
\end{align*}
Hence, we define
\begin{equation}
\begin{aligned}
&\mathcal{A}_\el(u,w_0,\dots, w_{\N+1}) \\
&\quad = -\sum_{j \geq 0} \sum_{k = 0}^K x_\el^{-\alpha + j} (\log x_\el)^k u_{\alpha - 1 - j,k}^* w_0 \otimes w_1 \otimes \dots \otimes w_{\N+1} \\
& \qquad - \sum_{\substack{p=1,\dots,\N \\ p \neq \el}} \sum_{j,k \geq 0} (x_\el - x_p)^{-1-j}  x_p^{-k} w_0 \otimes \dots \otimes \left( \left(\frac{x_p}{x_\el} \right)^{\mathcal{L}} {\mathcal{L} \choose k}   u\right)_{j+k} w_p \otimes \dots \otimes w_{\N+1} \\
& \qquad +\sum_{k \geq 0} x_\el^{- k}  w_0  \otimes \dots \otimes \left( {\mathcal{L} \choose k}   u  \right)_{-1+k} w_\el \otimes \dots \otimes w_{\N+1}  \\
& \qquad - \sum_{j \geq 0} \sum_{k = 0}^K x_\el^{-\alpha - j - 1}(\log x_\el)^k
w_0 \otimes \dots \otimes w_\N \otimes u_{\alpha + j, k} w_{\N+1},
\end{aligned}
\end{equation}
where $ u_{n,k}^* $ denotes the adjoint of an operator $ u_{n,k} $. Observe that the lower truncation property and the discrete grading condition for the modules keep the sums finite, hence $ \mathcal{A}_\el(u,w_0,\dots, w_{\N+1}) $ is indeed an element of $ T $ in the kernel of $ \phi $. By \eqref{eq:u alpha - 1 commutator}, we have
\begin{align*}
&\langle w_0, u_{\alpha-1} \mathcal{Y}(w_1,x_1) \cdots \mathcal{Y}(w_\N,x_\N) 
w_{\N+1} \rangle - \langle w_0, \mathcal{Y}(w_1,x_1) \cdots \mathcal{Y}(w_\N,x_\N) u_{\alpha-1}
w_{\N+1} \rangle \\
&\quad =  \sum_{p=1}^\N \langle w_0, \mathcal{Y}(w_1,x_1) \dots \sum_{k \geq 0} x_p^{-1 - k}  \mathcal{Y} \left(  \left( x_p^\mathcal{L} { \mathcal{L} - 1 \choose k} u \right)_{k} w_p, x_p \right) \cdots \mathcal{Y}(w_\N,x_\N) w_{\N+1} \rangle .
\end{align*}
Hence, we define
\begin{equation}
\begin{aligned}
&\mathcal{A}_{\N+1}(u,w_0,\dots, w_{\N+1}) \\
&\quad = -u_{\alpha - 1}^* w_0 \otimes w_1 \otimes  \dots \otimes w_{\N+1} \\
& \qquad + \sum_{p =1}^\N \sum_{k \geq 0} x_p^{-1 - k}  w_0 \otimes \dots \otimes  \left( x_p^\mathcal{L} { \mathcal{L} - 1 \choose k} u \right)_{k} w_p \otimes \dots \otimes w_{\N+1}\\
&\qquad + w_0 \otimes \dots \otimes w_\N \otimes u_{\alpha-1} w_{\N+1} ,
\end{aligned}
\end{equation}
which is an element of $ T $ in the kernel of $ \phi $. Recall that, by Theorem 6.1 in \cite{HUANG2018346}, we have the $\left(\!\begin{smallmatrix} g^{-1}\\ 1 \  g^{-1} \end{smallmatrix}\!\right)$-type twisted intertwining operators $ A_{\pm} (\mathcal{Y}) $ defined by
\begin{align*}
\langle A_{\pm} (\mathcal{Y})(w_1,x)w_3', w_2 \rangle = \langle w_3', \mathcal{Y}(e^{x L(1)}e^{\pm \pi \iu L(0)}(x^{-L(0)})^2 w_1,x^{-1}) w_2 \rangle,
\end{align*}
which is equivalent to 
\begin{align*}
\langle A_{\pm} (\mathcal{Y})((x^{-L(0)})^{2}e^{\mp \pi \iu L(0)}e^{-x^{-1} L(1)} w_1,x^{-1})w_3', w_2 \rangle = \langle w_3', \mathcal{Y}( w_1,x) w_2 \rangle.
\end{align*}
We define the operators
\begin{equation}
\begin{gathered}
\Op_{\pm}(x), \Op_{\pm}^{-1}(x) : W  \to  W\{x\} [ \log x], \\
\Op_{\pm}(x)w = e^{x L(1)}e^{\pm \pi \iu L(0)}(x^{-L(0)})^2 w, \qquad \Op_{\pm}^{-1}(x)w = (x^{L(0)})^2 e^{\mp \pi \iu L(0)} e^{-x L(1)}w,
\end{gathered}
\end{equation}
so we can briefly write 
\begin{align*}
\langle A_{\pm} (\mathcal{Y})(w_1,x)w_3', w_2 \rangle &= \langle w_3', \mathcal{Y}(\Op_\pm (x) w_1,x^{-1}) w_2 \rangle, \\
\langle A_{\pm} (\mathcal{Y})(\Op_\pm^{-1}(x^{-1}) w_1,x^{-1})w_3', w_2 \rangle &= \langle w_3', \mathcal{Y}( w_1,x) w_2 \rangle.
\end{align*}
By \eqref{eq:u alpha - 1 commutator}, we have
\begin{align*}
[u_{\alpha'-1}, A_\pm (\mathcal{Y})(\Op_{\pm}^{-1}(x_p^{-1})w_p, x_p^{-1}) ] &= \sum_{k \geq 0} x_p^{1 + k}  A_\pm(\mathcal{Y}) \left(  \left( x_p^{-\mathcal{L}_{g^{-1}}} { \mathcal{L}_{g^{-1}} - 1 \choose k} u \right)_{k} \Op_{\pm}^{-1}(x_p^{-1}) w_p, x_p^{-1} \right),
\end{align*}
where $ \mathcal{L}_{g^{-1}} $ is the logarithm of $ g^{-1} $ whose eigenvalues have real parts in $ [0,1) $. So,
\begin{align*}
&\langle u_{\alpha' - 1} w_0, \mathcal{Y}(w_1,x_1) \cdots \mathcal{Y}(w_\N,x_\N) w_{\N+1} \rangle  \\
&= \langle A_{\pm} (\mathcal{Y})(\Op^{-1}_\pm(x_\N^{-1}) w_\N, x_\N^{-1}) \cdots A_{\pm} (\mathcal{Y})(\Op^{-1}_\pm(x_1^{-1}) w_1,x_1^{-1}) u_{\alpha' - 1} w_0, w_{\N+1} \rangle \\
&= \langle u_{\alpha' - 1} A_{\pm} (\mathcal{Y})(\Op^{-1}_\pm(x_\N^{-1}) w_\N, x_\N^{-1}) \cdots A_{\pm} (\mathcal{Y})(\Op^{-1}_\pm(x_1^{-1}) w_1,x_1^{-1})  w_0, w_{\N+1} \rangle \\
&\quad - \sum_{p=1}^\N  \sum_{k \geq 0} x_p^{1 + k} \langle A_{\pm} (\mathcal{Y}_\N)(\Op^{-1}_\pm(x_\N^{-1}) w_\N, x_\N^{-1}) \cdots \\
&\quad  A_\pm(\mathcal{Y}) \left(  \left( x_p^{-\mathcal{L}_{g^{-1}}} { \mathcal{L}_{g^{-1}} - 1 \choose k} u \right)_{k} \Op_{\pm}^{-1}(x_p^{-1}) w_p, x_p^{-1} \right) \cdots A_{\pm} (\mathcal{Y}_1)(\Op^{-1}_\pm(x_1^{-1}) w_1,x_1^{-1}) w_0, w_{\N+1} \rangle \\
&= \langle w_0,  \mathcal{Y}_1(w_1,x_1) \cdots \mathcal{Y}_\N (w_\N,x_\N) u_{\alpha'-1}^* w_{\N+1} \rangle \\
&\quad - \sum_{p=1}^\N \sum_{k \geq 0} x_p^{1 + k} \langle w_0,  \mathcal{Y}_1(w_1,x_1) \cdots \\
&\quad  \mathcal{Y} \left( \Op_\pm(x_p^{-1}) \left( x_p^{-\mathcal{L}_{g^{-1}}} { \mathcal{L}_{g^{-1}} - 1 \choose k} u \right)_{k} \Op_{\pm}^{-1}(x_p^{-1}) w_p, x_p \right) \cdots \mathcal{Y}_\N(w_\N,x_\N) w_{\N+1} \rangle .
\end{align*}
Hence, we define
\begin{equation}
\begin{aligned}
&\mathcal{A}_0(u,w_0, \dots, w_{\N+1}) \\
& \quad = u_{\alpha'-1} w_0 \otimes w_1 \otimes \dots \otimes w_{\N+1} \\
& \qquad + \sum_{p=1}^\N \sum_{k \geq 0} x_p^{1+k} w_0 \otimes \dots \otimes \Op_\pm (x_p^{-1}) \left( x_p^{-\mathcal{L}_{g^{-1}}} {\mathcal{L}_{g^{-1}}-1 \choose k} u\right)_{k} \Op_{\pm}^{-1}(x_p^{-1})w_p \otimes \dots \otimes w_{\N+1} \\
&\qquad - w_0 \otimes \dots \otimes w_\N \otimes u_{\alpha'-1}^* w_{\N+1},
\end{aligned}
\end{equation}
which is an element of $ T $ in the kernel of $ \phi $. 

Define $ J $ to be the $ R $-submodule of $ T $ generated by $ \mathcal{A}_j(u,w_0, \dots,w_{\N+1}) $ for all $ u \in V_{+}^{[\alpha]} $, $ \alpha \in \mathbb{C} $ with $ \Re(\alpha) \in [0,1) $, and all $ w_i \in W_i $, $ i = 0, \dots, \N+1 $, and for all $ j = 0, \dots , \N+1 $. Define $ F_r(J) = F_r(T) \cap J $ so that we have the filtration $ F_r(J) \subseteq F_s(J) $, $ r \leq s $.

The following lemma explains why we have chosen the generators $ \mathcal{A}_{i}(u,w_0, \dots, w_{\N+1} ) $ to  define $ J $.

\begin{lemma}\label{lem:grade shift}
Let $ u \in V^{[\alpha]}_+ $ for $ \alpha \in \mathbb{C} $ with $ \Re(\alpha) \in [0,1) $, and let $ w_i \in W_i $ for $ i = 0, \dots, \N+1 $ be weight homogeneous. 
If $ \Re ( \wt u_{\alpha' - 1} w_0 \otimes \dots \otimes w_{\N+1} ) = s $, then $ u_{\alpha' - 1} w_0 \otimes \dots \otimes w_{\N+1} \in F_s(J) + F_r(T) $ for some $ r < s $. For $ \el = 1, \dots, \N $, if $ \Re( \wt w_0 \otimes \dots \otimes u_{-1} w_\el \otimes \dots \otimes w_{\N+1}) = s $, then $ w_0 \otimes \dots \otimes u_{-1} w_\el \otimes \dots \otimes w_{\N+1}\in F_s(J) + F_r(T) $ for some $ r < s $. If $ \Re( \wt w_0 \otimes \dots \otimes u_{\alpha - 1} w_{\N+1}) = s $, then $  w_0 \otimes \dots \otimes u_{\alpha - 1} w_{\N+1} \in F_s(J) + F_r(T) $ for some $ r < s $. 
\end{lemma}

\begin{proof}
We use $ \sigma $ to denote $ \Re (\wt w_0 + \dots + \wt w_{\N+1}) $. Let
\begin{align*}
s = \Re (\wt u_{\alpha' - 1} w_0 \otimes \dots \otimes w_{\N+1}) &= \wt u - \Re (\alpha') + \sigma.
\end{align*}
Then
\begin{align*}
\Re( \wt  w_0 \otimes \dots \otimes u_{\alpha' - 1}^* w_{\N+1}) &=  - \wt u + \Re(\alpha') + \sigma < \sigma < s , \\
\Re( \wt  w_0 \otimes \dots \otimes u_k w_p \otimes \dots \otimes w_{\N+1}) &= \wt u - k - 1 + \sigma \leq \wt u  -1 + \sigma \\
&< \wt u - \Re (\alpha') + \sigma = s,
\end{align*}
for all $ k \geq 0 $ and $ p = 1, \dots, \N $. Since $ \wt \mathcal{L}_{g^{-1}} = 0 $ and $ \wt L(1) = - 1 $, the real part of the weight of each weight-homogeneous component of $ x_p^{1+k} w_0 \otimes \dots \otimes \Op_\pm (x_p^{-1}) \left( x_p^{-\mathcal{L}_{g^{-1}}} {\mathcal{L}_{g^{-1}}-1 \choose k} u\right)_{k} \Op_{\pm}^{-1}(x_p^{-1})w_p \otimes \dots \otimes w_{\N+1} $ is no more than $ \Re( \wt  w_0 \otimes \cdots  \otimes u_k w_p \otimes \cdots \otimes w_{\N+1}) $, which is less than $ s $. Hence
\begin{align*}
&u_{\alpha' - 1} w_0 \otimes \dots \otimes w_{\N+1} \\
& \quad = \mathcal{A}_0(u,w_0, \dots, w_{\N+1})  \\
& \qquad - \sum_{p=1}^\N \sum_{k \geq 0} x_p^{1+k} w_0 \otimes \dots \otimes \Op_\pm (x_p^{-1}) \left( x_p^{-\mathcal{L}_{g^{-1}}} {\mathcal{L}_{g^{-1}}-1 \choose k} u\right)_{k} \Op_{\pm}^{-1}(x_p^{-1})w_p \otimes \dots \otimes w_{\N+1} \\
& \qquad + w_0 \otimes \dots \otimes u_{\alpha'-1}^* w_{\N+1} \\
& \quad \in F_s(J) + F_r(T) ,
\end{align*}
for some $ r < s $. Now let 
\begin{align*}
s = \Re ( \wt w_0 \otimes \dots \otimes u_{-1} w_\el \otimes \dots \otimes w_{\N+1} ) = \wt u + \sigma .
\end{align*}
Then
\begin{align*}
\Re ( \wt u_{\alpha -1-j,k}^* w_0 \otimes w_1 \otimes \dots \otimes w_{\N+1} ) &= \Re (\alpha) - j - \wt u + \sigma < \sigma < s , \\
\Re ( \wt w_0 \otimes \dots \otimes u_k w_p \otimes \dots \otimes w_{\N+1} ) &= \wt u - k - 1 + \sigma < \wt u + \sigma = s, \\
\Re ( \wt w_0 \otimes \dots \otimes w_\N \otimes u_{\alpha + j,k} w_{\N+1} ) &= \wt u - \Re(\alpha) - j - 1 + \sigma < \wt u + \sigma = s,
\end{align*}
for all $ j, k \geq 0 $ and $ p = 1, \dots, \N $. Similarly to before, since $ \wt \mathcal{L} = 0 $, we have
\begin{align*}
& w_0  \otimes \dots \otimes u_{-1} w_\el \otimes \dots \otimes w_{\N+1}  \\
&\quad =  \mathcal{A}_\el(u,w_0,\dots, w_{\N+1}) \\
&\qquad + \sum_{j \geq 0} \sum_{k = 0}^K x_\el^{-\alpha + j} (\log x_\el)^k u_{\alpha - 1 - j,k}^* w_0 \otimes w_1 \otimes \dots \otimes w_{\N+1} \\
& \qquad + \sum_{\substack{p=1,\dots,\N \\ p \neq \el}} \sum_{j,k \geq 0} (x_\el - x_p)^{-1-j}  x_p^{-k} w_0 \otimes \dots \otimes \left( \left(\frac{x_p}{x_\el} \right)^{\mathcal{L}} {\mathcal{L} \choose k}   u\right)_{j+k} w_p \otimes \dots \otimes w_{\N+1} \\
& \qquad -\sum_{k > 0} x_\el^{- k}  w_0  \otimes \dots \otimes \left( {\mathcal{L} \choose k}   u  \right)_{-1+k} w_\el \otimes \dots \otimes w_{\N+1}  \\
& \qquad + \sum_{j \geq 0} \sum_{k = 0}^K x_\el^{-\alpha - j - 1}(\log x_\el)^k w_0 \otimes \dots \otimes w_\N \otimes u_{\alpha + j,k} w_{\N+1} \\
& \quad \in F_s(J) + F_r(T)
\end{align*}
for some $ r < s $. Now let 
\begin{align*}
s = \Re( \wt w_0 \otimes \dots \otimes w_\N \otimes u_{\alpha - 1} w_{\N+1}) = \wt u - \Re(\alpha) + \sigma.
\end{align*}
Then
\begin{align*}
\Re( \wt u_{\alpha - 1}^* w_0 \otimes w_1 \otimes \dots \otimes w_{\N+1}) &= \Re(\alpha) - \wt u + \sigma < \sigma < s, \\
\Re ( \wt w_0 \otimes \dots \otimes u_k w_p \otimes \dots \otimes w_{\N+1} ) &= \wt u - k - 1 + \sigma < \wt u - \Re (\alpha) + \sigma = s,
\end{align*}
for all $ k \geq 0 $ and $ p = 1, \dots, \N $. Similarly to before, since $ \wt \mathcal{L} = 0 $, we have
\begin{align*}
w_0 \otimes \dots \otimes w_\N \otimes u_{\alpha-1} w_{\N+1} &= \mathcal{A}_{\N+1}(u,w_0,\dots, w_{\N+1}) + u_{\alpha - 1}^* w_0 \otimes w_1 \otimes  \dots \otimes w_{\N+1} \\
& \qquad - \sum_{p =1}^\N \sum_{k \geq 0} x_p^{-1 - k}  w_0 \otimes \dots \otimes  \left( x_p^\mathcal{L} { \mathcal{L} - 1 \choose k} u \right)_{k} w_p \otimes \dots \otimes w_{\N+1}  \\
& \in F_s(J) + F_r(T)
\end{align*}
for some $ r < s $.
\end{proof}

\begin{proposition}\label{prop:T = J + F_M(T)}
There exists $ M \in \mathbb{Z} $ such that for any $ r \in \mathbb{R} $, $ F_r(T) \subseteq F_r(J) + F_M(T) $. Furthermore, $ T = J + F_M(T) $.
\end{proposition}

\begin{proof}
Since $ \dim W_i / C_1(W_i) < \infty $, for $ i = 0,\dots , \N+1 $, there is a finite-dimensional subspace $ U_i \subseteq W_i $ such that $ W_i = U_i + C_1(W_i) $. Hence, we can find $ M_i \in \mathbb{Z} $ sufficiently large such that any element in $ (W_i)_{[m]} $ is necessarily in $  C_1(W_i) $ whenever $ \Re(m) > M_i $. Taking $ M = \sum_{i = 0}^{\N+1} M_i $, we see that any weight-homogeneous element in $ T $ that has real part of its weight larger than $ M $ must be in
\begin{align*}
\sum_{i = 0}^{\N+1} R \otimes W_0 \otimes \dots \otimes C_1(W_i) \otimes \dots \otimes W_{\N+1}.
\end{align*}
Hence, we have $ M \in \mathbb{Z} $ such that
\begin{equation}\label{eq:inclusion}
\begin{aligned}
\coprod_{m > M} T_{(m)} &\subseteq \sum_{i = 0}^{\N+1} R \otimes W_0 \otimes \dots \otimes C_1(W_i) \otimes \dots \otimes W_{\N+1}.
\end{aligned}
\end{equation}

We immediately have $ F_r(T) \subseteq F_M(T) \subseteq F_r(J) + F_M(T) $ for all $ r \leq M $. Since $ T $ is discretely graded, we can use induction on $ r \geq M $. Assume that we have some $ s > M $ such that $ F_r(T) \subseteq F_r(J) + F_M(T) $ for all $ r < s $. 

We want to show that $ T_{(s)} \subseteq F_s(J) + F_M(T) $. Since $ T_{(s)} \subseteq \coprod_{m > M} T_{(m)} $, we can consider elements in the summands of \eqref{eq:inclusion} without loss of generality. By Lemma~\ref{lem:grade shift}, these elements are in $  F_s(J) + F_r(T) $ for some $ r < s $. By the induction hypothesis, these elements are then in $ F_s(J) + F_r(T) \subseteq F_s(J) + F_M(T) $, and we have proved the first statement.

Furthermore, we have
\begin{align*}
T = \bigcup_{r \in \mathbb{R}} F_r(T) \subseteq \bigcup_{r \in \mathbb{R}} \left( F_r(J) + F_M(T) \right) \subseteq \bigcup_{r \in \mathbb{R}} F_r(J) + F_M(T) = J + F_M(T).
\end{align*}
Also, $ J + F_M(T) \subseteq T $, so we have $ T = J + F_M(T) $.
\end{proof}

\begin{corollary}
The $ R $-module $ T/J $ is finitely generated.
\end{corollary}

\begin{proof}
Since $ T = J + F_M(T) $, the image of $ F_M(T) $ in $ T/J $ is all of $ T/J $. Since $ F_M(T) $ is finitely generated, so is $ T/J $.
\end{proof}

Consider a change of variables from $ (z_1, \dots, z_\N) $ to $ (\zeta_1, \dots, \zeta_\N) $ given by
\begin{equation}
z_j = \sum_{i=1}^\N b_{ij} \zeta_i + \gamma_j,
\end{equation}
for $ B = (b_{ij})_{i,j = 1}^\N \in \GL(\N, \mathbb{C}) $ and $ (\gamma_1, \dots, \gamma_\N) \in \mathbb{C}^\N $.
Then 
\begin{equation}
\frac{\partial}{\partial \zeta_i} = \sum_{j = 1}^\N \frac{\partial z_j}{\partial \zeta_i} \frac{\partial}{\partial z_j} = \sum_{j = 1}^\N b_{ij} \frac{\partial}{\partial z_j} .
\end{equation}
Define the operators $ \widehat{L}_{\zeta_\el} $ on $ W_0 \otimes \cdots \otimes W_{\N+1} $ by 
\begin{equation}
\widehat{L}_{\zeta_\el} (w_0 \otimes \cdots \otimes w_{\N+1}) = \sum_{j = 1}^\N b_{\el j} w_0 \otimes \cdots \otimes L(-1) w_j \otimes \cdots \otimes w_{\N+1} ,
\end{equation}
for each $ \el = 1, \dots, \N $.

For $ w \in T $, we use $ [w] $ to denote the equivalence class $ w + J $ in $ T/J $. 
\begin{corollary}
Assume that $ R $ is Noetherian. For any $ w_i \in W_i $, $ i = 0,\dots, \N+1 $, and for any $ \el = 1, \dots, \N $, there exist $ m_\el \in \mathbb{Z}_{> 0} $ and $ a_{k,\el} (x_1, \dots, x_\N) \in R $ for $ k = 1, \dots, m_\el $ such that
\begin{equation}\label{eq:L(-1)-eq}
\begin{aligned}
&[\widehat{L}_{\zeta_\el}^{m_\el} (w_0 \otimes \dots \otimes w_{\N+1})] + a_{1,\el} (x_1, \dots, x_\N) [\widehat{L}_{\zeta_\el}^{m_\el-1} (w_0 \otimes \dots \otimes w_{\N+1})] + \\
&\cdots + a_{m_\el-1,\el} (x_1, \dots, x_\N)[\widehat{L}_{\zeta_\el} (w_0 \otimes \dots \otimes w_{\N+1})] + a_{m_\el,\el} (x_1, \dots, x_\N)[w_0 \otimes \dots \otimes w_{\N+1}] = 0.
\end{aligned}
\end{equation}
\end{corollary}

\begin{proof}
Consider an ascending chain $ (M_j)_{j = 0}^\infty $ of submodules $ M_j $ of $ T/J $ generated by 
\begin{equation}
\{ [\widehat{L}_{\zeta_\el}^k (w_0 \otimes \dots \otimes w_{\N+1})] : k \in \mathbb{Z}, \ 0 \leq k \leq j \} .
\end{equation} 
Since $ R $ is a Noetherian ring and $ T/J $ is a finitely generated $ R $-module, this chain stabilizes. So, there is some $ m_\el \in \mathbb{Z}_{> 0} $ such that $ M_{m_\el} = M_{m_\el-1} $. Hence, $ [\widehat{L}_{\zeta_\el}^{m_\el} (w_0 \otimes \dots \otimes w_{\N+1})] \in M_{m_\el-1} $, and we obtain our desired equation.
\end{proof}

\begin{theorem}\label{thm:formal DEs general}
Assume that $ R $ is Noetherian. Then there exists a system of differential equations
\begin{equation}\label{eq:DE}
\begin{aligned}
& \left(\frac{\partial}{\partial \zeta_\el} \right)^{m_\el} \psi + \sum_{k = 1}^{m_\el} a_{k,\el}(z_1, \dots, z_\N)  \left(\frac{\partial}{\partial \zeta_\el} \right)^{m_\el-k} \psi  = 0, \qquad \el = 1, \dots, \N , 
\end{aligned} 
\end{equation}
satisfied by the series
\begin{align*}
\langle w_0, \mathcal{Y}_1(w_1,z_1) \cdots \mathcal{Y}_\N(w_\N,z_\N) w_{\N+1} \rangle
\end{align*}
in the region $ |z_1| > \dots > | z_\N | > 0 $.
\end{theorem}

\begin{proof}
Since $ J $ is contained in the kernel of $ \phi $, we have  
\begin{align*}
\overline{\phi} : T/J \to \mathbb{C}\{x_1, \dots, x_\N \}[\log x_1, \dots, \log x_\N].
\end{align*}
Using this map on \eqref{eq:L(-1)-eq} followed by the $ L(-1) $-derivative property, we obtain
\begin{equation} 
\begin{aligned}
& \left(\sum_{j=1}^\N b_{\el j} \frac{\partial}{\partial x_j} \right)^{m_\el} \langle w_0, \mathcal{Y}_1(w_1,x_1) \cdots \mathcal{Y}_\N(w_\N,x_\N) w_{\N+1} \rangle \\
&\quad + \sum_{k = 1}^{m_\el} \iota_{|x_1| > \dots > |x_\N|} a_{k,\el}(x_1, \dots, x_\N)  \left(\sum_{j=1}^\N b_{\el j} \frac{\partial}{\partial x_j} \right)^{m_\el-k} \langle w_0, \mathcal{Y}_1(w_1,x_1) \cdots \mathcal{Y}_\N(w_\N,x_\N) w_{\N+1} \rangle  = 0.
\end{aligned} 
\end{equation}
Evaluating $ x_1 = z_1 , \dots, x_\N = z_\N $ gives the desired differential equations.
\end{proof}

\begin{remark}
Recall that $ R $ is Noetherian when $ V $ is finitely generated, for which there are many natural examples. Furthermore, $ R $ is Noetherian when the order of $ g $ is finite, which is an assumption in Section~\ref{sec:regular singularities}.
\end{remark}

\begin{remark}
We have shown that the series $ \langle w_0, \mathcal{Y}_1(w_1,z_1) \cdots \mathcal{Y}_\N(w_\N,z_\N) w_{\N+1} \rangle $ formally satisfies \eqref{eq:DE}, but this does not show that the series converges. Since there are non-integral powers and logarithms of $ z_i $, we cannot use the usual power series convergence for analytic systems (even if we reduce to ordinary differential equations in $ z_i $). To show convergence, we instead use the theory of differential equations with regular singular points in Section~\ref{sec:regular singularities} while assuming that $ g $ has finite order. We do not have a method for showing the convergence when $ g $ is a general automorphism.
\end{remark}

\section{The twisted KZ equations}
\label{sec:twisted KZ}
In this section, we work with the affine vertex operator algebras $ V_{\hat{\mathfrak{g}}}(\lvl,0) $ as an explicit family of vertex operator algebras. A slightly different method is used than in the previous section to derive a system of differential equations when $ w_0, \dots , w_{\N+1} $ are lowest weight vectors. When the modules are untwisted, this system is well known in the physics literature as ``the KZ equations" \cite{KZ}, named after Knizhnik and Zamolodchikov. When $ W_0' $ and $ W_{\N+1} $ are $ g $-twisted (by a semisimple inner automorphism of $ \mathfrak{g} $), this system is known in the physics literature as the ``\textit{twisted} KZ equations" (for ``inner-automorphic WZW orbifolds") \cite{deBoer:2001nw}. Observe that under the condition that $ W_1, \dots, W_\N $ are untwisted and $ W_0' $ and $ W_{\N+1} $ are $ g $-twisted, the twisted intertwining operators are necessarily of \gtype. This is an important assumption used by de Boer, Halpern and Obers to derive the twisted KZ equations. We keep the assumption of \gtype\ intertwining operators, but we will generalize the twisted KZ equations so that $ g $ can be an arbitrary automorphism of the vertex operator algebra $ V_{\hat{\mathfrak{g}}}(\lvl,0) $. In the following section, we will show that the solutions to the twisted KZ equations have ``regular singularities" when $ g $ has finite order. 

Let $ \mathfrak{g} $ be a finite-dimensional simple Lie algebra with an invariant symmetric bilinear form $ ( \cdot, \cdot) $. Let $ g $ be an automorphism of $ \mathfrak{g} $. Let $ \mathcal{L} $, $ \mathcal{S} $ and $ \mathcal{N} $ be operators on $ \mathfrak{g} $ such that $ g = e^{2 \pi \iu \mathcal{L} }$, $ \mathcal{S} $ and $ \mathcal{N} $ are the semisimple and nilpotent parts of $ \mathcal{L} $, and the eigenvalues of $ \mathcal{S} $ have real part in $ [0,1) $. 

As defined in \cite{Huang_twisted_af_modules} (see also \cite{bakalovtwist}), the \emph{twisted affine Lie algebra} $ \hat{\mathfrak{g}}^{[g]} $ is spanned by the elements 
\begin{equation}
\mathbf{k}, \ a \otimes t^{n} = a_n, \quad \text{for all } a \in \mathfrak{g} \text{ and } n \in \mathbb{C} \text{ such that } e^{2 \pi \iu \mathcal{S}} a = e^{2 \pi \iu n} a,
\end{equation} 
with Lie bracket given by 
\begin{align}
[a_m, b_n] &= [a,b]_{m + n} + ((m+\mathcal{N})a,b)\delta_{m+n,0} \mathbf{k} , \\
[\mathbf{k},a_m] &= 0.
\end{align}
We have the subspaces
\begin{gather*}
\hat{\mathfrak{g}}^{[g]}_+ := \Span\{a_n \in \hat{\mathfrak{g}}^{[g]} : \Re(n) > 0 \}, \quad \hat{\mathfrak{g}}^{[g]}_- := \Span\{a_n \in \hat{\mathfrak{g}}^{[g]} : \Re(n) < 0 \}, \\
\hat{\mathfrak{g}}^{[g]}_\mathbb{I} := \Span\{a_n \in \hat{\mathfrak{g}}^{[g]} : \Re(n) = 0 \}, \quad \hat{\mathfrak{g}}^{[g]}_0 := \hat{\mathfrak{g}}^{[g]}_\mathbb{I} \oplus \mathbb{C} \mathbf{k}.
\end{gather*}
Furthermore, $ \hat{\mathfrak{g}}^{[g]}_+ $, $ \hat{\mathfrak{g}}^{[g]}_- $ and $ \hat{\mathfrak{g}}^{[g]}_0 $ are subalgebras of $ \hat{\mathfrak{g}}^{[g]} $, giving a triangular decomposition
\begin{equation}
\hat{\mathfrak{g}}^{[g]} = \hat{\mathfrak{g}}^{[g]}_+ \oplus \hat{\mathfrak{g}}^{[g]}_0 \oplus \hat{\mathfrak{g}}^{[g]}_- .
\end{equation}
Note that when $ g = 1 $, we obtain the affine Lie algebra $ \hat{\mathfrak{g}}^{[1]} = \hat{\mathfrak{g}} $.

The vertex operator algebra that we consider here is the universal affine vertex operator algebra $ V_{\hat{\mathfrak{g}}}(\lvl,0) = U(\hat{\mathfrak{g}}) \otimes_{U(\hat{\mathfrak{g}}_0 \oplus \hat{\mathfrak{g}}_+)} \mathbb{C}_\lvl $ at a fixed level $ \lvl \in \mathbb{C} \backslash \{-h^\vee\} $, where $ h^\vee $ is the dual Coxeter number of $ \mathfrak{g} $ (see, for example, \cite{LepowskyLi}). More generally, we could take $ V $ to be any quotient of $  V_{\hat{\mathfrak{g}}}(\lvl,0) $. This vertex operator algebra is a representation of $ \hat{\mathfrak{g}} $, and we denote the action of $ a_n $ on $ V $ by $ a(n) $. The automorphism $ g $ of $ \mathfrak{g} $ induces automorphisms (which we still denote by $ g $)
\begin{gather}
g \mathbf{k} := \mathbf{k}, \quad g (a_m) := (g a)_m, \quad \text{for all } a\in \mathfrak{g}, \ m \in \mathbb{Z}; \text{ and}
\\
g (a^{(1)}(n_1) \cdots a^{(\el)}(n_\el) \one ):= (ga^{(1)})(n_1) \cdots (ga^{(\el)})(n_\el) \one, \quad \text{for all } a^{(1)}, \dots,  a^{(\el)}\in \mathfrak{g}, \ \el \in \mathbb{Z}_{\geq 0} ,
\end{gather}
of $ \hat{\mathfrak{g}} $ and $ V_{\hat{\mathfrak{g}}}(\lvl,0) $, respectively. Since $ g $ preserves the conformal vector $ \omega $ of $ V_{\hat{\mathfrak{g}}}(\lvl,0) $, $ g $ is furthermore an automorphism of $ V_{\hat{\mathfrak{g}}}(\lvl,0) $ viewed as a vertex operator algebra. The $ g $-twisted $  V_{\hat{\mathfrak{g}}}(\lvl,0) $-modules are naturally representations of $ \hat{\mathfrak{g}}^{[g]} $ (see \cite{Huang_twisted_af_modules}), and we denote the action of $ a_n $ on the twisted modules by $ a(n) $. 

Let $ \{ a^i \}_{i \in I} $ be a basis for $ \mathfrak{g} $ consisting of generalized eigenvectors of $ g $, i.e., eigenvectors of $ \mathcal{S} $. Since $ (\cdot, \cdot) $ is non-degenerate, there is a basis $ \{{a^i}'\}_{i\in I} $ dual to $ \{ a^i \}_{i \in I} $, with respect to $ (\cdot, \cdot) $. The operator $ e^{2\pi \iu \mathcal{S}} $ is also an automorphism of $ \mathfrak{g} $. Since invariant symmetric bilinear forms of $ \mathfrak{g} $ are invariant under automorphisms, we have
\begin{equation}
(e^{2 \pi \iu \mathcal{S}}{a^i}', a^j) = ( {a^i}', e^{-2 \pi \iu \mathcal{S}} a^j) = ({a^i}', e^{-2\pi \iu \alpha^{j}}{a^j}) =  e^{-2\pi \iu \alpha^j} \delta_{ij},
\end{equation}
where $ \mathcal{S} a^j = \alpha^j a^j $. Hence, $ \{{a^i}'\}_{i\in I} $ are also generalized eigenvectors of $ g $. 

We will frequently work with expressions that are symmetric in $ a^i $ and $ {a^i}' $. To reduce the amount of explicit terms in these expressions, we make the following definitions. Choose a new index set $ I' $, disjoint from $ I $, for the dual basis $ \{a^i \}_{i\in I'} = \{{a^i}'\}_{i \in I} $. Define $ \mathcal{I} = I \sqcup I' $. Since $ \{a^i\}_{i \in I} $ is also dual to $ \{a^i \}_{i\in I'} $, we can make simplifications such as
\begin{align*}
\sum_{i \in I} \left( {a^{i}}'(-1) a^{i}(0) +  a^{i}(-1){a^{i}}'(0) \right)w = \sum_{i \in \mathcal{I}} a^{i}(-1){a^{i}}'(0)w.
\end{align*}
Note that we can further define $ \cdot ': \mathcal{I} \to \mathcal{I} $ by $ {a^{i'}} = (a^i)' $, so there is no ambiguity in the notation. For each $ i \in \mathcal I $, we denote the $ \mathcal{S} $-eigenvalue of $ a^i $ by $ \alpha^i $.

In this section, we will work with the same assumptions on the $ V $-modules as in the previous section (summarized in Figure~\ref{fig:diagram}) with $ V = V_{\hat{\mathfrak{g}}}(\lvl,0) $. Let $ w_i \in W_i $, $ i = 1, \dots, \N $, $ w_0 \in {W}_0 $ and $ w_{\N+1} \in W_{\N+1} $ be lowest weight vectors, equivalently annihilated by $ \hat{\mathfrak{g}}_+ $, $ \hat{\mathfrak{g}}^{[g^{-1}]}_+ $ and $ \hat{\mathfrak{g}}^{[g]}_+ $, respectively. Using these assumptions, we will construct a first-order homogeneous linear system of partial differential equations satisfied by the formal series
\begin{equation}
\langle w_0, \mathcal{Y}_1(w_1,x_1) \cdots \mathcal{Y}_\N(w_\N, x_\N) w_{\N+1} \rangle
\end{equation}
by computing the right-hand side of
\begin{equation}
\begin{aligned}
&\frac{\partial}{\partial x_\el} \langle w_0, \mathcal{Y}_1(w_1,x_1) \cdots \mathcal{Y}_\N(w_\N, x_\N) w_{\N+1} \rangle \\
& \qquad = \langle w_0, \mathcal{Y}_1(w_1,x_1) \cdots \mathcal{Y}_\el(L(-1)w_\el, x_\el) \cdots \mathcal{Y}_\N(w_\N, x_\N) w_{\N+1} \rangle,
\end{aligned}
\end{equation}
for each fixed $ \el = 1, \dots, \N $. Since $ W_\el $ is an untwisted module,   the $ L(-1) $-action on $ W_\el $ is given by
\begin{align*}
 L(-1) = \sum_{i \in I} \sum_{j \in \mathbb{Z}_{ > 0}} \frac{1}{2(\lvl + h^\vee)} {a^i}'(-j) a^i(j-1) + \sum_{i \in I} \sum_{j \in \mathbb{Z}_{\leq 0}} \frac{1}{2(\lvl + h^\vee)} a^i(j-1){a^i}'(-j) .
\end{align*}
(Note that here we use a basis $ \{ a^i\}_{i\in I} $ of generalized eigenvectors of $ g $. Although this is not needed for the untwisted module $ W_\el $, it will later be used for the $ g $-twisted modules.) Utilizing that $ w_\el $ is a lowest weight vector, we obtain the relation
\begin{align*}
 \mathcal{Y}_\el(L(-1)w_\el, x_\el) &= \mathcal{Y}_\el \left( \sum_{i \in I} \frac{1}{2(\lvl + h^\vee)} {a^i}'(-1) a^i(0) w_\el + \sum_{i \in I} \frac{1}{2(\lvl + h^\vee)} a^i(-1){a^i}'(0) w_\el, x_\el \right) ,
\end{align*} 
 or rearranged as
\begin{align*} 
2(\lvl + h^\vee) \mathcal{Y}_\el(L(-1)w_\el, x_\el) &= \sum_{i \in I} \mathcal{Y}_\el \left(  {a^i}'(-1) a^i(0) w_\el +  a^i(-1){a^i}'(0) w_\el, x_\el \right) \\
&= \sum_{i \in \mathcal{I}} \mathcal{Y}_\el (  a^i(-1){a^i}'(0) w_\el, x_\el ).
\end{align*}

The vectors $ {a^i} = {a^i}(-1) \one $ and $ {a^i}'(0) w_\el $ satisfy the assumptions on $ u $ and $ w_1 $, respectively, used to derive  \eqref{eq:take u(-1) out Y - KZ}. Hence, we can use  \eqref{eq:take u(-1) out Y - KZ} to obtain
\begin{equation}\label{eq:L(-1) to normal order}
\begin{aligned} 
&2(\lvl + h^\vee) \mathcal{Y}_\el(L(-1)w_\el, x_\el) = \sum_{i \in \mathcal{I}} \mathcal{Y}_\el ( a^i(-1){a^i}'(0) w_\el, x_\el ) \\
&= \sum_{i \in \mathcal{I}} \left( Y^+ (a^i, x_\el)\mathcal{Y}_\el({a^i}'(0) w_\el, x_\el)  +  \mathcal{Y}_\el({a^i}'(0) w_\el, x_\el) Y^- (a^i, x_\el) - x_\el^{-1} \mathcal{Y}_\el( (\mathcal{L} a^i)(0) {a^i}'(0) w_\el, x_\el ) \right). 
\end{aligned}
\end{equation}

Assume $ u \in \left( V_{\hat{\mathfrak{g}}}(\lvl,0) \right)_{(1)} \cong \mathfrak{g} $ and $ \alpha \in \mathbb{C} $ with $ \mathcal{S} u = \alpha u $. Let $ \widetilde{w}_0 \in \widetilde{W}_0 $  and $ f \in \Hom (W_{\N+1}, \widetilde{W}_0 ) $. By use of $ Y^+(u,x) = Y_0^+(x^{-\mathcal{N}} u,x) $, the definition of the contragredient module action 
\begin{equation}
\langle Y'(v,x) w', w \rangle := \langle w', Y(e^{x L(1)} (-x^{-2})^{L(0)} v, x^{-1} ) w \rangle,
\end{equation}
and the fact that $ L(1) \mathfrak{g} = 0 $, we have
\begin{align} \label{eq:out-state annihilate}
\langle w_0, Y^+ (u, x) \widetilde{w}_0 \rangle = \sum_{k=0}^K \sum_{m \in -\alpha + \mathbb{Z}_{> 0}} \frac{(-1)^{k+1}}{k!}\langle (\mathcal{N}^k u)(m)w_0,  \widetilde{w}_0 \rangle x^{m-1} (\log x)^k = 0 .
\end{align}
By use of  $ Y^-(u,x) = Y_0^-(x^{-\mathcal{N}} u,x) $, we have
\begin{align}
\langle w_0, f Y^- (u, x) w_{\N+1} \rangle &= \sum_{k=0}^K  \sum_{m \in \alpha + \mathbb{Z}_{\geq 0}} \frac{(-1)^k}{k!} \langle w_0, f (\mathcal{N}^k u)(m)  w_{\N+1} \rangle x^{-m-1} (\log x)^k \nonumber \\
&= \sum_{k=0}^K \frac{(-1)^k}{k!} \langle w_0, f (\mathcal{N}^k u)(\alpha)  w_{\N+1} \rangle x^{-\alpha-1} (\log x)^k \nonumber \\
&= \langle w_0, f (x^{-\mathcal{N}} u)(\alpha) w_{\N+1} \rangle x^{-\alpha-1} \label{eq:in-state}.
\end{align}
If $ \Re(\alpha) = 0 $, then $ ({\mathcal{N}}^k u)(\alpha) w_{\N+1} $ is still of lowest weight. If $ \Re(\alpha) > 0 $, we further have $ (x^{-\mathcal{N}} u)(\alpha) w_{\N+1} = 0 $. 

We wish to remove all instances of $ Y^+ (a^i, x_\el) $ and  $ Y^- (a^i, x_\el) $ from 
\begin{equation*}
\frac{\partial}{\partial x_\el} \langle w_0, \mathcal{Y}_1(w_1,x_1) \cdots \mathcal{Y}_\N(w_\N, x_\N) w_{\N+1} \rangle .
\end{equation*} 
To do so, we can use  \eqref{eq: Y+/- commutators - KZ} to move them as far left and right as possible, respectively. 
For $ p < \el $, we have 
\begin{equation} \label{eq:p<l}
\mathcal{Y}_p(w_p, x_p) Y^+(a^i,x_\el) = Y^+(a^i,x_\el) \mathcal{Y}_p(w_p, x_p)  + (-x_p + x_\el)^{-1} \mathcal{Y}_p \left(( (x_p/x_\el )^{\mathcal{L}}a^i)(0) w_p, x_p \right). 
\end{equation}
For $ \el < p $, we have 
\begin{equation}\label{eq:l<p}
Y^-(a^i,x_\el) \mathcal{Y}_p(w_p, x_p)  =  \mathcal{Y}_p(w_p, x_p) Y^-(a^i,x_\el) + (x_\el - x_p)^{-1} \mathcal{Y}_p \left(( (x_p/x_\el )^{\mathcal{L}}a^i)(0) w_p, x_p \right).
\end{equation}
To work with these two relations on equal footing, we use the following notation
\begin{align*}
\colon (x_\el - x_p)^{-1} \colon  := \iota_{|x_1| > \dots > |x_\N|}(x_\el - x_p)^{-1} =  \begin{cases}
(x_\el - x_p)^{-1} & \text{when } \el < p, \\
(-x_p + x_\el)^{-1} & \text{when } p < \el.
\end{cases}
\end{align*}
We can now derive the \emph{twisted KZ equations} (first discovered in \cite{deBoer:2001nw}).
\begin{proposition}[The twisted KZ equations]
If the assumptions at the start of this section hold, then we have
\begin{equation}\label{eq:KZ equation}
\begin{aligned}
&2(\lvl + h^\vee) \frac{\partial}{\partial x_\el} \langle w_0, \mathcal{Y}_1(w_1,x_1) \cdots \mathcal{Y}_\N(w_\N, x_\N) w_{\N+1} \rangle \\
&= \sum_{i \in \mathcal{I}}  \Bigg( \sum_{ p \neq \el}  \sum_{k_1,k_2 = 0}^{K}  \frac{(-1)^{k_2}}{k_1!k_2!} (\log x_p)^{k_1}(\log x_\el)^{k_2} \left( \frac{x_p}{x_\el}
\right)^{\alpha^i} :(x_\el - x_p)^{-1} :  \langle w_0, \mathcal{Y}_1(w_1,x_1)  \\
& \quad \cdots \mathcal{Y}_\el({a^i}'(0)w_\el, x_\el) \cdots  \mathcal{Y}_p(( \mathcal{N}^{k_1 + k_2}a^i)(0)w_p, x_p) \cdots \mathcal{Y}_\N(w_\N, x_\N) w_{\N+1} \rangle \\
& \quad + \sum_{k = 0}^{K} \frac{(-1)^k}{k!} (\log x_\el)^k x_\el^{-\alpha^i -1} \langle w_0, \mathcal{Y}_1(w_1,x_1) \cdots \mathcal{Y}_\el({a^i}'(0)w_\el, x_\el) \cdots \mathcal{Y}_\N(w_\N, x_\N) (\mathcal{N}^k a^i)(\alpha^i) w_{\N+1} \rangle  \\
& \quad  - x_\el^{-1} \langle w_0, \mathcal{Y}_1(w_1,x_1) \cdots   \mathcal{Y}_\el( (\mathcal{L}a^i)(0) {a^i}'(0) w_\el, x_\el) \cdots \mathcal{Y}_\N(w_\N, x_\N) w_{\N+1} \rangle \Bigg),
\end{aligned}
\end{equation}
for each $ \el = 1, \dots, \N $, where $ K $ is chosen such that $ \mathcal{N}^{K+1} \mathfrak{g} = 0 $.
\end{proposition}
\begin{proof} 
We use \eqref{eq:L(-1) to normal order}, followed by repeated use of \eqref{eq:p<l} and \eqref{eq:l<p}, and then \eqref{eq:out-state annihilate} and \eqref{eq:in-state} to obtain
\begin{align*}
&2(\lvl + h^\vee) \frac{\partial}{\partial x_\el} \langle w_0, \mathcal{Y}_1(w_1,x_1) \cdots \mathcal{Y}_\N(w_\N, x_\N) w_{\N+1} \rangle \\
&= 2(\lvl + h^\vee) \langle w_0, \mathcal{Y}_1(w_1,x_1) \cdots \mathcal{Y}_\el(L(-1)w_\el, x_\el) \cdots \mathcal{Y}_\N(w_\N, x_\N) w_{\N+1} \rangle \\
&= \sum_{i \in \mathcal{I}}  \Big( \langle w_0, \mathcal{Y}_1(w_1,x_1) \cdots Y^+ (a^i, x_\el)\mathcal{Y}_\el({a^i}'(0) w_\el, x_\el)  \cdots \mathcal{Y}_\N(w_\N, x_\N) w_{\N+1} \rangle  \\
& \qquad + \langle w_0, \mathcal{Y}_1(w_1,x_1) \cdots  \mathcal{Y}_\el({a^i}'(0) w_\el, x_\el) Y^- (a^i, x_\el) \cdots \mathcal{Y}_\N(w_\N, x_\N) w_{\N+1} \rangle \\ 
& \qquad - x_\el^{-1}\langle w_0, \mathcal{Y}_1(w_1,x_1) \cdots \mathcal{Y}_\el( (\mathcal{L}a^i )(0) {a^i}'(0) w_\el, x_\el ) \cdots \mathcal{Y}_\N(w_\N, x_\N) w_{\N+1} \rangle \Big) \\
&= \sum_{i \in \mathcal{I}}  \Bigg(\sum_{p \neq \el} :(x_\el - x_p)^{-1} : \langle w_0, \mathcal{Y}_1(w_1,x_1) \cdots \mathcal{Y}_\el({a^i}'(0)w_\el, x_\el) \\
&\quad \cdots  \mathcal{Y}_p(( (x_p/x_\el)^{\mathcal{L}}a^i)(0)w_p, x_p) \cdots \mathcal{Y}_\N(w_\N, x_\N) w_{\N+1} \rangle \\
& \quad + x_\el^{-\alpha^i -1} \langle w_0, \mathcal{Y}_1(w_1,x_1) \cdots \mathcal{Y}_\el({a^i}'(0)w_\el, x_\el) \cdots \mathcal{Y}_\N(w_\N, x_\N) (x_\el^{-\mathcal{N}}a^i)(\alpha^i) w_{\N+1} \rangle  \\
&\quad  - x_\el^{-1} \langle w_0, \mathcal{Y}_1(w_1,x_1) \cdots   \mathcal{Y}_\el( (\mathcal{L}a^i)(0) {a^i}'(0) w_\el, x_\el) \cdots \mathcal{Y}_\N(w_\N, x_\N) w_{\N+1} \rangle \Bigg) \\
&= \sum_{i \in \mathcal{I}}  \Bigg( \sum_{ p \neq \el}  \sum_{k_1,k_2 = 0}^{K}  \frac{(-1)^{k_2}}{k_1!k_2!} (\log x_p)^{k_1}(\log x_\el)^{k_2} \left( \frac{x_p}{x_\el}
\right)^{\alpha^i} :(x_\el - x_p)^{-1} :  \langle w_0, \mathcal{Y}_1(w_1,x_1)  \\
& \quad \cdots \mathcal{Y}_\el({a^i}'(0)w_\el, x_\el) \cdots  \mathcal{Y}_p(( \mathcal{N}^{k_1 + k_2}a^i)(0)w_p, x_p) \cdots \mathcal{Y}_\N(w_\N, x_\N) w_{\N+1} \rangle \\
& \quad + \sum_{k = 0}^{K} \frac{(-1)^k}{k!} (\log x_\el)^k x_\el^{-\alpha^i -1} \langle w_0, \mathcal{Y}_1(w_1,x_1) \cdots \mathcal{Y}_\el({a^i}'(0) w_\el, x_\el) \cdots \mathcal{Y}_\N(w_\N, x_\N) (\mathcal{N}^k a^i)(\alpha^i) w_{\N+1} \rangle  \\
& \quad  - x_\el^{-1} \langle w_0, \mathcal{Y}_1(w_1,x_1) \cdots   \mathcal{Y}_\el( (\mathcal{L}a^i)(0) {a^i}'(0) w_\el, x_\el) \cdots \mathcal{Y}_\N(w_\N, x_\N) w_{\N+1} \rangle \Bigg),
\end{align*}
where $ K \in \mathbb{Z}_{\geq 0} $ is chosen sufficiently large so that $ \mathcal{N}^{K + 1} \mathfrak{g} = 0 $.
\end{proof}

This is our most general version of the twisted KZ equations. Note that when we specialize to $ g = \id $, we have $ \mathcal{L} = \mathcal{S} = \mathcal{N} = 0 $, removing all $ \log $-terms and setting $ \alpha^i = 0 $. We can further choose the basis $ \{a^i\}_{i \in I} $ to be orthonormal so the twisted KZ equations above produce the original KZ equations
\begin{align*}
&(\lvl + h^\vee) \frac{\partial}{\partial x_\el} \langle w_0, \mathcal{Y}_1(w_1,x_1) \cdots \mathcal{Y}_\N(w_\N,x_\N) w_{\N+1} \rangle \\
&= \sum_{i \in I} \Bigg( \sum_{p \neq \el} : (x_\el - x_p)^{-1}  : \langle w_0, \mathcal{Y}_1(w_1,x_1) \cdots \mathcal{Y}_\el (a^i(0) w_\el, x_\el) \cdots \mathcal{Y}_p(a^i(0) w_p,x_p) \cdots \mathcal{Y}_\N(w_\N,x_\N) w_{\N+1} \rangle \\
& \quad + x_\el^{-1} \langle w_0, \mathcal{Y}_1(w_1,x_1) \cdots \mathcal{Y}_\el (a^i(0) w_\el, x_\el)  \cdots \mathcal{Y}_\N(w_\N,x_\N) a^i(0) w_{\N+1} \rangle \Bigg),
\end{align*}
as in \cite{huanglepowskyaffine}, for example.

\begin{corollary}
Further assume that $ g $ acts semisimply. Then we have
\begin{equation}\label{eq:ss KZ eq}
\begin{aligned}
&2(\lvl + h^\vee) \frac{\partial}{\partial x_\el} \langle w_0, \mathcal{Y}_1(w_1,x_1) \cdots \mathcal{Y}_\N(w_\N, x_\N) w_{\N+1} \rangle \\
&\quad= \sum_{i \in \mathcal{I}}  \Bigg( \sum_{ p \neq \el}  \left( \frac{x_p}{x_\el}
\right)^{\alpha^i} :(x_\el - x_p)^{-1} :  \langle w_0, \mathcal{Y}_1(w_1,x_1)  \\
& \qquad \cdots \mathcal{Y}_\el({a^i}'(0)w_\el, x_\el) \cdots  \mathcal{Y}_p(a^i(0)w_p, x_p) \cdots \mathcal{Y}_\N(w_\N, x_\N) w_{\N+1} \rangle \\
& \qquad + x_\el^{-\alpha^i -1} \langle w_0, \mathcal{Y}_1(w_1,x_1) \cdots \mathcal{Y}_\el({a^i}'(0)w_\el, x_\el) \cdots \mathcal{Y}_\N(w_\N, x_\N) a^i(\alpha^i) w_{\N+1} \rangle  \\
& \qquad  - \alpha^i x_\el^{-1} \langle w_0, \mathcal{Y}_1(w_1,x_1) \cdots   \mathcal{Y}_\el(a^i(0) {a^i}'(0) w_\el, x_\el) \cdots \mathcal{Y}_\N(w_\N, x_\N) w_{\N+1} \rangle \Bigg),
\end{aligned}
\end{equation}
for each $ \el = 1, \dots, \N $.
\end{corollary}
\begin{proof}
In the case that $ g $ acts semisimply, we have $ \mathcal{L} = \mathcal{S} $ and $ \mathcal{N} = 0 $. Since $ \mathcal{N} = 0 $, the $ \mathcal{\log} $-terms vanish. 
\end{proof}

We can compare this with the ``inner automorphism" special case derived in \cite{deBoer:2001nw}. We choose a Cartan subalgebra $ \mathfrak{h} $, a set of simple roots $ \Pi $, and an invariant symmetric bilinear form $ ( \cdot, \cdot) $. We choose and fix an element $ h \in \mathfrak{h} $ and define $ g \in \Aut(\mathfrak{g}) $ by 
\begin{equation}
g a := e^{\ad h} a , \quad \text{for all } a \in \mathfrak{g}.
\end{equation} 
Since $ g $ is given by the adjoint action of a Cartan element, we immediately have
\begin{equation}
g a = e^{\beta(h)} a, \quad \text{for all } a \in \mathfrak{g}_\beta, \ \beta \in \Delta \sqcup \{ 0\} .
\end{equation} 

We construct a basis by picking an orthonormal basis of $ \mathfrak{h} $ and a set $ \{u_\beta \in \mathfrak{g}_\beta :\beta \in \Delta_+ \} $ of non-zero vectors. Lastly, we pick $ \{v_\beta \in  \mathfrak{g}_{-\beta} : \beta \in \Delta_+ \} $ such that $ (u_\beta, v_\beta) = 1 $ for each $ \beta \in \Delta_+ $. This gives a basis $ \{ a^i \}_{i \in I} $ of eigenvectors of $ g $ whose dual basis with respect to the bilinear form $ ( \cdot, \cdot) $ consists of the same vectors. (However, the basis need not be orthonormal.) In this case, the twisted KZ equations are

\begin{equation}
\begin{aligned}
&(\lvl + h^\vee)\frac{\partial}{\partial x_\el} \langle w_0, \mathcal{Y}_1(w_1,x_1) \cdots \mathcal{Y}_\N(w_\N, x_\N) w_{\N+1} \rangle \\
&\quad= \sum_{i \in I} \Bigg( \sum_{ p \neq \el} :(x_\el - x_p)^{-1} : \left( \frac{x_p}{x_\el} \right)^{\alpha^{i}} \langle w_0, \mathcal{Y}_1(w_1,x_1) \\
&\qquad\cdots \mathcal{Y}_\el({a^i}'(0)w_\el, x_\el) \cdots  \mathcal{Y}_p(a^i(0)w_p, x_p) \cdots \mathcal{Y}_\N(w_\N, x_\N) w_{\N+1} \rangle \\
&\qquad  + x_\el^{-\alpha^i -1} \langle w_0, \mathcal{Y}_1(w_1,x_1) \cdots \mathcal{Y}_\el({a^i}'(0)w_\el, x_\el) \cdots \mathcal{Y}_\N(w_\N, x_\N) a^i(\alpha^i) w_{\N+1} \rangle  \\
&\qquad - \alpha^{i} x_\el^{-1} \langle w_0, \mathcal{Y}_1(w_1,x_1) \cdots   \mathcal{Y}_\el( a^i(0) {a^i}'(0) w_\el, x_\el) \cdots \mathcal{Y}_\N(w_\N, x_\N) w_{\N+1} \rangle  \Bigg).
\end{aligned}
\end{equation}

Further assuming that $ w_{\N+1} $ is annihilated by $ \hat{\mathfrak{g}}^{[g]}_\mathbb{I} $, we obtain the twisted KZ equations as given in (10.44a-c) of \cite{deBoer:2001nw}. (Note that our eigenvalues $ \{\alpha^i\}_{i \in I} $ are the negatives of those in \cite{deBoer:2001nw} because $ Y(u,e^{2\pi \iu}z) = Y(gu,z) $ is used in (10.11a-d), instead of our convention of $ Y(gu,e^{2\pi \iu}z) = Y(u,z) $ in \eqref{eq:equivariance}. But this is simply a change between $ g $ and $ g^{-1} $.) 

\section{Regularity of the twisted KZ equations}
\label{sec:kz regularity}
Now we assume that $ g $ is of finite order $ t $. We will also assume that each module $ W_0, \dots, W_{\N+1} $ has a finite-dimensional space of lowest weight vectors. We will prove that the ``component-isolated singularities" of the solutions to the KZ equations are ``regular singularities". 

First, we recall some notions about functions with ``regular singularities" from \cite{DuHuang} and some theorems about solutions of differential equations with ``simple singularities" (see Section 4 of Appendix B in \cite{Knapp+1986}). 

For $ r \in \mathbb{R}_{>0} $, let $ \mathbb{D}^\times_r(0) $ denote the open punctured disk $ \{z \in \mathbb{C}: 0 < |z| < r \} $, let $ \mathbb{D}_r(0) $ denote the open disk $ \{z \in \mathbb{C}: |z| < r \} $, and let $ \mathbb{D}^\times_r(\infty) $ denote the open punctured disk $ \{z \in \mathbb{C}: 0 < |z^{-1}| < r \} $. We also use $ \mathbb{D} = \mathbb{D}_1(0) $ and $ \mathbb{D}^\times = \mathbb{D}_1^\times (0) $ for the (punctured) unit disk.  

\begin{definition}
Let $ f(z_1, \dots, z_n) $ be a multivalued function defined on some subset of $ \widehat{\mathbb{C}}^n $. Let $ A \in \GL(n,\mathbb{C}) $ and $ (\beta_1, \dots, \beta_n) \in \mathbb{C}^n $, and consider the change of variables
\begin{equation}
(\zeta_1, \dots, \zeta_n) = (z_1, \dots, z_n) A - (\beta_1, \dots, \beta_n).
\end{equation}
Define $ g(\zeta_1, \dots, \zeta_n) = f ( (\zeta_1, \dots, \zeta_n) A^{-1} + (\beta_1, \dots, \beta_n) A^{-1} ) $. Let $ (\delta_1, \dots, \delta_n) \in \{0, \infty\}^n $. We say that $ f(z_1, \dots, z_n) $ has a \emph{component-isolated singularity} at $ (\zeta_1, \dots, \zeta_n) = (\delta_1, \dots, \delta_n) $ if there exist $ r_1, \dots, r_n \in \mathbb{R}_{>0} $ such that $ g (\zeta_1, \dots, \zeta_n) $ is defined on $  \mathbb{D}^\times_{r_1}(\delta_1) \times \cdots \times  \mathbb{D}^\times_{r_n}(\delta_n) $.
\end{definition}

\begin{definition}
We say that a multivalued holomorphic function $ f(z_1, \dots, z_n) $ has a \textit{regular singularity} at $ (z_1, \dots, z_n) = (0,\dots,0) $ if there exist $ r_1, \dots, r_n \in \mathbb{R}_{>0} $ such that $ f $ has a local expansion in the region $ \D_{r_1}^{\times}(0) \times \cdots \times \D_{r_n}^{\times}(0) $ of the form 
\begin{equation}
f(z_1, \dots, z_n) = \sum_{j=1}^J z_1^{s_{1,j}} \cdots z_n^{s_{n,j}} (\log z_1)^{m_{1,j}} \cdots (\log z_n)^{m_{n,j}}  h_{j}(z_1, \dots, z_n),
\end{equation}
for some numbers $ s_{i,j} \in \mathbb{C} $, $ m_{i,j} \in \mathbb{Z}_{\geq 0} $, and for some holomorphic functions $ h_{j}(z_1, \dots, z_n) $ on $ \D_{r_1}(0) \times \cdots \times \D_{r_n}(0)  $. Define $ \varphi_{0}(z) = z $ and $ \varphi_{\infty}(z) = z^{-1} $. Let  $ (\delta_1, \dots, \delta_n) \in \{0, \infty\}^n $. We say that $ f(z_1, \dots, z_n) $ has a \emph{regular singularity} at $  (z_1, \dots, z_n) = (\delta_1, \dots, \delta_n) $ if $ h(\xi_1, \dots, \xi_n) =  f(\varphi_{\delta_1}(\xi_1), \dots ,\varphi_{\delta_n}(\xi_n) ) $ has a regular singularity at $ (\xi_1, \dots, \xi_n) = (0,\dots,0) $.
Let $ A \in \GL(n,\mathbb{C}) $ and $ (\beta_1, \dots, \beta_n) \in \mathbb{C}^n $, and consider the change of variables
\begin{equation} \label{eq:change of variables}
(\zeta_1, \dots, \zeta_n) = (z_1, \dots, z_n) A - (\beta_1, \dots, \beta_n).
\end{equation}
Define $ g(\zeta_1, \dots, \zeta_n) = f( (\zeta_1, \dots, \zeta_n) A^{-1} + (\beta_1, \dots, \beta_n) A^{-1} ) $. Let $ (\delta_1, \dots, \delta_n) \in \{0, \infty\}^n $. We say that $ f(z_1, \dots, z_n) $ has a \emph{regular singularity} at $ (\zeta_1, \dots, \zeta_n) = (\delta_1, \dots, \delta_n) $ if $ g(\zeta_1, \dots, \zeta_n) $ has a regular singularity at $ (\zeta_1, \dots, \zeta_n) = (\delta_1, \dots, \delta_n) $.
\end{definition}

Note that we are using the term ``regular singularity" to refer to a property of a function, and not a property of a system of differential equations. Also note that a function $ f(z_1, \dots, z_n) $ with a regular singularity at $ (z_1, \dots, z_n) = (0, \dots, 0) $ is potentially singular on all of $ ( \D_{r_1}(0) \times \cdots \times \D_{r_n}(0) ) \backslash ( \D_{r_1}^\times(0) \times \cdots \times \D_{r_n}^\times(0) ) $, so this function does not necessarily have an isolated singularity at $ (z_1, \dots, z_n) = (0,\dots,0) $. (In the case where the singularity at $ (z_1, \dots, z_n) = (0,\dots,0) $ is isolated with $ n \geq 2 $, it is in fact removable.) This is why we have the notion of a component-isolated singularity. 

Let $ \mathcal{H}_U $ denote the space of holomorphic functions from $ \D^n $ to a finite-dimensional vector space $ U $. Let $ \mathcal{H}_{\End U} $ denote the space of holomorphic functions from $ \D^n $ to $ \End U $. For $ \vect{k} \in (\mathbb{Z}_{\geq 0})^n $, use $ \partial^{\vect{k}} $ to denote $ \left( \frac{\partial}{\partial z_1} \right)^{k_1} \cdots \left( \frac{\partial}{\partial z_n} \right)^{k_n} $ and use $ (z\partial)^{\vect{k}} $ to denote $ \left( z_1\frac{\partial}{\partial z_1} \right)^{k_1} \cdots \left( z_n \frac{\partial}{\partial z_n} \right)^{k_n} $. Let 
\begin{equation}
\mathcal{D} = \left\{ \sum_{\substack{\vect{k} \in (\mathbb{Z}_{\geq 0})^n \\ \text{finite sum}}} A_{\vect{k}} \partial^{\vect{k}} : A_{\vect{k}} \in \mathcal{H}_{\End U} \right\},
\end{equation}
which is naturally a left $ \mathcal{H}_{\End U} $-module by left multiplication. Let 
\begin{equation}
\mathcal{D}^* = \left\{ \sum_{\substack{\vect{k} \in (\mathbb{Z}_{\geq 0})^n \\ \text{finite sum}}} A_{\vect{k}} (z\partial)^{\vect{k}} : A_{\vect{k}} \in \mathcal{H}_{\End U} \right\},
\end{equation} 
be the $ \mathcal{H}_{\End U} $-submodule of $ \mathcal{D} $ generated by all monomials in $ \{ z_i \frac{\partial}{\partial z_i} \}_{i =1}^n $, which is also naturally an associative unital algebra over $ \mathbb{C} $. Let $ \{ D_\alpha \}_{\alpha \in A} $ be a finite set of operators from $ \mathcal{D}^* $. Let $ \mathcal{I} $ be the left ideal in $ \mathcal{D}^* $ generated by $ \{D_\alpha\}_{\alpha \in A} $. 

\begin{definition}
If $ \mathcal{D}^* / \mathcal{I} $ is finitely generated as an $ \mathcal{H}_{\End U} $-module, we say that the system of partial differential equations $ \{D_\alpha \phi = 0 \}_{\alpha \in A} $ has a \emph{simple singularity} on $ \D^n \backslash (\D^\times)^n $. 
\end{definition} 

For example, let $ U = \mathbb{C}^m $ and let $ \{ D_\alpha \}_{\alpha \in A} $ be the set of operators
\begin{align}
D_i = z_i \frac{\partial}{\partial z_i} - H_i(z), \qquad i = 1, \dots, n ,
\end{align}
where $ H_i(z) $ are $ m \times m $ matrices of holomorphic functions on $ \D^n $. Then the system $ \{ z_i \frac{\partial}{\partial z_i} \psi = H_i(z) \psi \}_{i = 1}^n $ has a simple singularity on $ \D^n \backslash (\D^\times)^n $. 

We have the following theorem about the form of a solution of a system with a simple singularity.

\begin{theorem}[Theorem B.16 from \cite{Knapp+1986}]\label{thm:simple sing PDE}
Let $ \{ D_\alpha \}_{\alpha \in A} $ be a finite set of operators from $ \mathcal{D}^* $ such that the system of partial differential equations $ \{D_\alpha \phi = 0 \}_{\alpha \in A} $ has a simple singularity on $ \D^n \backslash (\D^\times)^n $. Then any multivalued holomorphic solution to $ \{D_\alpha \phi = 0 \}_{\alpha \in A} $ on $ (\D^\times)^n $ has a regular singularity at $ (0,\dots,0) $. Moreover, any $ C^\infty $ solution on $ (0,1)^n $ extends to a multivalued holomorphic solution on $ (\D^\times)^n $.
\end{theorem}

By scaling the variables, we can see that this theorem still holds if the domain $ \mathbb{D}^n $ in $ \mathcal{H}_{U} $ and $ \mathcal{H}_{\End U} $ is replaced by $ \D_{r_1}(0) \times \cdots \times \D_{r_n}(0)  $ for any $ r_1, \dots, r_n \in \mathbb{R}_{>0} $. Note that this theorem does \emph{not} show that a solution exists for any initial condition, nor that any formal solution converges. In the following section, we will use the theory of \textit{ordinary} differential equations with regular singular points to show that formal solutions converge. 

We will now show that the solutions to the twisted KZ equations have regular singularities when $ g $ has finite order $ t $. Since $ g $ restricts to an automorphism of finite order when acting on each finite-dimensional summand $ V_{(n)} $, $ g $ is semisimple. So, we can use the twisted KZ equations in the form of \eqref{eq:ss KZ eq}. Furthermore, the eigenvalues of $ \mathcal{S}_g $ are all of the form $ \alpha \in \{0, 1/t, \dots, (t-1)/t \} $.

First, we will re-express the twisted KZ equations \eqref{eq:ss KZ eq} in a more implicit form. Let $ L(W) $ denote the space of lowest weight vectors of the (twisted) module $ W $. Recall that we will assume that each $ L(W_i) $, for $ i = 0, \dots, \N+1 $, has finite dimension.

Define the operator $ \widetilde{\Omega}_{\el p}^i $ on $ L(W_0) \otimes \cdots \otimes L(W_{\N+1}) $ to be given by
\begin{equation}
 \widetilde{\Omega}_{\el p}^i (w_0 \otimes \cdots \otimes w_{\N+1}) = \frac{1}{2(\lvl + h^\vee)} w_0 \otimes \cdots \otimes {a^i}'(0)w_\el \otimes \cdots \otimes a^i(0)w_p \otimes \cdots \otimes w_{\N+1} .
\end{equation}
Observe that $ \widetilde{\Omega}_{\el p}^i = \widetilde{\Omega}_{p \el}^{i'} $.
Define the operator $  \widetilde{\Omega}_{\el} $ on $ L(W_0) \otimes \cdots \otimes L(W_{\N+1}) $ to be given by
\begin{equation}
\begin{aligned}
\widetilde{\Omega}_{\el}(w_0 \otimes \cdots \otimes w_{\N+1}) &= \frac{1}{2(\lvl + h^\vee)} \sum_{i \in \mathcal{I}} \Big( w_0 \otimes \cdots \otimes {a^i}'(0)w_\el \otimes \cdots \otimes a^i(\alpha^i) w_{\N+1} \\ 
&\qquad - \alpha^i w_0 \otimes \cdots \otimes {a^i}(0) {a^i}'(0)w_\el \otimes \cdots \otimes  w_{\N+1} \Big).
\end{aligned}
\end{equation}
Note that the first term of the summand only contributes when $ \alpha^i = 0 $ and the second term only contributes when $ \alpha^i \neq  0 $.

Consider the (finite-dimensional) vector space $ U = (L(W_0) \otimes \cdots \otimes L(W_{\N+1}))^* $. The formal counterpart to $ \mathcal{H}_U $ is contained within $ U\{ x_1, \dots, x_\N\} $.

Define the operators $ \Omega_{\el p}^i $ and $ \Omega_{\el} $ in $ \End U $ by $ \Omega_{\el p}^i(f) = f \circ \widetilde{\Omega}_{\el p}^i $ and $ \Omega_{\el} (f) = f \circ \widetilde{\Omega}_{\el} $, for $ f \in U $, respectively. These operators extend to operators on $ U\{ x_1, \dots, x_\N\}[\log x_1, \dots, \log x_\N] $, which we also denote by $ \Omega_{\el p}^i $ and $ \Omega_{\el} $. Observe that $ \Omega_{\el p}^i = \Omega_{p \el}^{i'} $.
 
We will consider the twisted KZ equations as the following system of formal partial differential equations
\begin{equation}
\frac{\partial}{\partial x_\el} \psi = \left( \sum_{i \in \mathcal{I}} \sum_{ p \neq \el} \left( \frac{x_p}{x_\el}
\right) ^{\alpha^i} : (x_\el - x_p)^{-1} : \Omega_{\el p}^i + x_\el^{-1} \Omega_{\el} \right) \psi, \qquad \el = 1, \dots, \N,
\end{equation}
or as the following system of partial differential equations
\begin{equation}
\label{eq:KZ equations implicit form}
\frac{\partial}{\partial z_\el} \psi =\left( \sum_{i \in \mathcal{I}} \sum_{ p \neq \el} \left( \frac{z_p}{z_\el}
\right) ^{\alpha^i} (z_\el - z_p)^{-1} \Omega_{\el p}^i + z_\el^{-1} \Omega_{\el} \right)\psi, \qquad \el = 1, \dots, \N.
\end{equation}

By seeing where the coefficients in the right-hand side of \eqref{eq:KZ equations implicit form} are holomorphic, we can see that every solution to the twisted KZ equations is defined on 
\begin{equation}
M^\N = \{ (z_1, \dots, z_\N) \in \widehat{\mathbb{C}}^\N : z_i \neq  0, \infty \text{ and } z_i \neq z_j \text{ when } i \neq j  \}.
\end{equation}
Assume that we have a solution to the twisted KZ equations defined on $ M^\N $. We will show that for any change of variables \eqref{eq:change of variables}, every component-isolated singularity at $ (\delta_1, \dots, \delta_\N) \in \{0, \infty\}^\N $ is a regular singularity. To show this, we will show that the system written in the new variables has a simple singularity at $ (0, \dots, 0) $ by showing that the coefficient matrix is holomorphic in some product of sufficiently small disks centered at $ 0 $. 

Assume we have a change of variables \eqref{eq:change of variables}, or equivalently
\begin{equation}
(z_1, \dots,  z_\N) = ( \zeta_1, \dots, \zeta_\N) A^{-1} + (\beta_1, \dots, \beta_\N) A^{-1} =: ( \zeta_1, \dots, \zeta_\N) B + (\gamma_1, \dots, \gamma_\N).
\end{equation}

Then 
\begin{equation}
z_\el = \sum_{j=1}^\N b_{j \el} \zeta_j + \gamma_{\el} \quad \text{and} \quad \frac{\partial z_\el}{\partial \zeta_j} = b_{j \el} .
\end{equation}

So, 
\begin{align*}
\zeta_j \frac{\partial }{\partial \zeta_j} \psi &= \zeta_j \sum_{\el} \frac{\partial z_\el}{\partial \zeta_j} \frac{\partial}{\partial z_\el} \psi \\
&= \left( \zeta_j \sum_{\el} b_{j \el} \left( \sum_{i \in \mathcal{I}} \sum_{ p \neq \el} \left( \frac{z_p}{z_\el}
\right) ^{\alpha^i} (z_\el - z_p)^{-1} \Omega_{\el p}^i + z_\el^{-1} \Omega_{\el} \right) \right) \psi \\
&=  \left(  \sum_{i \in \mathcal{I}} \sum_{\el < p }\zeta_j \left( b_{j \el}  \left( \frac{z_p}{z_\el}
\right) ^{\alpha^i} (z_\el - z_p)^{-1} \Omega_{\el p}^i+   b_{j p}  \left( \frac{z_\el}{z_p}
\right) ^{\alpha^i} (z_p - z_\el)^{-1}  \Omega_{p \el}^i \right) +\sum_{\el} b_{j \el} \zeta_j z_\el^{-1} \Omega_{\el} \right) \psi \\
&=  \left(  \sum_{i \in \mathcal{I}} \sum_{\el < p } \zeta_j \left( b_{j \el}  \left( \frac{z_p}{z_\el}
\right) ^{\alpha^i} (z_\el - z_p)^{-1} \Omega_{\el p}^i + b_{j p} \left( \frac{z_\el}{z_p}
\right) ^{\alpha^{i}} (z_p - z_\el)^{-1}  \Omega_{\el p}^{i'} \right) + \sum_{\el} b_{j \el} \zeta_j z_\el^{-1} \Omega_{\el} \right) \psi \\
&=  \left( \sum_{i \in \mathcal{I}} \sum_{\el < p } \left( b_{j \el}  \left( \frac{z_p}{z_\el}
\right) ^{\alpha^i} -  b_{j p}  \left( \frac{z_\el}{z_p}
\right) ^{\alpha^{i'}} \right)  \zeta_j  (z_\el - z_p)^{-1} \Omega_{\el p}^i + \sum_{\el} b_{j \el} \zeta_j z_\el^{-1} \Omega_{\el} \right) \psi .
\end{align*}

Assume that $ (\zeta_1, \dots, \zeta_\N) = (\delta_1,\dots, \delta_\N) $ is a component-isolated singularity of a function defined on $ M^\N $. Define the function 
\begin{equation}
s(\delta) = \begin{cases}
1 &\text{if } \delta = 0, \\
-1 &\text{if } \delta = \infty. \\
\end{cases}
\end{equation}
We perform a change of variables from $ \zeta_j $ to $ \eta_j $ such that
\begin{equation}
\eta_j^{s(\delta_j)t} = \zeta_j.
\end{equation}
Hence,
\begin{equation}
\eta_j \frac{\partial}{\partial \eta_j} = \eta_j \frac{\partial \zeta_j}{\partial \eta_j} \frac{\partial}{\partial \zeta_j} = \eta_j s(\delta_j)t \eta_j^{s(\delta_j)t-1} \frac{\partial}{\partial \zeta_j} = s(\delta_j) t \zeta_j \frac{\partial}{\partial \zeta_j}.
\end{equation}

This change of variables handles singularities at both $ 0 $ and $ \infty $. Furthermore, we claim that the possible multivaluedness of $ z_\el^{1/t} = \left(\sum_{k} b_{k \el} \zeta_k + \gamma_{\el} \right)^{1/t} $ as a function of $ (\zeta_1, \dots, \zeta_\N) $ is turned into a single-valued function after changing variables to $ (\eta_1, \dots, \eta_\N) \in \D_{r_1}(0) \times \cdots \times  \D_{r_\N}(0) $.

Suppose that $ z_\el^{1/t} $ is multivalued, and $ \delta_k = 0 $ when $ b_{k\el} \neq 0 $, for all $ k $. If there is still multivaluedness after choosing $ r_k > 0 $ to be arbitrarily small, then we see that $ \gamma_\el =0 $. Then we must have
$ z_\el^{1/t} = (b_{j \el} \zeta_j)^{1/t} $ for some $ j $ (to ensure we have a component-isolated singularity for a function defined on $ M^\N $), which is single-valued as a function of $ \eta_j $. 
Now suppose that $ \delta_j = \infty $ when $ b_{j\el} \neq 0 $, for some $ j $. Then we must have
$ z_\el^{1/t} = \left(\sum_{k} b_{k \el} \zeta_k + \gamma_{\el} \right)^{1/t} $ with $ b_{k \el} = 0 $ when $ \delta_k = \infty $ and $ k \neq j $ (to ensure we have a component-isolated singularity for a function defined on $ M^\N $). After choosing each $ r_j $ to be sufficiently small, we have single-valuedness as a function of $ (\eta_1, \dots, \eta_\N) $.

Hence, to obtain a differential equation with the desired singularity at $ (\zeta_1, \dots, \zeta_\N) = (\delta_1, \dots, \delta_\N) $ and single-valued coefficients, we will look at 
\begin{equation}
\eta_j \frac{\partial }{\partial \eta_j} \psi =  s(\delta_j) t\left( \sum_{i \in \mathcal{I}} \sum_{\el < p } \left( b_{j \el}  \left( \frac{z_p}{z_\el} \right) ^{\alpha^i} -  b_{j p}  \left( \frac{z_\el}{z_p} \right) ^{\alpha^{i'}} \right)  \zeta_j  (z_\el - z_p)^{-1} \Omega_{\el p}^i +  \sum_{\el}  b_{j \el} \zeta_j z_\el^{-1} \Omega_{\el} \right) \psi ,
\end{equation}
where $ \zeta_k $ and $ z_k^{1/t} $ are single-valued holomorphic functions of $ (\eta_1, \dots, \eta_\N) $. 

We want to show that $ \left( b_{j \el}  \left( \dfrac{z_p}{z_\el}
\right) ^{\alpha^i} -  b_{j p}  \left( \dfrac{z_\el}{z_p}
\right) ^{\alpha^{i'}} \right) \zeta_j (z_\el - z_p)^{-1} $ and $  b_{j \el} \zeta_j z_\el^{-1} $ are holomorphic functions of $ (\eta_1, \dots, \eta_\N) $ on $ \D_{r_1}(0) \times \cdots \times  \D_{r_\N}(0) $ for some sufficiently small $ r_1, \dots, r_\N > 0 $. 

First, we show that $ b_{j \el} \zeta_j  z_\el^{-1} $ is holomorphic. If $ b_{j \el} = 0 $, then $ b_{j \el} \zeta_j  z_\el^{-1} = 0 $, which is holomorphic. Otherwise, when $ b_{j \el} \neq 0 $, we have two cases depending on the value of $ \delta_j $. If $ \delta_j = 0 $, then the only possible singularity can come from $ z_\el \to 0 $ as $ (\eta_1,\dots, \eta_\N) \to 0 $. Then $ \gamma_\el = 0 $, and $ b_{k \el} = 0 $ for all but $ k = j $ to ensure that $ z_\el \to 0 $ and $ z_\el \neq 0 $ when $ (\eta_1,\dots, \eta_\N) \in \D_{r_1}^\times(0) \times \cdots \times \D_{r_\N}^\times(0) $ (this is necessary since we have a component-isolated singularity for a function defined on $ M^\N $). Then $ b_{j \el} \zeta_j  z_\el^{-1} = 1 $, which is holomorphic. If $ \delta_j = \infty $, then $ b_{k \el} = 0 $ when $ \delta_k = \infty $ and $ k \neq j $ to ensure that $ z_\el \neq 0 $ when $ (\eta_1,\dots, \eta_\N) \in \D_{r_1}^\times(0) \times \cdots \times \D_{r_\N}^\times(0) $. Then
\begin{align*}
b_{j \el} \zeta_j  z_\el^{-1} = \frac{b_{j \el}}{\sum_{k} b_{k\el} \eta_k^{s(\delta_k) t} \eta_j^{t} + \gamma_\el \eta_j^{t}} \to \frac{b_{j \el}}{b_{j \el}} = 1
\end{align*}
as $ (\eta_1, \dots, \eta_\N) \to (0,\dots, 0) $. So, $ b_{j \el} \zeta_j  z_\el^{-1} $ is holomorphic in all cases.

Second, we show that $ \left( b_{j \el}  \left( \dfrac{z_p}{z_\el}
\right) ^{\alpha^i} -  b_{j p}  \left( \dfrac{z_\el}{z_p}
\right) ^{\alpha^{i'}} \right) \zeta_j (z_\el - z_p)^{-1} $ is holomorphic. If $ \alpha^i = 0 $, then $ \alpha^{i'} = 0 $ and
\begin{align*}
\left( b_{j \el}  \left( \dfrac{z_p}{z_\el}
\right) ^{\alpha^i} -  b_{j p}  \left( \dfrac{z_\el}{z_p}
\right) ^{\alpha^{i'}} \right) \zeta_j (z_\el - z_p)^{-1} = \left( b_{j \el} -  b_{j p} \right) \zeta_j (z_\el - z_p)^{-1},
\end{align*}
which is holomorphic by arguments similar to those used for $ b_{j \el} \zeta_j  z_\el^{-1} $. Otherwise, $ \alpha^i \neq 0 $ and $ \alpha^{i'} = 1 - \alpha^i $. Then 
\begin{align*}
\left( b_{j \el}  \left( \dfrac{z_p}{z_\el}
\right) ^{\alpha^i} -  b_{j p}  \left( \dfrac{z_\el}{z_p}
\right) ^{\alpha^{i'}} \right) \zeta_j (z_\el - z_p)^{-1} = z_\el^{- \alpha^i} z_p^{-\alpha^{i'}} \left( b_{j \el} z_p -  b_{j p} z_\el \right) \zeta_j (z_\el - z_p)^{-1}.
\end{align*}
If $ b_{j\el} - b_{jp} = 0 $, then 
\begin{align*}
\left( b_{j \el}  \left( \dfrac{z_p}{z_\el}
\right) ^{\alpha^i} -  b_{j p}  \left( \dfrac{z_\el}{z_p}
\right) ^{\alpha^{i'}} \right) \zeta_j (z_\el - z_p)^{-1} =  - b_{j \el} \frac{\zeta_j }{z_\el^{ \alpha^i} z_p^{\alpha^{i'}}}.
\end{align*}
This is holomorphic by arguments similar to those used for $ b_{j \el} \zeta_j  z_\el^{-1} $.

If $ b_{j\el} - b_{jp} \neq 0 $, then we need to check what happens when $ \delta_j = 0 $ and $ \delta_j = \infty $.

If $ \delta_j = 0 $ and $ z_\el - z_p \to 0 $, then $ \gamma_\el - \gamma_p = 0 $ and $ b_{k \el} - b_{k p} = 0 $ for all $ k \neq j $. Then 
\begin{align*}
\left( b_{j \el}  \left( \dfrac{z_p}{z_\el}
\right)^{\alpha^i} -  b_{j p}  \left( \dfrac{z_\el}{z_p}
\right)^{\alpha^{i'}} \right) \zeta_j (z_\el - z_p)^{-1} =  \left( b_{j \el}  \left( \dfrac{z_p}{z_\el}
\right)^{\alpha^i} -  b_{j p}  \left( \dfrac{z_\el}{z_p}
\right)^{\alpha^{i'}} \right) (b_{j\el} - b_{jp})^{-1} .
\end{align*}
Since $ z_\el - z_p \to 0 $, we have $ z_\el, z_p \nrightarrow 0 $, so the only possible singularity comes from $ z_\el, z_p \to \infty $. Since $ b_{k \el} - b_{k p} = 0 $ for all $ k \neq j $, we have $ z_\el/z_p \to 1 $.

If $ \delta_j = 0 $ and $ z_\el \to 0 $, then $ z_p, z_\el - z_p \nrightarrow 0 $. So, the only possible singularity comes from the term

\begin{align*}
 b_{j \el}  \left( \dfrac{z_p}{z_\el}
\right) ^{\alpha^i} \zeta_j (z_\el - z_p)^{-1} =  b_{j \el} \frac{\zeta_j}{z_\el^{\alpha^i}} \frac{1}{z_p^{\alpha^{i'}}} \left(\frac{z_\el}{z_p} - 1 \right)^{-1} ,
\end{align*}
which is holomorphic by arguments similar to those used for $ b_{j \el} \zeta_j  z_\el^{-1} $. Similar arguments work for $ \delta_j = 0 $ and $ z_p \to 0 $.

If $ \delta_j = \infty $, then $ z_\el - z_p = \sum_{k} (b_{k \el} - b_{kp} ) \zeta_k + \gamma_\el - \gamma_p $ with $ b_{k \el} - b_{kp} = 0 $ when $ \delta_k = \infty $ and $ k \neq j $. So
\begin{align*}
 \zeta_j (z_\el - z_p)^{-1} \to (b_{j \el} - b_{j p} )^{-1} .
\end{align*} 
Since $ b_{j \el} -b_{j p} \neq 0 $, we have  $ b_{j \el} \neq 0 $ or $  b_{j p} \neq 0 $.
If $ b_{j \el} \neq 0 $, then  $ b_{k \el} = b_{k p} = 0 $ when $ \delta_k = \infty $ and $ k \neq j $. So
\begin{align*}
\frac{z_p}{z_\el} =  \frac{\sum_{k} b_{kp} \zeta_k + \gamma_p }{\sum_{k} b_{k\el} \zeta_k + \gamma_\el} =  \frac{\sum_{k} b_{kp} \eta_k^{s(\delta_k) t}\eta_j^{t}  + \gamma_p \eta_j^{t} }{\sum_{k} b_{k \el} \eta_k^{s(\delta_k) t}\eta_j^{t} + \gamma_\el \eta_j^{t}} \to \frac{b_{j p}}{b_{j \el}}
\end{align*}
is holomorphic. Similarly if $ b_{j p} \neq 0 $.

Thus, 
\begin{align*}
\left( b_{j \el}  \left( \dfrac{z_p}{z_\el}
\right) ^{\alpha^i} -  b_{j p}  \left( \dfrac{z_\el}{z_p}
\right) ^{\alpha^{i'}} \right) \zeta_j (z_\el - z_p)^{-1}
\end{align*}
is holomorphic in all cases.

Thus, we have proven the following result.

\begin{proposition}
If the order of $ g $ is finite and the lowest weight spaces $ L(W_0) , \dots , L(W_{\N+1}) $ are finite dimensional, then the solutions to the twisted KZ equations \eqref{eq:ss KZ eq} (or equivalently \eqref{eq:KZ equations implicit form}) defined on $ M^\N $ have regular singularities at every component-isolated singularity.
\end{proposition}

\begin{remark}
We now show that simple singularities do not immediately follow from the fact that the coefficients in \eqref{eq:KZ equations implicit form} have singularities at $ z_j = 0, \infty $ and $ z_j - z_k = 0 $. Take for simplicity $ \N = 2 $ and $ g = 1 $, so that the KZ equations are 
\begin{align*}
\frac{\partial}{\partial z_1} \psi =\left( (z_1 - z_2)^{-1} \Omega_{12} + z_1^{-1} \Omega_{1} \right)\psi, \\
\frac{\partial}{\partial z_2} \psi =\left( (z_2 - z_1)^{-1} \Omega_{12} + z_2^{-1} \Omega_{2} \right)\psi,
\end{align*}
where $ \Omega_{12} = \Omega_{21} = \sum_{i \in \mathcal{I}} \Omega^i_{12} $.

Let $ z_1 = \zeta_1 + \zeta_2 + 1 $, $ z_2 = \zeta_1 - \zeta_2 + 1 $, which implies $ z_1 - z_2 = 2 \zeta_2 $. This provides a component-isolated singularity at $ (\zeta_1, \zeta_2) = (0, 0) $, and hence the KZ equations have a simple singularity at $ (\zeta_1, \zeta_2) = (0, 0) $, as we have just shown above. However, when we consider the seemingly similar system  
\begin{align*}
\frac{\partial}{\partial z_1} \psi &=\left( (z_1 - z_2)^{-1} \Omega_{12} + z_1^{-1} \Omega_{1} \right)\psi, \\
\frac{\partial}{\partial z_2} \psi &=\left( (z_1 - z_2)^{-1} \Omega_{12} + z_2^{-1} \Omega_{2} \right)\psi,
\end{align*}
we have 
\begin{align*}
\zeta_1 \frac{\partial}{\partial \zeta_1} \psi &= \left( \frac{\zeta_1}{\zeta_2} \Omega_{12} + \frac{\zeta_1}{\zeta_1 + \zeta_2 + 1} \Omega_{1} + \frac{\zeta_1}{\zeta_1 - \zeta_2 + 1} \Omega_2 \right)\psi, \\
\zeta_2 \frac{\partial}{\partial \zeta_2} \psi &= \left( \frac{\zeta_2}{\zeta_1 + \zeta_2 + 1} \Omega_{1} - \frac{\zeta_2}{\zeta_1 - \zeta_2 + 1} \Omega_2 \right)\psi. 
\end{align*}
The coefficient of $ \Omega_{12} $ in the first equation is not holomorphic in any product of small disks centered at zero. Hence, the simple singularities of the twisted KZ equations follow from more than the fact that the coefficients in \eqref{eq:KZ equations implicit form} have singularities at $ z_j = 0, \infty $ and $ z_j - z_k = 0 $. 
\end{remark}
 
\begin{remark}
When dealing with first-order homogeneous linear systems of partial differential equations of the form
 \begin{align*}
 \frac{\partial}{\partial z_i} \psi = A_i \psi, \quad \text{for } i = 1, \dots, \N, 
\end{align*}  
one usually shows that the system is consistent (also known as integrable or compatible) in the sense that 
\begin{align*}
 \frac{\partial}{\partial z_j} A_i - \frac{\partial}{\partial z_i} A_j + [A_i, A_j] = 0\quad, \text{ for all } i, j = 1, \dots, \N . 
\end{align*}
This guarantees a unique solution exists for any initial condition. We have not shown that the twisted KZ equations are consistent. In the following section, we will use a different method to show that solutions exist and that the formal series solution given by the correlation function converges in certain domains.
\end{remark}

\section{Convergence when $ g $ has finite order}
\label{sec:regular singularities}
In this section, we return to the setting of Section~\ref{sec:DEs} while further assuming that $ g $ has finite order $ t $. We continue to generalize Huang's method in \cite{HuangDEs} by using specific filtrations of the ring $ R $ and the $ R $-module $  T $ to derive differential equations with ``regular singular points". As a consequence, we obtain the ``convergence property for products of $ \N $ twisted intertwining operators", as conjectured in \cite{Reptheoryandorbifoldcft}, in the special case that all of the twisted intertwining operators are of \gtype\ and $ W_0 , \dots, W_{N+1} $ are $ C_1 $-cofinite and discretely graded. 

Recall that Theorem~\ref{thm:simple sing PDE} does not show that a solution exists for every initial condition, nor that every formal solution converges. Compare this to the standard theory for an ordinary differential equation with \emph{regular singular points} (see, for example, Section B1 of \cite{Knapp+1986}). 

\begin{theorem}\label{thm:reg sing ODE}
Let $ z $ be a complex variable and let $ z' = (z_1,\dots, z_d) $ be complex parameters. Let 
\begin{equation}\label{eq:reg sing ODE 1}
\left(z\frac{d}{dz} \right)^n \psi(z';z) + \sum_{k = 0}^{n-1} a_{k}(z';z) \left(z\frac{d}{dz} \right)^k \psi(z';z) = 0, 
\end{equation}
be an ordinary differential equation such that $ a_i(z';z) $ are holomorphic in $ (z';z) \in U \times \D_{r}(0) $ for some open set $ U \in \mathbb{C}^d $ and some $ r > 0 $.  Then a solution exists for any initial condition at a point in $ \D_{r}^\times(0) $, and furthermore, every solution has the form of a regular singularity at $ z = 0 $. Conversely, any formal solution $ \psi(z';z) $ with a regular singularity at $ z = 0 $, whose coefficients are all holomorphic functions of $ z' \in U $, converges absolutely in $ z $ to a multivalued solution in $ z \in \D_{r}^\times(0) $. Furthermore, the solution is holomorphic in the parameters $ z' \in U $.
\end{theorem}

We will now show that the coefficients of the differential equation \eqref{eq:DE} can be chosen so that there is a regular singularity at certain prescribed points. Since $ g $ is of finite order, $ g $ is semisimple and $ \mathcal{L} = \mathcal{S} $. So, we can ignore all $ \log x_i $ in $ R $, and the eigenvalues for $ \mathcal{S} $ are in $ \frac{1}{t}\mathbb{Z} $. Hence, we will use
\begin{equation}
R = \mathbb{C}[x_i^{\pm 1/t},(x_i - x_j)^{-1}: i, j = 1, \dots, \N  \text{ with } i < j ] .
\end{equation} 
We define the \emph{degree} of a monomial $ \prod_{i} x_i^{p_i} \prod_{j<k} (x_j - x_k)^{q_{jk}} $, $ p_i \in \frac{1}{t} \mathbb{Z} $ and $ q_{jk} \in \mathbb{Z}_{\leq 0} $ to be $ \sum_{i} p_i + \sum_{j < k} q_{jk} $. An element $ f(x_1, \dots, x_\N) \in R $ has \emph{degree} $ m \in \frac{1}{t} \mathbb{Z} $, denoted by $ \degree f(x_1, \dots, x_\N) = m $, if $ f(x_1, \dots, x_\N) $ is a $ \mathbb{C} $-linear combination of degree $ m $ monomials.

Let $ c = x_i $  or $ c = \ x_i - x_j $ for $ i , j = 1 ,\dots, \N $ with $ i < j $. To study the singularity $ c = 0 $, we use a certain filtration of $ R $. We first define the subrings 
\begin{equation}
R^{(c = 0)} = \mathbb{C}[x_i^{\pm 1/t},(x_i - x_j)^{-1}: i, j = 1, \dots, \N  \text{ with } i < j, \ \text{ excluding negative powers of } c] .
\end{equation}

Define $ R^{(c = 0)}_{(\deg 0)} $ to be the subspace of degree zero elements in $ R^{(c = 0)} $. We define a filtration of $ R $ with 
\begin{equation}
F_m^{(c = 0)}(R) = \Span_\mathbb{C} \{ f(x_1,\dots, x_\N) \in R : f(x_1,\dots, x_\N) c^m \in R^{(c=0)} \}
\end{equation}
with $ m \in \mathbb{Z}_{\geq 0} $, or $ m \in \frac{1}{t}\mathbb{Z}_{\geq 0} $ if $ c $ is proportional to one of $ x_1, \dots, x_\N $. Define $ F_r^{(c = 0)} (T) $ to be the vector subspace of $ T $ spanned by $ f(x_1,\dots, x_\N) w_0 \otimes \cdots \otimes w_{\N+1} $ for all $ f(x_1,\dots, x_\N) \in F_m^{(c=0)}(R) $ and weight-homogeneous $ w_i \in W_i $ satisfying $ m + \sigma \leq r $. Define $ F_r^{(c = 0)} (J) = F_r^{(c = 0)} (T) \cap J $. These are compatible filtrations on the $ R $-modules $ T $ and $ J $ in the sense that $ F_m^{(c=0)}(R) F_r^{(c=0)}(T) \subseteq F_{m+r}^{(c=0)}(T) $ and $ F_m^{(c=0)}(R) F_r^{(c=0)}(J) \subseteq F_{m+r}^{(c=0)}(J) $. Define
\begin{equation}
T^{(c = 0)} = R^{(c = 0)} \otimes W_0 \otimes \cdots \otimes W_{\N+1},
\end{equation}
with grading $ T^{(c=0)} = \coprod_{m \in \mathbb{R}} T^{(c=0)}_{(m)} $ given by the grading by real parts of the weights for $ W_i $, with trivial grading for $ R^{(c = 0)} $. Then $ F_m^{(c=0)}(R) T_{(r)}^{(c=0)} \subseteq F_{m+r}^{(c=0)}(T) $. 

We recall the generators $ \mathcal{A}_i(u,w_0, \dots, w_{\N+1}) $ as defined in Section~\ref{sec:DEs}, but now with $ g $ of finite order:

\begin{equation}
\begin{aligned}
&\mathcal{A}_0(u,w_0, \dots, w_{\N+1}) \\
& \quad = u_{\alpha'-1} w_0 \otimes w_1 \otimes \dots \otimes w_{\N+1} \\
& \qquad + \sum_{p=1}^\N \sum_{k \geq 0} {\alpha' -1 \choose k} x_p^{1+k-\alpha'} w_0 \otimes \dots \otimes \Op_\pm (x_p^{-1}) u_{k} \Op_{\pm}^{-1}(x_p^{-1})w_p \otimes \dots \otimes w_{\N+1} \\
&\qquad - w_0 \otimes \dots \otimes w_\N \otimes u_{\alpha'-1}^* w_{\N+1},
\end{aligned}
\end{equation}

\begin{equation}\label{eq:A_ell}
\begin{aligned}
&\mathcal{A}_\el(u,w_0,\dots, w_{\N+1}) \\
&\quad = -\sum_{k \geq 0} x_\el^{-\alpha + k} u_{\alpha - 1 - k}^* w_0 \otimes w_1 \otimes \dots \otimes w_{\N+1} \\
& \qquad - \sum_{\substack{p=1,\dots,\N \\ p \neq \el}} \sum_{j,k \geq 0}{\alpha \choose k} x_p^{\alpha-k} x_\el^{-\alpha}  (x_\el - x_p)^{-1-j}  w_0 \otimes \dots \otimes  u_{j+k} w_p \otimes \dots \otimes w_{\N+1} \\
& \qquad +\sum_{k \geq 0} {\alpha \choose k} x_\el^{- k}  w_0  \otimes \dots \otimes u_{-1+k} w_\el \otimes \dots \otimes w_{\N+1}  \\
& \qquad - \sum_{k \geq 0} x_\el^{-\alpha - k - 1}
w_0 \otimes \dots \otimes w_\N \otimes u_{\alpha + k} w_{\N+1},
\end{aligned}
\end{equation}

\begin{equation}
\begin{aligned}
&\mathcal{A}_{\N+1}(u,w_0,\dots, w_{\N+1}) \\
&\quad = -u_{\alpha - 1}^* w_0 \otimes w_1 \otimes  \dots \otimes w_{\N+1} \\
& \qquad + \sum_{p =1}^\N \sum_{k \geq 0} { \alpha - 1 \choose k} x_p^{-1 - k+ \alpha}  w_0 \otimes \dots \otimes   u_{k} w_p \otimes \dots \otimes w_{\N+1}\\
&\qquad + w_0 \otimes \dots \otimes w_\N \otimes u_{\alpha-1} w_{\N+1} .
\end{aligned}
\end{equation}

\begin{lemma}\label{lem:(c=0) grade shift}
Let $ u \in V^{[\alpha]}_+ $ for $ \alpha \in \mathbb{C} $ with $ \Re(\alpha) \in [0,1) $, and let $ w_i \in W_i $ be weight homogeneous. Let 
\begin{gather*}
w^{(0)} =  u_{\alpha' - 1} w_0 \otimes w_1 \otimes \cdots \otimes w_{\N+1}, \qquad w^{(\N+1)} = w_0 \otimes \cdots \otimes w_\N \otimes u_{\alpha-1}w_{\N+1} \\
\text{or } \ w^{(\el)} = w_0 \otimes  \cdots \otimes w_{\el-1} \otimes u_{-1} w_\el \otimes w_{\el+1} \otimes \cdots \otimes w_{\N+1},  \quad \el = 1, \dots, \N ,
\end{gather*} 
and assume $ \Re (\wt w^{(i)}) = s $. Then $ w^{(i)} \in F_s^{(c = 0)}(J) + \coprod_{m > 0} F_m^{(c=0)}(R) T_{(s-m)}^{(c=0)} $. 
\end{lemma}

\begin{proof}
This follows from an explicit case-by-case check that $ \mathcal{A}_i(u,w_0, \dots ,w_{\N+1}) \in F_s^{(c = 0)}(T) $ and that $ \mathcal{A}_i(u,w_0, \dots ,w_{\N+1}) - w^{(i)} \in F_{s-1/t}(T) $. For example, we will verify the case $ w^{(\el)} $ for $ \el = 1, \dots, \N $. In this case, $ s = \Re (\wt w^{(\el)}) = \wt u + \sigma $, where $ \sigma = \Re( \wt w_0 + \cdots + \wt w_{\N+1} ) $. First, we check that $ \mathcal{A}_\el(u,w_0, \dots ,w_{\N+1}) \in F_s^{(c = 0)}(T) $ by checking each term in \eqref{eq:A_ell}.

The only case where $ x_\el^{-\alpha + k} u_{\alpha - 1 - k}^* w_0 \otimes w_1 \otimes \dots \otimes w_{\N+1} $ may have a negative power of $ c $ is when $ c = x_\el $, $ \alpha > 0 $ and $ k = 0 $, i.e., we have $ x_\el^{-\alpha} $. In this case, 
\begin{equation*}
\alpha + \Re ( \wt u_{\alpha - 1}^* w_0 \otimes w_1 \otimes \dots \otimes w_{\N+1} ) = \alpha - \wt u + \alpha + \sigma < \wt u + \sigma = s,
\end{equation*}
since $ \wt u \geq 1 > \alpha $ implies $ \wt u > 2 \alpha - \wt u $. Otherwise, for any other $ c $ and $ k \geq 0 $, we have
\begin{equation*}
\Re ( \wt u_{\alpha - 1 - k}^* w_0 \otimes w_1 \otimes \dots \otimes w_{\N+1} ) = - \wt u + \alpha - k + \sigma < \wt u + \sigma = s,
\end{equation*}
since $ \wt u \geq 1 > \alpha $ implies $ \wt u > \alpha - \wt u $.

There are three cases where $  x_p^{\alpha-k} x_\el^{-\alpha}  (x_\el - x_p)^{-1-j}  w_0 \otimes \dots \otimes  u_{j+k} w_p \otimes \dots \otimes w_{\N+1} $, for $ p=1,\dots,\N $ with $ p \neq \el $, may have a negative power of $ c $. 

In the case that $ c = x_p $ and $ \alpha - k < 0 $, we have 
\begin{align*}
- \alpha + k + \Re( \wt  w_0 \otimes \dots \otimes  u_{j+k} w_p \otimes \dots \otimes w_{\N+1} ) &= - \alpha + k + \wt u - j - k - 1 + \sigma \\
&= - \alpha + \wt u - j - 1 + \sigma \\
&< \wt u + \sigma = s,
\end{align*}
since $ \alpha, j \geq 0 $.

In the case that $ c = x_\el $ and $ - \alpha < 0 $, we have 
\begin{align*}
 \alpha + \Re( \wt  w_0 \otimes \dots \otimes  u_{j+k} w_p \otimes \dots \otimes w_{\N+1} ) &=  \alpha + \wt u - j - k - 1 + \sigma \\
&< \wt u + \sigma = s,
\end{align*}
since $ \alpha - 1 < 0 $ and $ j,k \geq 0 $.

In the case that $ c = x_\el - x_p $, we have
\begin{align*}
1 + j + \Re( \wt  w_0 \otimes \dots \otimes  u_{j+k} w_p \otimes \dots \otimes w_{\N+1} ) &=  1 + j + \wt u - j - k - 1 + \sigma \\
&= \wt u - k + \sigma \\
&\leq \wt u + \sigma = s,
\end{align*}
since $ k \geq 0 $.

Otherwise, for other possible $ c $ or values of $ \alpha $ and $ k $, we have
\begin{align*}
\Re( \wt  w_0 \otimes \dots \otimes  u_{j+k} w_p \otimes \dots \otimes w_{\N+1} ) &=  \wt u - j - k - 1 + \sigma < \wt u + \sigma = s.
\end{align*}

The only case where $ x_\el^{- k}  w_0  \otimes \dots \otimes u_{-1+k} w_\el \otimes \dots \otimes w_{\N+1}  $ may have a negative power of $ c $ is when $ c = x_\el $ and $ k > 0 $. In this case, 

\begin{equation*}
k + \Re ( \wt w_0  \otimes \dots \otimes u_{-1+k} w_\el \otimes \dots \otimes w_{\N+1} ) = k + \wt u - k + \sigma = \wt u + \sigma = s.
\end{equation*}
Otherwise, for any other $ c $ and $ k \geq 0 $, we have
\begin{equation*}
\Re ( \wt w_0  \otimes \dots \otimes u_{-1+k} w_\el \otimes \dots \otimes w_{\N+1} ) =  \wt u - k + \sigma \leq \wt u + \sigma = s.
\end{equation*}
Note that this becomes a strict inequality when $ k > 0 $, i.e., the element on the left is not $ w^{(\el)} $.

The only case where $ x_\el^{-\alpha - k - 1} w_0 \otimes \dots \otimes w_\N \otimes u_{\alpha + k} w_{\N+1} $ may have a negative power of $ c $ is when $ c = x_\el $ and $ k \geq 0 $. In this case, 

\begin{equation*}
\alpha + k + 1 + \Re ( \wt w_0 \otimes \dots \otimes w_\N \otimes u_{\alpha + k} w_{\N+1} ) = \alpha + k + 1 + \wt u - \alpha - k - 1 + \sigma = \wt u + \sigma = s.
\end{equation*}
Otherwise, for any other $ c $ and $ k \geq 0 $, we have
\begin{equation*}
\Re ( \wt w_0 \otimes \dots \otimes w_\N \otimes u_{\alpha + k} w_{\N+1} ) =  \wt u - \alpha - k - 1 + \sigma < \wt u + \sigma = s.
\end{equation*}

In the above calculations, we have also checked that each term in $  \mathcal{A}_\el(u,w_0, \dots ,w_{\N+1}) - w^{(\el)} $ has weight with a real part strictly less than $ s $. Hence, we have confirmed that $ \mathcal{A}_\el (u,w_0, \dots ,w_{\N+1}) \in F_s^{(c = 0)}(T) $.

We omit the checks for $ w^{(0)} $ and $ w^{(\N+1)} $, which can be performed similarly to the check above (and similarly to the proof of Lemma~\ref{lem:grade shift}).
\end{proof}

\begin{proposition} \label{prop:T^(c=0) = J + F_M(T)}
There exists $ M \in \mathbb{Z} $ such that for any $ r \in \mathbb{R} $, $ F_r^{(c = 0)}(T) \subseteq F_r^{(c = 0)}(J) + F_M(T) $. 
\end{proposition}

\begin{proof}
Choose the same $ M $ as in Proposition~\ref{prop:T = J + F_M(T)}. When $ r \leq M $, $ F_r^{(c=0)}(T) $ is spanned by elements of the form $ f(x_1, \dots, x_\N) w_0 \otimes \cdots \otimes w_{\N+1} $ with $ f(x_1, \dots, x_{\N}) \in F_m^{(c=0)}(R) $ and $ m + \sigma \leq r \leq M $. Since $ \sigma \leq M $, this element is in $ F_M(T) $. Hence, $ F_r^{(c=0)}(T) \subseteq  F_M(T) \subseteq F_{r}^{(c = 0)}(J) + F_M(T) $. 
Since $ T $ is discretely graded, we can proceed by induction on $ r $. Assume that $ s > M $ and $ F_r^{(c = 0)}(T) \subseteq F_{r}^{(c = 0)}(J) + F_M(T) $ whenever $ r < s $. Consider an element $ w \in T_{(s)}^{(c = 0)} $. Since $ s > M $, $ T_{(s)}^{(c = 0)} $ is $ R^{(c=0)} $-spanned by elements $ w $ in Lemma~\ref{lem:(c=0) grade shift}. Hence, $ w \in F_{s}^{(c = 0)}(J) + \coprod_{m > 0} F^{(c=0)}_{m}(R) T^{(c = 0)}_{(s - m)} $. Since $ s - m < s $ when $ m >0 $, we have $ T_{(s-m)}^{(c=0)} \subseteq F^{(c=0)}_{s-m}(J) + F_M(T) $ by the induction hypothesis. So, $ F^{(c=0)}_{m}(R) T_{(s-m)}^{(c=0)} \subseteq F^{(c=0)}_{s}(J) + F_M(T) $. So, $ T_{(s)}^{(c = 0)} \subseteq F^{(c=0)}_{s}(J) + F_M(T) $. Thus, $ F_s^{(c=0)}(T) \subseteq  F^{(c=0)}_{s}(J) + F_M(T) $.
\end{proof}

\begin{lemma}\label{lem:sigma+S}
For any $ s \in [0,1) $, there exists $ S \in \mathbb{R} $ such that $ s + S \in \mathbb{Z}_{>0} $, and for any weight-homogeneous $ w_i \in W_i $ and  any $ \mathcal{W}_1 \in F_\sigma^{(c = 0)}(J), \mathcal{W}_2 \in  F_M(T) $ satisfying $ \sigma \in s + \mathbb{Z} $ and
\begin{align*}
w_0 \otimes \cdots \otimes w_{\N+1} = \mathcal{W}_1 + \mathcal{W}_2,
\end{align*} we have
\begin{align*}
c^{\sigma + S} \mathcal{W}_2 \in T^{(c = 0)}.
\end{align*}
\end{lemma}

\begin{proof}
Pick $ S \in \mathbb{R} $ with $ s + S \in \mathbb{Z}_{> 0} $ to be sufficiently large so that $ T_{(r)} = 0 $ when $ r < - S $. Let $ w_i \in W_i $ be weight-homogeneous with $ \sigma \in s + \mathbb{Z} $, and let $ \mathcal{W}_1 \in F_\sigma^{(c = 0)}(J) $ and $ \mathcal{W}_2 \in  F_M(T) $ satisfy
\begin{align*}
w_0 \otimes \cdots \otimes w_{\N+1} = \mathcal{W}_1 + \mathcal{W}_2.
\end{align*}
Then $ \sigma + S \in \mathbb{Z} $ and $ \sigma \geq - S $ so $ \sigma + S \in \mathbb{Z}_{\geq 0} $. Then
\begin{align*}
c^{\sigma + S} \mathcal{W}_2 = c^{\sigma + S} w_0 \otimes \cdots \otimes w_{\N+1} - c^{\sigma + S} \mathcal{W}_1 \in T^{(c=0)} + c^{\sigma + S} F_\sigma^{(c = 0)}(J) \subseteq T^{(c=0)}.
\end{align*} 
To show the last inclusion, consider $ f(x_1,\dots,x_{\N}) \widetilde{w}_0 \otimes \cdots \otimes \widetilde{w}_{\N+1} \in F_\sigma^{(c = 0)}(T) $ with $ f(x_1, \dots, x_\N) \in F^{(c=0)}_m(R) $ and $ \Re(\wt \widetilde{w}_0 \otimes \cdots \otimes \widetilde{w}_{\N+1} ) = \widetilde{\sigma} $. Then
\begin{align*}
m - S - \sigma \leq m + \widetilde{\sigma} - \sigma  \leq \sigma - \sigma = 0 
\end{align*}
implies that $ f(x_1,\dots,x_\N) c^{\sigma + S} \in R^{(c=0)} $.
\end{proof}

We now have an extension of Theorem~\ref{thm:formal DEs general}. 
\begin{theorem}  \label{thm: DEs c}
Assume $ g $ has finite order. For any change of variables 
\begin{align*}
(\zeta_1, \dots, \zeta_\N) = (z_1, \dots, z_\N)A - (\beta_1, \dots, \beta_\N),
\end{align*}
and $ c = x_i, \ x_i - x_j $, there exist differential equations
\begin{equation}\label{eq:DE with coefficients}
\begin{aligned}
& (c(z_1, \dots, z_\N) )^{m_\el} \left(\frac{\partial}{\partial \zeta_\el} \right)^{m_\el} \psi \\
 &\qquad + \sum_{k = 1}^{m_\el}  d^{(c)}_{k,\el}(z_1, \dots, z_\N) (c(z_1, \dots, z_\N) )^{m_\el-k} \left(\frac{\partial}{\partial \zeta_\el} \right)^{m_\el-k} \psi  = 0, \qquad \el = 1, \dots, \N ,
\end{aligned} 
\end{equation}
for some $ d^{(c)}_{k,\el}(x_1, \dots, x_\N) \in R^{(c=0)}_{(\deg 0)} $, satisfied by 
\begin{align*}
\langle w_0, \mathcal{Y}_1(w_1,z_1) \cdots \mathcal{Y}_\N(w_\N,z_\N) w_{\N+1} \rangle
\end{align*}
in the region $ |z_1| > \dots > | z_\N | > 0 $.
\end{theorem}

\begin{proof}
Let $ w_i \in W_i $ be weight-homogeneous.  By Proposition~\ref{prop:T^(c=0) = J + F_M(T)}, there exist $ \mathcal{W}_1^{(k)} \in F_{\sigma + k}^{(c=0)}(J) $ and $ \mathcal{W}_2^{(k)} \in F_M(T) $, for each $ k \in \mathbb{Z}_{\geq 0} $, such that
\begin{align*}
\widehat{L}_{\zeta_\el}^k (w_0 \otimes \dots \otimes w_{\N+1}) = \mathcal{W}_1^{(k)}  + \mathcal{W}_2^{(k)} .
\end{align*} 
By Lemma~\ref{lem:sigma+S}, there exists $ S \in \mathbb{R} $ such that $ \sigma + S \in \mathbb{Z}_{>0} $ and
\begin{align*}
c^{\sigma + k + S} \mathcal{W}_2^{(k)} \in T^{(c=0)}.
\end{align*}
So,
\begin{align*}
c^{\sigma + k + S} \mathcal{W}_2^{(k)} \in \coprod_{r \leq M} T_{(r)}^{(c=0)}.
\end{align*}
Now we consider the finitely generated $ R^{(c=0)} $-module $ \coprod_{r \leq M} T_{(r)}^{(c=0)} $. Consider the ascending chain $ (M_j)_{j=0}^\infty $ of submodules $ M_j $ generated by
\begin{align*}
\{c^{\sigma + k + S} \mathcal{W}_2^{(k)} : k \in \mathbb{Z}, 0 \leq k \leq j \}.
\end{align*} 
Since $ R^{(c=0)} $ is Noetherian, there exist $ m_\el \in \mathbb{Z}_{> 0} $ and $ d^{(c)}_{1,\el}(x_1, \dots, x_\N), \dots, d^{(c)}_{m_\el,\el}(x_1,\dots,x_\N) \in R^{(c=0)} $ such that
\begin{align*}
c^{\sigma + m_\el + S} \mathcal{W}_2^{(m_\el)} + \sum_{k=0}^{m_\el-1} d^{(c)}_{m_\el-k,\el}(x_1, \dots, x_\N) c^{\sigma + k + S} \mathcal{W}_2^{(k)} = 0.
\end{align*}
Hence,
\begin{align*}
& \phi\left( c^{m_\el}\widehat{L}_{\zeta_\el}^{m_\el} (w_0 \otimes \dots \otimes w_{\N+1}) \right) \\
&\quad + \sum_{k =1}^{m_\el} \phi \left( d^{(c)}_{k,\el}(x_1, \dots, x_\N) c^{m_\el -k} \widehat{L}_{\zeta_\el}^{m_\el - k} (w_0 \otimes \dots \otimes w_{\N+1}) \right) = 0.
\end{align*}
When $ T $ is given the grading
\begin{align*}
\degree \left( f(x_1, \dots, x_\N) \overline{w}_0 \otimes \cdots \otimes \overline{w}_{\N+1} \right) = - \deg f(x_1, \dots, x_\N) - \wt \overline{w}_0 + \wt \overline{w}_1 + \cdots + \wt \overline{w}_{\N+1},
\end{align*}
for all weight-homogeneous $ \overline{w}_i \in W_i $, and the generators for $ J $ are checked to be homogeneous with respect to this new grading, we can see that $ d^{(c)}_{k,\el}(x_1, \dots, x_\N) $ can be chosen to be degree-zero elements.
\end{proof}

We can immediately  extend the previous theorem to the case where $ c $ is any non-zero element in $ \Span_{\mathbb{C}}\{x_1, \dots, x_\N\} $. If $ c $ is a scalar multiple of $ x_i $ or $ x_i - x_j $, the result follows from Theorem~\ref{thm: DEs c} by rescaling $ c $. Otherwise, no negative powers of $ c $ occur in $ R $, so we can take $ R^{(c =0)} = R $, and the result follows directly from Theorem~\ref{thm:formal DEs general}.

The following special case of Theorem~\ref{thm: DEs c} will be used to prove convergence.
\begin{corollary}
The series 
\begin{align*}
\langle w_0, \mathcal{Y}_1(w_1,z_1) \cdots \mathcal{Y}_\N(w_\N,z_\N) w_{\N+1} \rangle
\end{align*}
satisfies the following differential equations
\begin{equation}\label{eq:DE for convergence}
\begin{aligned}
\sum_{k = 0}^{m_\el} a_{k,\el}(z_1, \dots, z_\N) \left(z_\el \frac{\partial}{\partial z_\el} \right)^{k} \psi = 0, \qquad \el = 1, \dots, \N ,
\end{aligned} 
\end{equation}
with $ a_{k,\el}(x_1, \dots, x_\N) \in R^{(x_\el = 0)} $ and $ a_{m_\el, \el} = 1 $.
\end{corollary}

We now present the convergence property for products of $ \N $ \gtype \ twisted intertwining operators, and three lemmas for its proof.

\begin{lemma} \label{lem:form of series}
Let $ W_0, \dots, W_{\N+1} $, $ \widetilde{W}_0, \dots, \widetilde{W}_\N $, $ \mathcal{Y}_1, \dots, \mathcal{Y}_\N $, $ w_0, \dots, w_{\N+1} $ be as in Section~\ref{sec:DEs}. Assume that $ g $ is an automorphism of $ V $ with finite order. Then the multi-series 
\begin{equation*}
\langle w_0, \mathcal{Y}_1(w_1,x_1) \cdots \mathcal{Y}_\N(w_\N,x_\N) w_{\N+1} \rangle
\end{equation*}
is of the form
\begin{equation}\label{eq:form of series}
\sum_{s_1 \in S_1} \sum_{k_1 = 0}^{K_1} x_1^{s_1} (\log x_1)^{k_1} \cdots \sum_{s_\N \in S_\N} \sum_{k_\N = 0}^{K_\N} x_\N^{s_\N} (\log x_\N)^{k_\N}  \sum_{n_1, \dots, n_\N \in \frac{1}{t}\mathbb{Z} } a_{n_1, \dots, n_\N}^{(s_1,\dots, s_\N; k_1, \dots, k_\N)} x_1^{n_1} \cdots x_\N^{n_\N}
\end{equation}
for some $ a_{n_1, \dots, n_\N}^{(s_1,\dots, s_\N; k_1, \dots, k_\N)} \in \mathbb{C} $, finite sets $ S_1, \dots, S_\N \subseteq \mathbb{C} $ with distinct $ s,s' \in S_i $ satisfying $ s-s' \notin \frac{1}{t}\mathbb{Z} $, and $ K_1, \dots, K_\N \in \mathbb{Z}_{\geq 0} $.
\end{lemma}

\begin{proof}
Let $ W $ be a $ C_1 $-cofinite discretely graded $ g $-twisted $ V $-module. Then $ W = U + C_1(W) $ for some finite-dimensional space $ U $. We can enlarge $ U $ so that it has a basis of weight-homogeneous vectors, say with finitely many weights $ r_1, \dots, r_p \in \mathbb{C}$. Let $ w \in W $ be a homogeneous vector. Then $ w $ is a linear combination of vectors in $ U $ and vectors of the form $ u_{\alpha - 1} \overline{w} $ for some weight homogeneous $ u \in V_+^{[\alpha]} $ with $ \alpha = 0, \frac{1}{t}, \dots, \frac{t-1}{t} $, and some weight homogeneous $ \overline{w} \in W $. We have
\begin{equation*}
\Re (\wt \overline{w}) = \Re (\wt w )- \wt u + \alpha \leq \Re(\wt w) - \frac{1}{t}.
\end{equation*} 
Since $ W $ is discretely graded and hence lower bounded, this repeated argument will eventually terminate. This shows that the weights of $ W $ are contained in $ \{ r_1, \dots, r_p \} + \frac{1}{t} \mathbb{Z}_{\geq 0} $. 
Similar reasoning shows that $ C_1 $-cofinite discretely graded untwisted $ V $-modules have weights contained in $ \{ s_1, \dots, s_q \} + \mathbb{Z}_{\geq 0} $ for some finitely many weights $ s_1, \dots, s_q \in \mathbb{C}$.

Let $ i = 1, \dots, \N $. Let $ w_i \in W_i $, $ \widetilde{w}_i \in \widetilde{W}_i $ be weight-homogeneous. Then the summands of
\begin{align*}
\mathcal{Y}_i(w_i, x) \widetilde{w}_i = \sum_{k = 0}^K \sum_{n \in \mathbb{C}} (\mathcal{Y}_i)_{n,k}(w_i) \widetilde{w}_i x^{-n-1} (\log x)^k,
\end{align*}
must have $ \wt ( (\mathcal{Y}_i)_{n,k}(w_i) \widetilde{w}_i ) = \wt w_i - n - 1 + \wt \widetilde{w}_i $ being a weight of $ \widetilde{W}_{i-1} $ if $ (\mathcal{Y}_i)_{n,k}(w_i) \widetilde{w}_i $ is to be non-zero. Hence, $ n $ can only take values in $ \{ h_1, \dots, h_l \} + \frac{1}{t}\mathbb{Z} $, for some finitely many numbers $ h_1, \dots, h_l \in \mathbb{C}$. We can further assume that no two elements of $ h_1, \dots, h_l $ are congruent modulo $ \frac{1}{t} \mathbb{Z} $. This set of numbers is dependent on $ W_i , \widetilde{W}_i , \widetilde{W}_{i-1} $, and in fact independent of $ w_i $ and $  \widetilde{w}_i $. The upper bound $ K $ is dependent on only $ w_i $ and independent of $  \widetilde{w}_i $.

Since all modules are assumed to be $ C_1 $-cofinite and there are only finitely many $ \mathcal{Y}_1, \dots, \mathcal{Y}_\N $, we obtain the form \eqref{eq:form of series}. 
\end{proof}

\begin{lemma} \label{lem:abs convergence no logs}
Let $ a_{n_1, \dots, n_\N} \in \mathbb{C} $ be a sequence indexed by $ n_1 , \dots, n_\N \in \mathbb{Z} $. Assume the multi-series 
\begin{equation}\label{eq:multi-series 1}
\sum_{n_1, \dots, n_{\N-1} \in \mathbb{Z} } a_{n_1, \dots, n_\N} z_1^{n_1} \cdots z_{\N-1}^{n_{\N-1}} 
\end{equation}
in $ z_1, \dots, z_{\N-1} $ converges absolutely for each $ n_\N \in \mathbb{Z} $ in the region $ |z_1| > \cdots > | z_{\N-1} | > 0 $. Denote the sum of \eqref{eq:multi-series 1} by $ g_{n_\N}(z_1, \dots, z_{\N-1}) $ for each $ n_\N \in \mathbb{Z} $. Assume the series 
\begin{equation}\label{eq:multi-series 2}
\sum_{n_\N \in \mathbb{Z}} g_{n_\N}(z_1, \dots, z_{\N-1}) z_\N^{n_\N}  
\end{equation}
converges absolutely as a series in $ z_\N $ in the region $ |z_1| > \cdots > |z_{\N}| > 0 $ to a holomorphic multivariable function.
Then the multi-series
\begin{equation}\label{eq:multi-series 3}
\sum_{n_1, \dots, n_{\N} \in \mathbb{Z} } a_{n_1, \dots, n_\N} z_1^{n_1} \cdots z_{\N}^{n_{\N}} 
\end{equation}
in $ z_1, \dots, z_\N $ converges absolutely in the region $ |z_1| > \cdots > |z_{\N}| > 0 $. 
\end{lemma}

\begin{proof}
Denote the sum \eqref{eq:multi-series 2} by $ f(z_1, \dots, z_\N) $. Let $ r_1, \dots, r_\N \in \mathbb{R}_{>0} $ with $ r_1 > \cdots > r_\N > 0 $. Then for all $ n_1, \dots, n_\N \in \mathbb{Z} $, we have
\begin{align*}
&\frac{1}{(2 \pi \iu)^{\N}} \oint_{|z_1| = r_1} \cdots \oint_{|z_{\N} | = r_{\N}} f(z_1, \dots, z_\N) z_1^{-n_1 - 1} \cdots z_\N^{-n_\N-1} dz_\N \cdots dz_1 \\
&= \frac{1}{(2 \pi \iu)^{\N-1}} \oint_{|z_1| = r_1} \cdots \oint_{|z_{\N-1} | = r_{\N-1}} g_{n_\N}(z_1, \dots, z_{\N-1}) z_1^{-n_1 - 1} \cdots z_{\N-1}^{-n_{\N-1} -1}  dz_{\N-1} \cdots dz_1 \\ 
&= a_{n_1, \dots, n_\N} .
\end{align*}
Since $ f(z_1, \dots, z_\N) $ is holomorphic, it is bounded on compact sets such as $ \{ |z_1 | = r_1 \} \times \cdots \times  \{ |z_\N | = r_\N \} $. Say $ M $ bounds $ |f(z_1, \dots, z_\N)| $ on this set. Then
\begin{align*}
|a_{n_1, \dots, n_\N}|  
& \leq \frac{1}{(2 \pi)^{\N}} \int_{0}^{2\pi} \cdots \int_{0}^{2\pi} |f(z_1, \dots, z_\N)| r_1^{-n_1} \cdots r_\N^{-n_\N} dt_\N \cdots dt_1 \\
& \leq M  r_1^{-n_1} \cdots r_\N^{-n_\N}.
\end{align*}
For any fixed $ (z_1, \dots, z_\N) \in \mathbb{C}^\N $ with $ |z_1| > \cdots > |z_{\N}| > 0 $ we can choose $ r_1, \dots, r_\N \in \mathbb{R}_{>0} $ such that $ r_1 > |z_1| > \cdots > r_\N > |z_\N| > 0 $. So
\begin{equation*}
\sum_{n_1, \dots, n_\N \in \mathbb{Z}_{\geq 0}} |a_{n_1, \dots, n_\N} z_1^{n_1} \cdots z_\N^{n_\N} | \leq \sum_{n_1, \dots, n_\N \in \mathbb{Z}_{\geq 0}} M \left( \frac{|z_1|}{r_1} \right)^{n_1} \cdots \left( \frac{|z_\N|}{r_\N} \right)^{n_\N} < \infty. 
\end{equation*}
We can do similar analyses for all sums of $ |a_{n_1, \dots, n_\N} z_1^{n_1} \cdots z_\N^{n_\N}|$ over all index sets made from combinations of $ n_i \in \mathbb{Z}_{\geq 0} $ or $ n_i \in \mathbb{Z}_{< 0} $. If the index set has $ n_i \in \mathbb{Z}_{< 0} $, we instead choose $ |z_i| > r_i $. Each of these finitely many sums converges absolutely. Hence, \eqref{eq:multi-series 3} is absolutely convergent as a multi-series.
\end{proof}

\begin{lemma} \label{lem:abs convergence with logs}
Let $ S_1, \dots, S_\N $ be finite subsets of $ \mathbb{C} $ with distinct $ s,s' \in S_i $ satisfying $ s-s' \notin \mathbb{Z} $. Let $ K_1, \dots, K_\N \in \mathbb{Z}_{\geq 0} $. Let $ a_{n_1, \dots, n_\N}^{(s_1,\dots, s_\N; k_1, \dots, k_\N)} \in \mathbb{C} $  be a sequence indexed by $ n_1 , \dots, n_\N \in \mathbb{Z} $, $ s_i \in S_i $ and $ k_i = 0, \dots, K_i $. Assume the multi-series 
 \begin{equation}\label{eq:multi-series 4}
\sum_{s_1 \in S_1} \sum_{k_1 = 0}^{K_1} z_1^{s_1} (\log z_1)^{k_1} \cdots \sum_{s_{\N-1} \in S_{\N-1}} \sum_{k_{\N-1} = 0}^{K_{\N-1}} z_{\N-1}^{s_{\N-1}} (\log z_{\N-1})^{k_{\N-1}}  \sum_{n_1, \dots, n_{\N-1} \in \mathbb{Z}} a_{n_1, \dots, n_\N}^{(s_1,\dots, s_\N; k_1, \dots, k_\N)} z_1^{n_1} \cdots z_{\N-1}^{n_{\N-1}}
\end{equation}
in $ z_1, \dots, z_{\N-1} $ converges absolutely for each $ n_\N \in \mathbb{Z} $, $ s_\N \in S_\N $ and $ k_\N = 0, \dots, K_\N $ in the region $ |z_1| > \cdots > | z_{\N-1} | > 0 $. Denote the sum of \eqref{eq:multi-series 4} by $ g_{n_\N}^{(s_\N; k_\N)}(z_1, \dots, z_{\N-1}) $ for each $ n_\N \in \mathbb{Z} $. Assume the series 
\begin{equation}\label{eq:multi-series 5}
\sum_{s_\N \in S_\N} \sum_{k_\N = 0}^{K_\N} z_\N^{s_\N} (\log z_\N)^{k_\N} \sum_{n_\N \in \mathbb{Z}} g_{n_\N}^{(s_\N; k_\N)}(z_1, \dots, z_{\N-1})  z_\N^{n_\N}  
\end{equation}
converges absolutely as a series in $ z_\N $ in the region $ |z_1| > \cdots > |z_{\N}| > 0 $ to a multivalued holomorphic multivariable function.
Then the multi-series
\begin{equation}\label{eq:multi-series 6}
\sum_{s_1 \in S_1} \sum_{k_1 = 0}^{K_1} z_1^{s_1} (\log z_1)^{k_1} \cdots \sum_{s_\N \in S_\N} \sum_{k_\N = 0}^{K_\N} z_\N^{s_\N} (\log z_\N)^{k_\N}  \sum_{n_1, \dots, n_\N \in \mathbb{Z} } a_{n_1, \dots, n_\N}^{(s_1,\dots, s_\N; k_1, \dots, k_\N)} z_1^{n_1} \cdots z_\N^{n_\N}
\end{equation}
in $ z_1, \dots, z_\N $ converges absolutely in the region $ |z_1| > \cdots > |z_{\N}| > 0 $. 
\end{lemma}

\begin{proof}
We seek to reduce this statement to a case where we can apply Lemma \ref{lem:abs convergence no logs}. Let $ \rho_i = e^{2 \pi \iu z_i \frac{\partial}{\partial z_i}} $ be the operator on multivariable functions in  $ |z_1| > \cdots > |z_{\N}| > 0 $ that are holomorphic in at least $ z_i $ that rotates the variable $ z_i $ to the next branch. That is, $ \rho_i f(z_1, \dots, z_\N) = f(z_1, \dots, e^{2\pi \iu} z_i, \dots, z_\N) $, for every such function $ f(z_1, \dots, z_\N) $. 

Let $ h(z_1, \dots, z_\N ) $ be a multivalued function in $ |z_1| > \cdots > |z_{\N}| > 0 $ that is single-valued and holomorphic in $ z_i $. Let $ \alpha , \beta \in \mathbb{C} $ and $ k \in \mathbb{Z}_{\geq 0} $. Then we observe that
\begin{align*}
&(\rho_i - e^{2 \pi \iu \alpha}) z_i^\beta (\log z_i)^k  h(z_1, \dots, z_\N ) \\
&\quad = e^{2 \pi \iu \beta} z_i^\beta (\log z_i + 2 \pi \iu)^k  h(z_1, \dots, z_\N ) - e^{2 \pi \iu \alpha} z_i^\beta (\log z_i)^k  h(z_1, \dots, z_\N ) .
\end{align*}
So, $ (\rho_i - e^{2 \pi \iu \alpha}) $ reduces the highest power of $ (\log z_i) $ in $ z_i^\beta (\log z_i)^k  h(z_1, \dots, z_\N ) $ if $ \alpha - \beta \in \mathbb{Z} $, and retains the highest power otherwise. No application of $ (\rho_i - e^{2 \pi \iu \alpha}) $ annihilates $ z_i^\beta (\log z_i)^k  h(z_1, \dots, z_\N ) $ unless $ k = 0 $ and $ \alpha - \beta \in \mathbb{Z} $. Hence, we can always construct a polynomial 
\begin{equation*}
p^{(s_i,K_i)}(\rho_i) = (\rho_i - e^{2 \pi \iu s_{i}})^{K_i} \prod_{\substack{s' \in S_i \\ s' \neq s_i}} (\rho_i - e^{2 \pi \iu s'})^{K_i + 1} 
\end{equation*}
to extract $ z_i^{s_{i}} h^{(s_i;K_i)}(z_1, \dots, z_\N) $ (up to some non-zero factor) from a multivalued holomorphic function of the form $ f(z_1, \dots, z_\N) = \sum_{r_i \in S_i} \sum_{k_i = 0}^{K_i} z_i^{r_{i}} (\log z_i)^{k_i} h^{(r_i;k_i)}(z_1, \dots, z_\N) $, where $ h^{(r_i;k_i)}(z_1, \dots, z_\N) $ are multivalued functions that are single-valued and holomorphic in $ z_i $. We can then subtract $ z_i^{s_{i}} (\log z_i)^{K_i} h^{(s_i;K_i)}(z_1, \dots, z_\N) $ from $ f(z_1, \dots, z_\N) $, lowering the highest power of $ \log z_i $, and repeat the process to extract $ z_i^{s_{i}} h^{(s_i;k_i)}(z_1, \dots, z_\N) $ for all $ k_i = 0, \dots, K_i $. Observe that we can express each $ h^{(s_i;k_i)}(z_1, \dots, z_\N) $ as a linear combination of $ z_i^s (\log z_i)^k \rho_i^p f(z_1, \dots, z_\N) $, for $ s \in \mathbb{C} $ and $ k,p \in \mathbb{Z}_{\geq 0} $. Also observe that this implies that if $ f(z_1,\dots, z_\N) $ is holomorphic in all variables, then so is each $ h^{(s_i;k_i)}(z_1, \dots, z_\N) $.

Since \eqref{eq:multi-series 5} converges absolutely as a series in $ z_\N $ in the region $ |z_1| > \cdots > |z_{\N}| > 0 $, we know that each series $ \sum_{n_\N \in \mathbb{Z}} g_{n_\N}^{(s_\N;k_\N)}(z_1, \dots, z_{\N-1})  z_\N^{n_\N} $ converges absolutely as a series in $ z_\N $ in the region $ |z_1| > \cdots > |z_{\N}| > 0 $. Denote the sum by $ h^{(s_\N;k_\N)}(z_1, \dots, z_\N) $. By the reasoning in the previous paragraph, we know that $ h^{(s_\N;k_\N)}(z_1, \dots, z_\N) $ is a sum of holomorphic multivariable functions, so is itself holomorphic in all variables.

Similarly, we know that $ \sum_{n_1, \dots, n_{\N-1} \in \mathbb{Z} } a_{n_1, \dots, n_\N}^{(s_1,\dots, s_\N;k_1,\dots,k_\N)} z_1^{n_1} \cdots z_{\N-1}^{n_{\N-1}} $ converges absolutely in the region $ |z_1| > \cdots > | z_{\N-1} | > 0 $. Denote this single-valued holomorphic multivariable function by $ h^{(s_1,\dots, s_\N;k_1,\dots,k_\N)}_{n_\N} $. We can apply some operator $ P^{(s_1,\dots, s_{\N-1};k_1,\dots,k_{\N-1})} $ to $ g_{n_\N}^{(s_\N;k_\N)}(z_1, \dots, z_{\N-1}) $ to extract $ h^{(s_1,\dots, s_\N;k_1,\dots,k_\N)}_{n_\N} $. Since the operator $ P^{(s_1,\dots, s_{\N-1};k_1,\dots,k_{\N-1})} $ is independent of $ n_\N \in \mathbb{Z} $, we obtain
\begin{align*}
 \sum_{n_\N \in \mathbb{Z}} h^{(s_1,\dots, s_\N;k_1,\dots,k_\N)}_{n_\N} z_\N^{n_\N} = \sum_{n_\N \in \mathbb{Z}} P^{(s_1,\dots, s_{\N-1};k_1,\dots,k_{\N-1})} g_{n_\N}^{(s_\N;k_\N)}(z_1, \dots, z_{\N-1}) z_\N^{n_\N},
\end{align*}
which is absolutely convergent in $ z_\N $ by the triangle inequality, and furthermore holomorphic in all variables. By Lemma \ref{lem:abs convergence no logs}, we know that $ \sum_{n_1, \dots, n_{\N} \in \mathbb{Z} } a_{n_1, \dots, n_\N}^{(s_1,\dots, s_\N;k_1,\dots,k_\N)} z_1^{n_1} \cdots z_{\N}^{n_{\N}} $ converges absolutely in the region $ |z_1| > \cdots > | z_{\N} | > 0 $. Thus, we obtain the absolute convergence of the desired series \eqref{eq:multi-series 6}.
\end{proof}

\begin{theorem}
[The convergence property for products of $ \N $ \gtype \ twisted intertwining operators] Let $ W_0, \dots, W_{\N+1} $, $ \widetilde{W}_0, \dots, \widetilde{W}_\N $, $ \mathcal{Y}_1, \dots, \mathcal{Y}_\N $, $ w_0, \dots, w_{\N+1} $ be as in Section~\ref{sec:DEs}. Assume that $ g $ is an automorphism of $ V $ with finite order. 
\begin{itemize}
\item[(1)] The multi-series 
\begin{equation*}
\langle w_0, \mathcal{Y}_1(w_1,z_1) \cdots \mathcal{Y}_\N(w_\N,z_\N) w_{\N+1} \rangle
\end{equation*}
in $ z_1, \dots, z_\N $ is absolutely convergent in the region $ |z_1| > \cdots > |z_\N| > 0 $.
\item[(2)] The sum of
\begin{equation*}
\langle w_0, \mathcal{Y}_1(w_1,z_1) \cdots \mathcal{Y}_\N(w_\N,z_\N) w_{\N+1} \rangle
\end{equation*}
in $ |z_1|> \dots >|z_\N|>0 $ can be analytically continued to a multivalued analytic function 
\begin{equation}
F(\langle w_0, \mathcal{Y}_1(w_1,z_1) \cdots \mathcal{Y}_\N(w_\N,z_\N) w_{\N+1} \rangle)
\end{equation}
defined on all of $ M^\N := \{ (z_1, \dots, z_\N) \in \widehat{\mathbb{C}}^\N : z_i \neq  0, \infty \text{ and } z_i \neq z_j \text{ when } i \neq j  \} $. 
\end{itemize}
 
\end{theorem}

\begin{proof}
When $ \N = 1 $, we have 
\begin{align*}
\langle w_0, \mathcal{Y}_1(w_1,z_1) w_2 \rangle  = \sum_{n \in \mathbb{C}} \sum_{k = 0}^K \langle w_0,  (\mathcal{Y}_1)_{n,k}(w_1) w_2 \rangle z_1^{-n-1} (\log z_1)^k.
\end{align*}  
This sum is finite by the $ L(0) $-grading, hence converges absolutely when $ z_1 \in M^1 $. We now proceed by induction. Assume that $ \langle w_0, \mathcal{Y}_1(w_1,z_1) \cdots \mathcal{Y}_{\N-1}(w_{\N-1},z_{\N-1}) \widetilde{w}_{\N-1} \rangle$ converges absolutely in the region $ |z_1| > \cdots > |z_{\N-1}| > 0 $ for all $ \widetilde{w}_{\N-1} \in \widetilde{W}_{\N-1} $. Further assume that the function it converges to is holomorphic in $ z_1, \dots, z_{\N-1} $ and can be analytically extended to a multivalued analytic function defined on $ M^{\N-1} $. Now consider the series 
\begin{align}
&\langle w_0, \mathcal{Y}_1(w_1,z_1) \cdots \mathcal{Y}_\N(w_\N,z_\N) w_{\N+1} \rangle \nonumber \\
&= \sum_{n \in \mathbb{C}} \sum_{k = 0}^K \langle w_0, \mathcal{Y}_1(w_1,z_1) \cdots \mathcal{Y}_{\N-1}(w_{\N-1},z_{\N-1}) (\mathcal{Y}_\N)_{n,k}(w_\N) w_{\N+1} \rangle z_\N^{-n-1} (\log z_\N)^k, \label{eq:correlator sum}
\end{align}  
as a series in $ z_\N $, with fixed $ z_1, \dots, z_{\N-1} $ satisfying $ |z_1| > \cdots > |z_{\N-1}| > 0 $. The \textit{ordinary} differential equation  \eqref{eq:DE for convergence} for $ \el = \N $ after a change of variable $ \zeta_\N = z_\N^{1/t} $ has a regular singularity at $ \zeta_\N = 0 $ (equivalently $ z_\N = 0 $) and is satisfied by \eqref{eq:correlator sum} for all $ |z_1| > \cdots > |z_{\N-1}| > |z_\N| > 0 $. Since the coefficients of \eqref{eq:DE for convergence} are analytic in $ z_1, \dots, z_\N $, and the coefficients in \eqref{eq:correlator sum} are analytic in $ z_1, \dots, z_{\N-1} $, the series \eqref{eq:correlator sum} converges absolutely as a series in $ z_\N $ by Theorem \ref{thm:reg sing ODE} to a function that is analytic in $ z_1, \dots, z_\N $. By the induction hypothesis, Lemma~\ref{lem:form of series}, and Lemma~\ref{lem:abs convergence with logs} after a change of variables to $ ( \zeta_1, \dots, \zeta_\N) = (z_1^{1/t} , \dots, z_\N^{1/t}) $, the multi-series \eqref{eq:correlator sum} is absolutely convergent in the region $ |z_1| > \cdots > |z_{\N-1}| > |z_\N| > 0 $. 
Since the coefficients in \eqref{eq:correlator sum} can be analytically extended to multivalued functions on $ M^{\N-1} $, the sum \eqref{eq:correlator sum} satisfies \eqref{eq:DE for convergence} for all $ (z_1, \dots, z_{\N-1}) \in M^{\N-1} $ when $ |z_1|, \dots ,|z_{\N-1}| > |z_\N| > 0 $. Since the coefficients of \eqref{eq:DE for convergence} are analytic on $ M^{\N} $, for fixed $ (z_1, \dots, z_{\N-1}) $ we can analytically extend the sum $ \eqref{eq:correlator sum} $ to all $ (z_1, \dots, z_\N) \in M^\N $.
\end{proof}

Finally, we will prove that certain component-isolated singularities of the four-point correlation function are regular singularities. These regular singularities include those that are needed for the convergence and extension property for products of two \gtype\ intertwining operators (see \cite{Reptheoryandorbifoldcft}).

\begin{lemma}\label{lem:L(0) property}
Let $ w_0 \in W_0 $, $ w_1 \in W_1 $, $ w_2 \in W_2 $ and $ w_3 \in W_3 $ be weight homogeneous, and define $ \Delta = \wt w_0 - \wt w_1 - \wt w_2 - \wt w_3 $. Let $ A \in \GL(2, \mathbb{C}) $. Let $ (\zeta_1, \zeta_2) = (z_1, z_2) A $ be a change of variables. 
Then 
\begin{equation}\label{eq:homogeneous derivative relation}
\left( \zeta_1 \frac{\partial}{\partial \zeta_1} + \zeta_2 \frac{\partial}{\partial \zeta_2} - \Delta \right)^K \langle w_0, \mathcal{Y}(w_1, z_1) \mathcal{Y}(w_2, z_2) w_{3} \rangle = 0, \quad \text{for sufficiently large } K >0.
\end{equation}
\end{lemma}

\begin{proof}
We use the property $ [L(0), \mathcal{Y}(w,x)] = \mathcal{Y}(L(0) w,x) + x  \mathcal{Y}(L(-1) w,x) $ to derive
\begin{align*}
\left( z_1 \frac{\partial}{\partial z_1} + z_2 \frac{\partial}{\partial z_2} - \Delta \right) \langle w_0, \mathcal{Y}(w_1, z_1) \mathcal{Y}(w_2, z_2) w_{3} \rangle &= \langle (L(0) - \wt w_0)w_0, \mathcal{Y}(w_1, z_1) \mathcal{Y}(w_2, z_2) w_{3} \rangle \\
& \quad - \langle w_0, \mathcal{Y}((L(0) - \wt w_1)w_1, z_1) \mathcal{Y}(w_2, z_2) w_{3} \rangle \\
& \quad - \langle w_0, \mathcal{Y}(w_1, z_1) \mathcal{Y}((L(0) - \wt w_2)w_2, z_2) w_{3} \rangle \\
& \quad - \langle w_0, \mathcal{Y}(w_1, z_1) \mathcal{Y}(w_2, z_2) (L(0) - \wt w_3)w_{3} \rangle .
\end{align*}
Repeated application gives 
\begin{align*}
\left( z_1 \frac{\partial}{\partial z_1} + z_2 \frac{\partial}{\partial z_2} - \Delta \right)^K \langle w_0, \mathcal{Y}(w_1, z_1) \mathcal{Y}(w_2, z_2) w_{3} \rangle &=0
\end{align*}
whenever $ K $ is sufficiently large. Then we use
\begin{align*}
\zeta_1 \frac{\partial}{\partial \zeta_1} + \zeta_2 \frac{\partial}{\partial \zeta_2} &= \zeta_1 \left( \frac{\partial z_1}{\partial \zeta_1}\frac{\partial}{\partial z_1} + \frac{\partial z_2}{\partial \zeta_1}\frac{\partial}{\partial z_2}  \right) + \zeta_2 \left( \frac{\partial z_1}{\partial \zeta_2}\frac{\partial}{\partial z_1} + \frac{\partial z_2}{\partial \zeta_2}\frac{\partial}{\partial z_2}  \right) \\
&= \left( \zeta_1 \frac{\partial z_1}{\partial \zeta_1} +  \zeta_2 \frac{\partial z_1}{\partial \zeta_2} \right) \frac{\partial}{\partial z_1} +  \left( \zeta_1 \frac{\partial z_2}{\partial \zeta_1} +  \zeta_2 \frac{\partial z_2}{\partial \zeta_2} \right)  \frac{\partial}{\partial z_2} \\
&=  z_1 \frac{\partial}{\partial z_1} + z_2 \frac{\partial}{\partial z_2}. \qedhere
\end{align*}
\end{proof}

\begin{proposition}
Assume that $ g $ has finite order. Then every component-isolated singularity of $ F( \langle w_0, \mathcal{Y}(w_1, z_1) \mathcal{Y}(w_2, z_2) w_{3} \rangle) $ with a change of coordinates of the form $ (\zeta_1, \zeta_2) = (z_1, z_2)A $ is a regular singularity.
\end{proposition}

\begin{proof}
Let $ \delta_1, \delta_2 \in \{0, \infty \} $ and let $ (\zeta_1, \zeta_2) = (z_1, z_2) A $ be a change of variables such that $ (\zeta_1, \zeta_2) = (\delta_1, \delta_2) $ is a component-isolated singularity of a multivalued holomorphic function defined on $ M^2 $. 

We also define $ (\eta_1, \eta_2) $ such that $ \eta_j^{s(\delta_j)t} = \zeta_j $, with $ s(\delta) = \begin{cases}
1 &\text{if } \delta = 0, \\
-1 &\text{if } \delta = \infty. \\
\end{cases} $

Recall that we have shown the existence of the following differential equations
\begin{equation}\label{eq:DE n=2}
\sum_{k = 0}^{m} a_{k}(z_1,z_2) c^{k} \left( \frac{\partial}{\partial \zeta_i} \right)^k \psi = 0, \quad i = 1,2, \quad \text{for all non-zero } c \in \Span_{\mathbb{C}}\{z_1, z_2 \},
\end{equation}
with $ a_{k} \in R^{(c = 0)}_{(\deg 0)} $ and $ a_{m} = 1 $, satisfied by $ F(\langle w_0, \mathcal{Y}_1(w_1,z_1) \mathcal{Y}_2(w_2,z_2) w_{3} \rangle) $. Observe that by application of $ c \frac{\partial}{\partial \zeta_i} $, we can choose $ m $ to be arbitrarily high. We choose it to be some sufficiently large $ K $ as in \eqref{eq:homogeneous derivative relation}.

We will use $ a_{k}^{(1)}, a_{k}^{(2)}, a_{k}^{(3)}, b_{p,q}^{(1)}, b_{p,q}^{(2)} $ for elements in $ R^{(c = 0)}_{(\deg 0)} $ with $ a_{m}^{(1)} = a_{m}^{(2)} = a_{m}^{(3)} = 1 $. 

For some fixed $ i = 1,2  $, choose $ c = \zeta_i $ to obtain
\begin{equation}
\sum_{k = 0}^{m} a_{k}^{(1)}(z_1,z_2) \zeta_i^{k} \left( \frac{\partial}{\partial \zeta_i} \right)^k \psi = 0.
\end{equation}

Then we change basis from $ \left\{ \zeta_i^{k} \left( \frac{\partial}{\partial \zeta_i} \right)^k : k \in  \mathbb{Z}_{\geq 0} \right\} $ to $ \left\{ \left( \zeta_i \frac{\partial}{\partial \zeta_i} \right)^k : k \in  \mathbb{Z}_{\geq 0} \right\} $, obtaining
\begin{equation} \label{eq:DE in xi_i}
\sum_{k = 0}^{m} a_{k}^{(2)}(z_1,z_2) \left( \zeta_i \frac{\partial}{\partial \zeta_i} \right)^k \psi = 0.
\end{equation}

We use Lemma~\ref{lem:L(0) property} on \eqref{eq:DE in xi_i} to obtain, for $ j = 1,2 $ with $ i \neq j $,
\begin{equation}
\left(\zeta_j \frac{\partial}{\partial \zeta_j} \right)^m \psi + \sum_{p,q = 0}^{m-1} b_{p,q}^{(1)}(z_1,z_2) \left(\zeta_i \frac{\partial}{\partial \zeta_i} \right)^p  \left(\zeta_j \frac{\partial}{\partial \zeta_j} \right)^q \psi = 0,
\end{equation}

Then $ \zeta_i \frac{\partial}{\partial \zeta_i} = \frac{1}{s(\delta_i) t}\eta_i\frac{\partial}{\partial \eta_i} $ and $ \zeta_j \frac{\partial}{\partial \zeta_j} = \frac{1}{s(\delta_j) t}\eta_j \frac{\partial}{\partial \eta_j} $ give
\begin{align}
\sum_{k = 0}^{m} a_{k}^{(3)}(z_1,z_2) \left( \eta_i \frac{\partial}{\partial \eta_i} \right)^k \psi &= 0,\\
\left(\eta_j \frac{\partial}{\partial \eta_j} \right)^m \psi + \sum_{p,q = 0}^{m-1} b_{p,q}^{(2)}(z_1,z_2) \left(\eta_i \frac{\partial}{\partial \eta_i} \right)^p  \left(\eta_j \frac{\partial}{\partial \eta_j} \right)^q \psi &= 0.
\end{align}
To show that $ (\eta_1, \eta_2) = (0,0) $ is a regular singularity using Theorem~\ref{thm:simple sing PDE}, we are left to check case by case that the coefficients are holomorphic in $ (\eta_1, \eta_2) \in \D_{r_1}(0) \times \D_{r_2}(0) $ for some sufficiently small $ r_1, r_2 > 0 $. 

\textbf{Case 1.} Assume $ \delta_1 = \delta_2 = 0 $. At least one of $ z_1, z_2 ,z_1 - z_2 $ would need to be equal to $ \alpha \zeta_1 + \beta \zeta_2 $ for some $ \alpha, \beta \in \mathbb{C}^\times $. So, $ z_1 = 0 $, $ z_2 = 0 $ or $ z_1 - z_2 = 0 $ would occur for some $ (\eta_1, \eta_2) \in \D_{r_1}^{\times}(0) \times \D_{r_2}^{\times}(0) $, for any $ r_1, r_2 > 0 $. So, there are no component-isolated singularities in this case.

\textbf{Case 2.} Assume $ \delta_1 = \infty $ or $ \delta_2 = \infty $. Only one can have this value for a similar reason as in the previous case. Say $ \delta_j = \infty $, and we must have $ \delta_i = 0 $ for $ i \neq j $.
Since $ a_{k}^{(3)} $ and $ b_{p,q}^{(2)} $ are in $ R^{(\zeta_i = 0)}_{(\deg 0)} $, they are linear combinations of monomials of the form $ z_1^{p/t} z_2^{q/t} (z_1 - z_2)^{-k} $, where $ p, q \in \mathbb{Z} $, $ k \in \mathbb{Z}_{\geq 0} $ and $ p + q = kt $. We only need to assess these monomials, which we rewrite as
\begin{equation*}
(\alpha_1 \eta_i^t + \beta_1 \eta_j^{-t})^{p/t} (\alpha_2 \eta_i^t + \beta_2 \eta_j^{-t})^{q/t} (\alpha_3 \eta_i^t + \beta_3 \eta_j^{-t})^{-k} ,
\end{equation*}
for some $ \alpha_1,\alpha_2,\alpha_3, \beta_1, \beta_2, \beta_3 \in \mathbb{C} $. Since $ p + q - kt = 0 $, we can rewrite this as
\begin{equation*}
(\alpha_1 \eta_i^t \eta_j^t + \beta_1)^{p/t} (\alpha_2 \eta_i^t \eta_j^t + \beta_2)^{q/t} (\alpha_3 \eta_i^t \eta_j^t + \beta_3)^{-k}.
\end{equation*}
If none of $ \beta_1, \beta_2, \beta_3 $ are equal to zero, then this expression is defined for $ \eta_i = 0 $ or $ \eta_j = 0 $. Otherwise, if some of $ \beta_1, \beta_2, \beta_3 $ are equal to zero, then the respective variables $ z_1, z_2 ,z_1 - z_2 $ are proportional to $ \zeta_i $. The sum of the powers of the variables $ z_1, z_2 ,z_1 - z_2 $ that are proportional to $ \zeta_i $ must be non-negative to ensure the monomial lies in $ R^{(\zeta_i = 0)} $. So, this expression is still defined for $ \eta_i = 0 $ or $ \eta_j = 0 $.

Since we have shown that the coefficients  $ a_{k}^{(3)} $ and $ b_{p,q}^{(2)} $ are holomorphic in $ (\eta_1, \eta_2) \in \D_{r_1}(0) \times \D_{r_2}(0) $, we have shown that the singularity $ (\eta_1, \eta_2) = (0,0) $ is indeed regular. Hence, we know that the solutions to this system of differential equations have an expansion of the form
\begin{equation}
\sum_{i = 1}^K \eta_1^{r_i} \eta_2^{s_i} (\log \eta_1)^{j_i} (\log \eta_2)^{k_i} f_i(\eta_1, \eta_2),
\end{equation}
for $ r_i, s_i \in \mathbb{C} $, $ j_i, k_i \in \mathbb{Z}_{\geq 0} $ and $ f_i $ holomorphic in some product of sufficiently small disks centered at 0. Thus, after changing back to $ (\zeta_1, \zeta_2) $, we see that there is a regular singularity at $ (\zeta_1, \zeta_2) = (\delta_1, \delta_2) $.
\end{proof}

\begin{remark}
Changes of coordinates of the form $ (\zeta_1, \zeta_2) = (z_1, z_2)A $ are not as general as the affine change of coordinates. However, they include important component-isolated singularities such as $ (\zeta_1, \zeta_2) = (z_1 - z_2, z_2) = (0,\infty) $. This regular singularity is needed for the convergence and extension property, and is used to prove the associativity of twisted intertwining operators.
\end{remark}

\begin{remark}
The previous results do not show that every component-isolated singularity of $ F(\langle w_0, \mathcal{Y}_1(w_1,z_1) \cdots \mathcal{Y}_\N(w_\N,z_\N) w_{\N+1} \rangle) $ is a regular singularity when $ \N > 2 $. However, the $ \N = 2 $ case above can be used to show associativity and commutativity of twisted intertwining operators (see \cite{Reptheoryandorbifoldcft}). Certain (and important) regular singularities in the $ \N > 2 $ case can be shown using the convergence property, associativity and commutativity, and the $ \N = 2 $ case.

For example, consider $ (\zeta_1, \zeta_2, \zeta_3) = (z_1 - z_3, z_2, z_3) $ and $ (\delta_1, \delta_2, \delta_3) = (0,0,\infty) $. Given intertwining operators $ \mathcal{Y}_1 ,\mathcal{Y}_2, \mathcal{Y}_3 $, there exist intertwining operators $ \mathcal{Y}_4, \mathcal{Y}_5 $ such that 
\begin{align*}
&F( \langle w_0, \mathcal{Y}_1(w_1, z_1) \mathcal{Y}_2(w_2, z_2) \mathcal{Y}_3(w_3, z_3) w_4  \rangle ) \\
&\qquad = \sum_{m \in \mathbb{C}} \sum_{k = 0}^{K} F(\langle w_0, (\mathcal{Y}_1)_{m,k}(w_1) \mathcal{Y}_2(w_2, z_2) \mathcal{Y}_3(w_3, z_3) w_4  \rangle )z_1^{-m-1} ( \log z_1)^k \\
&\qquad = \sum_{m \in \mathbb{C}} \sum_{k = 0}^{K} F(\langle w_0, (\mathcal{Y}_1)_{m,k}(w_1) \mathcal{Y}_4(w_3, z_3) \mathcal{Y}_5(w_2, z_2)  w_4  \rangle) z_1^{-m-1} ( \log z_1)^k \\
&\qquad = F( \langle w_0, \mathcal{Y}_1(w_1, z_1) \mathcal{Y}_4(w_3, z_3) \mathcal{Y}_5(w_2, z_2) w_4  \rangle ) \\
&\qquad = \sum_{m \in \mathbb{C}} \sum_{k = 0}^{K} F( \langle w_0, \mathcal{Y}_1(w_1, z_1) \mathcal{Y}_4(w_3, z_3) (\mathcal{Y}_5)_{m,k}(w_2) w_4  \rangle )z_2^{-m-1} ( \log z_2)^k  .
\end{align*}
The last expression gives the desired singularity.
\end{remark}

\printbibliography

\noindent {\small \sc Department of Mathematics, Rutgers University,
110 Frelinghuysen Rd., Piscataway, NJ 08854-8019}

\noindent {\em E-mail address}: \url{dwt24@scarletmail.rutgers.edu}

\end{document}